\pdfoutput = 1
\documentclass[leqno]{jnsao}

\usepackage[utf8]{inputenc}

\usepackage[nameinlink, capitalise]{cleveref}

\usepackage{tkz-euclide}
\usepackage{multicol}

\usepackage{enumitem}
\newlist{assumption}{enumerate}{1}
\setlist[assumption]{label=(\textsc{a}\arabic*)}
\crefname{assumptioni}{Assumption}{Assumptions}

\newlist{hypothesis}{enumerate}{1}
\setlist[hypothesis]{label=(\textsc{h}\arabic*)}
\crefname{hypothesis}{Hypothesis}{Hypotheses}

\newcommand{\R}{\mathbb{R}}
\newcommand{\N}{\mathbb{N}}
\newcommand{\1}{\mathbb{1}}

\DeclareMathOperator{\sign}{\mathrm{sign}}

\DeclareMathOperator{\dive}{\mathrm{div}}

\DeclareMathOperator{\meas}{\mathrm{meas}}

\newcommand{\dx}{\,\mathrm{d}x}

\newcommand{\dH}{\,\mathrm{d}\mathcal{H}}
\newcommand{\Ha}{\mathcal{H}}
\newcommand{\Cu}{\mathcal{C}}

\newcommand{\norm}[1]{\|#1\|}
\def\textcolor{blue}

\usepackage{yhmath}

\renewcommand{\epsilon}{\varepsilon}

\providecommand{\orcid}[1]{\mbox{\scshape\sffamily orcid:}\,\href{https://orcid.org/#1}{\detokenize{#1}}}

\title{Numerical analysis of a nonsmooth quasilinear elliptic control problem: I. Explicit second-order optimality conditions}
\author{Christian Clason\thanks{Faculty of Mathematics, Universit\"at Duisburg-Essen, 45117 Essen, Germany; \emph{current address:} Institute of Mathematics and Scientific Computing, University of Graz, Heinrichstrasse 36, 8010 Graz, Austria (\email{c.clason@uni-graz.at}, \orcid{0000-0002-9948-8426})}
    \and Vu Huu Nhu\thanks{Faculty of Fundamental Sciences, PHENIKAA University, Yen Nghia, Ha Dong, Hanoi 12116, Vietnam (\email{nhu.vuhuu@phenikaa-uni.edu.vn}, \orcid{0000-0003-4279-3937})}
    \and Arnd Rösch\thanks{Faculty of Mathematics, University of Duisburg-Essen, Thea-Leymann-Strasse 9, 45127 Essen, Germany
        (\email{arnd.roesch@uni-due.de})
    }
}
\shortauthor{Clason, Nhu, Rösch}

\manuscripteprinttype{arxiv}
\manuscripteprint{2203.16865v3}
\manuscriptlicense{}
\manuscriptcopyright{}
\date{2023-09-12}

\begin{document}

\maketitle

\begin{abstract}
    In this paper, we derive explicit second-order necessary and sufficient optimality conditions of a local minimizer to an optimal control problem for a quasilinear second-order partial differential equation with a piecewise smooth but not differentiable nonlinearity in the leading term.
    The key argument rests on the analysis of level sets of the state. Specifically,
    we show that if a function vanishes on the boundary and its the gradient is different from zero on a level set, then this set decomposes into finitely many closed simple curves.
    Moreover, the level sets depend continuously on the functions defining these sets.
    We also prove the continuity of the integrals on the level sets. In particular, Green's first identity is shown to be applicable on an open set determined by two functions with nonvanishing gradients.
    In the second part to this paper, the explicit sufficient second-order conditions will be used to derive error estimates for a finite-element discretization of the control problem.

    \medskip

    \noindent\textcolor{structure}{Key words}\quad Level set, optimal control, nonsmooth optimization, quasilinear elliptic equation, piecewise differentiable function
\end{abstract}

\section{Introduction}

In this paper, we investigate the nonsmooth quasilinear elliptic optimal control problem
\begin{equation}
    \label{eq:P}
    \tag{P}
    \left\{
        \begin{aligned}
            \min_{u\in L^\infty(\Omega)} & j(u) := \int_\Omega L(x,y_u(x)) \dx + \frac{\nu}{2} \norm{u}_{L^2(\Omega)}^2 \\
            \text{s.t.} \quad &-\dive [(b + a(y_u)) \nabla y_u] = u \quad \text{in } \Omega, \quad y_u = 0 \, \text{on } \partial\Omega, \\
            & \alpha \leq u(x) \leq \beta \quad \text{a.e. } x \in \Omega,
        \end{aligned}
    \right.
\end{equation}
where $\Omega$ is a bounded, convex and polygonal domain $\Omega \subset \R^2$; $L: \Omega \times \R \to\R$ is a Carath\'{e}odory function of class $C^2$ with respect to (w.r.t.) the second variable; $b: \overline \Omega \to \R$ is a Lipschitz continuous function; $a: \R \to \R$ is a nonsmooth function; and $\alpha, \beta, \nu \in \R$ satisfy $\beta > \alpha$ and $\nu > 0$. For the precise hypotheses on the data of \eqref{eq:P}, we refer to \cref{sec:assumption}.

The control problem \eqref{eq:P} is interesting since the corresponding state equation arises, for instance, in models of heat conduction where the coefficient in the divergence term of the state equation is the heat conductivity and depends on the temperature $y$ and on the spatial coordinate $x$; see, e.g. \cite{Bejan2013,ZeldovichRaizer1966}. %
When the data are of class $C^2$, the numerical analysis of the discrete approximation of such optimal control problems was investigated by Casas et al. in \cite{CasasTroltzsch2011,CasasTroltzsch2012} for distributed control and in \cite{CasasDhamo2012} for Neumann control. Here, we only assume that the coefficient $a$ is continuous and piecewise $C^2$, which makes the analysis significantly more challenging.

In \cite{ClasonNhuRosch_os2nd}, we studied so-called no-gap second-order necessary and sufficient optimality conditions for \eqref{eq:P} in terms of a general nonsmooth curvature functional of $j$. In this paper, we will formulate an explicit formula for the curvature functional; see \cref{thm:explicit-curvature-func} below. This explicit derivation is based only on an assumption on the gradient of the state function $y_u$ and on the level set of $y_u$ with respect to points $\bar t$ where the nonlinearity $a$ is not differentiable. The sufficient condition will be used to prove a priori numerical error estimates for the discretization of the nonsmooth optimal control problem using the finite element method in the second part of this work.

Let us comment on related works in \cite{ChristofWachsmuth2018} and \cite{WachsmuthWachsmuth2022} for explicitly computing curvature terms of nonsmooth functionals. In \cite{ChristofWachsmuth2018}, the authors considered bang-bang optimal control problems, where the associated state equation is linear and the controls satisfy the box-constraint $-1 \leq u \leq 1$ a.e. in $\Omega$. They observed that the set $\{u \in L^\infty(\Omega) \mid -1 \leq u \leq 1 \, \text{a.e. in } \Omega \}$ possesses a curvature in the space $\mathcal{M}(\Omega)$, the dual of the space consisting of all continuous functions vanishing on the boundary. Under a structural assumption imposed on the corresponding adjoint state and its level set, they established the curvature of the set $\{u \in L^\infty(\Omega) \mid -1 \leq u \leq 1 \, \text{a.e. in } \Omega \}$ by using a directional Taylor-like expression in the $L^1(\Omega)$-norm.
In \cite{WachsmuthWachsmuth2022},
the authors first considered general optimization problems with a nonuniformly convex and nonsmooth integral functional and then applied the obtained results to bang-off-bang optimal control problems with smooth semilinear elliptic partial differential equations.
In these papers, the authors first proved explicit formulas for the case $\Omega \subset \R$ and then lifted the obtained result to the higher-dimensional situation by using partitions of unity.

Here we construct an explicit formula of the nonsmooth curvature functional through a thorough analysis of level sets. In particular, we study the structure and the continuity in terms of determining functions of the level sets, a Green's formula, and the continuity of integrals over level sets.
To fix ideas, let $y: \overline\Omega \to \R$ be a $C^1$ function with an image $\mathcal{R}(y)$. For any $t \in \mathcal{R}(y)$, the preimage
\[
    \{y=t\} := \{ x \in \overline\Omega \mid y(x) =t \}
\]
is called \emph{level set} of $y$ corresponding to the \emph{level value} $t$. These level sets are compact as a result of the continuity of $y$ and the boundedness of $\bar\Omega$. When $\Omega = \R^2$ and $y$ is of class $C^2$, Sard's theorem (see, e.g. \cite{Sard1942} and Appendix 1 in \cite{Guillemin1974}) states that the set of critical values
\[
    \{t \in \mathcal{R}(y) \mid \nabla y(x) = 0 \, \text{for some } x \in \{y=t\} \}
\]
is of one-dimensional (Lebesgue) measure zero.
The Implicit Function Theorem further yields that for almost all (a.a.) $t \in \mathcal{R}(y)$, the level sets $\{y=t\}$ are simple curves of class $C^2$. When $y$ is Lipschitz continuous, it was shown in \cite{AlbertiBianchiniCrippa2013} that for a.a. $t \in \mathcal{R}(y)$, any connected component of $\{y=t\}$ is either a point or a simple curve with Lipschitz parametrization; see statement (iv) in Theorem 2.5 in \cite{AlbertiBianchiniCrippa2013}.
If $y$ is an element of the Sobolev space $W^{2,1}(\Omega)$, then for a.a. $t \in \mathcal{R}(y)$ the level set $\{y=t\}$ is a union of a finitely many disjoint $C^1$ simple curves; see Theorem 5.1 and Corollary 5.2 in \cite{Bourgain2013}.
However, here we are interested in the structure of the level set $\{y=t\}$ for a given concrete value of $t$ and cannot rely on properties that only hold almost everywhere.

The study of level sets at given level values is also relevant to the analysis of level set methods, first proposed by Osher and Sethian \cite{OsherSethian1988}. These methods rely in part on the theory of curve and surface evolution given in \cite{Sethian1985} and stem from the observation that the level set function $y$ in compressible flows carries information on the stretching of the fluid-structure interface. A key property in the analysis is the continuity of level sets with respect to changes in the function or the level value.
For example, \cite{CottetMaitre2006} proved the continuity at $t=0$ of a functional, mapping $t \in \R$ to a corresponding integral over the level set $\{y=t\}$, provided that the function $y$ belongs to class of $C^2$ and that the gradient $\nabla y$ is different from zero on $\{y=0\}$.

In this work, under the sole assumption that $y$ is $C^1$ and vanishes on the boundary $\partial\Omega$, we apply the Implicit Function Theorem to show that any connected component of the level set on which the gradient of $y$ does not vanish is indeed a $C^1$ closed simple curve. Furthermore, if the gradient of $y$ is different from zero on $\{y=t\}$, then the level set consists of finitely many closed simple curves; see \cref{prop:decomposition-levelset} below.

We then study the continuity of integrals over level sets in term of functions that determine the level sets. Specifically, we shall show that if $y_n \to y$ strongly in $C^1(\overline\Omega)$ and if the gradient of $y$ does not vanish on $\{y=t\}$, then the integrals over $\{y_n=t\}$ tend to the one over $\{y=t\}$; see \cref{prop:continuity-integral-levelsets} below. In order to show this, we first present in \cref{lem:Implicit_Function-Theorem-extended-boundary} an extension of the Implicit Function Theorem applying to a point belonging to $\{y=t\}$, where the implicit functions can be extended to the boundary of a neighborhood of the mentioned point. We then use partitions of unity to show the continuity of integrals on level sets.
In addition, we prove in \cref{prop:levelset-unique-component-continuity} (see also \cref{cor:levelset-continuity-in-functions}) that in any small $\epsilon$-neighborhood of a connected component of $\{y=t\}$, there is one and only one connected component of $\{y_n=t\}$ for $n$ large enough.

We use these properties to establish a version of Green's first identity; see, e.g. identity (2.10) in Chapter II in \cite{Gilbarg_Trudinger} for the classical version. There, the domains of integration, $\{y_1 < t < y_2\} \cup \{y_1 > t > y_2\}$, are determined by two $C^1$ functions $y_1, y_2$ with nonvanishing gradients, and boundaries of these domains are parts of level sets of $y_1$ and $y_2$ at the same level value $t$. As a result, these domains of integration might admit some cusps and then not have Lipschitz boundaries in the sense of Ne\v{c}as, see, e.g. Definition 1.2.1.1 \cite{Grisvard1985}.
In order to validate Green's first identity on these nonsmooth domains,
we show in \cref{lem:decomposition-domains} below that the open set $\{y_1 < t < y_2\}$ can be decomposed into at most countably many disjoint open connected subdomains.
Each of such subdomains is approximated by regular domains with boundaries being curvilinear polygons of class $C^1$. By passing to the limit in Green's first identity applied over approximating regular domains and summing up the obtained identities, we arrive at the Green formula on $\{y_1 < t < y_2\} \cup \{y_1 > t > y_2\}$; see \cref{lem:Green-identity} and \cref{prop:Green-identity-general}.
Together with continuity properties of level sets and of integrals on them,
with the aid of this, we derive an explicit formula for the curvature functional in \cref{thm:explicit-curvature-func} and use this to obtain second-order necessary and sufficient optimality conditions in explicit forms obtained in \cref{thm:2nd-OS-nec-explicit,thm:2nd-OS-suf-explicit}, respectively.

\medskip

The plan of the paper is as follows. We conclude this section with some general notation. In the next section, we investigate the structure and the continuity properties of the level sets as well as derive Green's first identity on open sets determined by two functions and a level value. \cref{sec:nonsmooth-optimal-control} is then devoted to the study of the nonsmooth quasilinear optimal control problem \eqref{eq:P}: In \cref{sec:assumption},
we rigorously state the assumptions for \eqref{eq:P} and provide some preliminary results from \cite{ClasonNhuRosch_os2nd}.
An explicit formula of the curvature functional of the objective functional is derived in \cref{sec:Q2-formulation}.
Finally, in \cref{sec:explicit-2nd-OCs}, explicit second-order necessary and sufficient optimality conditions are presented.

\medskip

\paragraph*{Notation.}
We denote by $B_X(u,\rho)$ and $\overline B_X(u,\rho)$ the open and closed balls in a Banach space $X$ of radius $\rho >0$ centered at $u \in X$, respectively.
For Banach spaces $X$ and $Y$, the notation $X \hookrightarrow (\Subset) Y$ is understood that $X$ is continuously (compactly) embedded in $Y$.
For a given function $g:\overline\Omega\to \R$ and a subset $A \subset \R$, $\{ g \in A \}$ denotes the set of all points $x \in \overline \Omega$ for which $g(x) \in A$. For functions $g_1,g_2$ and subsets $A_1, A_2 \subset \R$, we set $\{ g_1 \in A_1, g_2 \in A_2 \} := \{ g_1 \in A_1 \} \cap \{ g_2 \in A_2\}$.
For any set $\omega \subset \overline \Omega$, we denote by $\1_{\omega}$
the characteristic function
of $\omega$, i.e., $\1_\omega(x) = 1$ if $x \in \omega$ and $\1_\omega(x) =0$ otherwise.
The symbol $\Ha^1$ stands for the one-dimensional Hausdorff measure on $\R^2$ that is scaled as in \cite{Evans1992}, Def.~2.1.
For a measurable two-dimensional subset $A$, by $\meas_{\R^2}(A)$, we denote the two-dimensional Lebesgue measure of $A$.
Finally, we write the symbol $C$ for a generic positive constant, which may be different at different places of occurrence and the notation, e.g. $C_\xi$ for a constant depending only on the parameter $\xi$.

\section{Analysis of level sets}
\label{sec:levelset-facts}
In this section, $\Omega$ is assumed to be a convex and bounded domain in $\R^2$ only.

\subsection{Structure of level sets}
\label{sec:levelset-structure}
We first recall from \cite{GiovanniLeoni2017} that a connected component of a set $V \subset \R^d$, $d \geq 1$,
is any element of the class of connected subsets of $V$ that is maximal with respect to inclusion and that
a simple closed $C^1$ curve in $\R^d$
is a curve that admits a $C^1$ parametrization $\gamma: [h,k] \to \R^d$ which is injective on $[h,k)$ and satisfies $\gamma(h)=\gamma(k)$.

We shall investigate the structure of any connected component of the level sets of $C^1$-functions that vanish on the boundary and have nonvanishing gradients on these components. We begin with the following $C^1$ extension of a smooth function defined over convex domains.

\begin{proposition}
    \label{prop:C1-extension-on-convex-domain}
    Let $G$ be a convex domain in $\R^d$ with $d \geq 1$ and let $f \in C^1(\overline{G})$ be arbitrary. Then there exists a function $\bar f: \R^d \to \R$ with the following properties:
    \begin{enumerate}[label=(\alph*)]
        \item \label{item:C1-continuous} $\bar f$ is of class $C^1$;
        \item \label{item:C1-extension-on-domain} $\bar f = f$ and $\nabla \bar f = \nabla f$ on $\overline{G}$.
    \end{enumerate}
\end{proposition}
\begin{proof}
    We will show that Whitney's Extension Theorem (see, e.g. Chapter 6 in \cite{Evans1992} and Theorem 9.43 in \cite{GiovanniLeoni2017}) can be applied and thus derive the desired extension function. To this end, for any $x_1, x_2 \in \overline{G}$ with $x_1 \neq x_2$, set
    \[
        R(x_2, x_1) := \frac{f(x_2)-f(x_1) - \nabla f(x_1) \cdot (x_2-x_1)}{|x_2-x_1|}.
    \]
    For any compact subset $M \subset \overline{G}$, we define the function $\rho_M : (0,1) \to (0, \infty)$ by
    \[
        \rho_M( \tau) := \sup\{|R(x_2,x_1)| \mid 0 < |x_1 -x_2| < \tau, x_1, x_2 \in M \}, \quad \tau \in (0,1).
    \]
    In light of Whitney's Extension Theorem, in order to derive assertions \ref{item:C1-continuous} and \ref{item:C1-extension-on-domain}, we need to show that
    \begin{equation}
        \label{eq:C1-extension-limit}
        \rho_M(\tau) \to 0 \quad \text{as} \quad \tau \to 0^+.
    \end{equation}
    Indeed, for any compact subset $M \subset \overline{G}$, and any $x_1, x_2 \in M$, $x_1 \neq x_2$, we deduce from the Mean Value Theorem and the convexity of $G$ that
    \begin{align*}
        |R(x_2, x_1)| & = \frac{| \nabla f(x_1 + \theta(x_2-x_1)) \cdot (x_2 -x_1) - \nabla f(x_1) \cdot (x_2 -x_1)|}{|x_1 -x_2|}\\
        & \leq |\nabla f(x_1 + \theta(x_2-x_1))- \nabla f(x_1) |
    \end{align*}
    with some $\theta \in (0,1)$. There then holds
    \[
        \rho_M(\tau) \leq \sup \{|\nabla f(x_2)- \nabla f(x_1) | \mid | x_1 -x_2| < \tau, x_1, x_2 \in M \}
    \]
    for all $\tau \in (0,1)$.
    From this and the uniform continuity of $\nabla f$ on the compact set $M$, we have \eqref{eq:C1-extension-limit}.
\end{proof}

It is well-known from differential geometry (see, e.g. \cite{JohnMilnor1965} and \cite{Guillemin1974}) that every compact one-dimensional manifold without boundary is a union of finitely many disjoint closed curves. As a direct consequence, we obtain in \cref{prop:decomposition-levelset} below the decomposition into finitely many disjoint closed curves of level sets of a $C^1$-function that vanishes on the boundary of a two-dimensional convex domain and has a nonvanishing gradient. For the sake of convenience, we provide its proof based on \cref{prop:C1-extension-on-convex-domain} and the following result on the extension of an implicit function. Besides, this extension result plays an important role in proving the continuity of integrals on level sets in \cref{sec:continuity-integral-levelsets} below.

\begin{lemma}
    \label{lem:Implicit_Function-Theorem-extended-boundary}
    Let $t \in \R$ be arbitrary and let $y \in C^1(\overline\Omega)$ be such that $y_{\mid \partial\Omega} =0$ and that there is a point $x_0 := (x^1_0, x^2_0) \in \{y=t \}$.
    Let $\Cu$ be a connected component of $\{y=t\}$ such that $x_0 \in \Cu$. Assume that there exists a constant $r>0$ satisfying
    \begin{equation}
        \label{eq:radii-condition}
        \{y=t\} \cap \overline{B}_{\R^2}(x_0, r) = \Cu \cap \overline{B}_{\R^2}(x_0, r)
    \end{equation}
    and
    \begin{equation}
        \label{eq:boundary-touch}
        \{y=t \} \cap \partial \overline{B}_{\R^2}(x_0, r) = \{p, q \}, \quad p \neq q.
    \end{equation}
    Assume further that
    \begin{equation}
        \label{eq:nonvanishing-gradient-at-point}
        \left|\frac{\partial y}{\partial{x^2}}(x)\right| \geq \alpha > 0 \quad \text{for all } x \in \overline{B}_{\R^2}(x_0, r) \cap \overline\Omega
    \end{equation}
    with some constant $\alpha>0$. Then the following assertions hold.
    \begin{enumerate}[label=(\alph*)]
        \item \label{item:paramterization-g} There exist constants $h_0 >0, k_0> 0$ and a $C^1$-function $g_0: [x^1_0-h_0,x^1_0+k_0] \to \R$ satisfying
            \begin{equation}
                \label{eq:unique-C1-curve-locally-extension-boundary}
                \left\{
                    \begin{aligned}
                        & g_0(x^1_0) =x^2_0,\\
                        & \overline B_{\R^2}(x_0, r) \cap \{y=t\} = \{ (\tau, g_0(\tau)) \mid \tau \in [x^1_0-h_0,x^1_0+k_0] \}, \\
                        & \{(x^1_0-h_0, g_0(x^1_0-h_0)), (x^1_0+k_0, g_0(x^1_0+k_0))\} = \partial\overline B_{\R^2}(x_0, r) \cap \{y=t\}.
                    \end{aligned}
                \right.
            \end{equation}
        \item \label{item:parameterization-approximation}
            If, in addition, there exists a sequence $\{y_n\} \in C^1(\overline\Omega)$ satisfying that ${y}_{n\mid{\partial\Omega}} =0$ and that
            \begin{equation}
                \label{eq:approximation-func-in-ball}
                \left\{
                    \begin{aligned}
                        & y_n \to y \quad \text{strongly in } C^1(\overline\Omega),\\
                        & \{y_n = t\} \cap \overline B_{\R^2}(x_0,r_0) \, \text{is an arc for all $n$ large enough},\\
                        & \frac{\partial y}{\partial{x^2}}(x)\frac{\partial y_n}{\partial{x^2}}(x) \geq \tilde \alpha >0 \quad \text{for all } x \in \overline{B}_{\R^2}(x_0, r) \cap \overline\Omega \, \text{and for all $n$ sufficiently large,}
                    \end{aligned}
                \right.
            \end{equation}
            then there is an integer $n_0 = n_0(x_0,r)$ such that, for any $n \geq n_0$, constants $h_n, k_n >0$ and a $C^1$ function $g_n: [x^1_0-h_n,x^1_0+k_n] \to \R$ exist and satisfy the following properties:
            \begin{enumerate}[label=(\roman*)]
                \item \label{item:parameterization-gn}
                    $\overline B_{\R^2}(x_0, r) \cap \{y_n=t\} = \{ (\tau, g_n(\tau)) \mid \tau \in [x^1_0-h_n,x^1_0+k_n] \}$;
                \item \label{item:boundary-points}
                    $\{(x^1_0-h_n, g_n(x^1_0-h_n)), (x^1_0+k_n, g_n(x^1_0+k_n))\} = \partial\overline B_{\R^2}(x_0, r) \cap \{y_n=t\}$;
                \item \label{item:limit-hn-kn} $(x^1_0-h_n, g_n(x^1_0-h_n)) \to (x_0^1 - h_0, g_0(x_0^1 -h_0))$ and $(x^1_0+k_n, g_n(x^1_0+k_n)) \to (x^1_0+k_0, g_0(x^1_0+k_0))$;
                \item \label{item:gn-g-Lip}
                    For any $\epsilon >0$ and for all $\tau \in [x^1_0-h_0 + \epsilon,x^1_0+k_0 - \epsilon]$, there hold
                    \begin{equation}
                        \label{eq:gn-g-esti}
                        |g_n(\tau) - g_0(\tau)| \leq C_1 \norm{y_n - y}_{C(\overline\Omega)} \quad \text{ for $n$ large enough }
                    \end{equation}
                    and
                    \begin{equation}
                        \label{eq:gn-g-limits}
                        \left\{
                            \begin{aligned}
                                & g_n'(\tau) \to g_0'(\tau), \\
                                & g_n'(x^1_0-h_n) \to g_0'(x_0^1-h_0), \\
                                & g_n'(x^1_0 + k_n) \to g_0'(x_0^1 + k_0).
                            \end{aligned}
                        \right.
                    \end{equation}
                    Moreover, for all $\tau \in [x^1_0-h_n,x^1_0+k_n]$, one has
                    \begin{equation}
                        \label{eq:gn-deri-bound}
                        |g_n'(\tau)| \leq C_2.
                    \end{equation}
            \end{enumerate}
            Here $\tilde{\alpha}$ and $C_1, C_2$ are constants independent of $n, \epsilon$, and $\tau$.
    \end{enumerate}
\end{lemma}
\begin{remark}
    The second condition in \eqref{eq:approximation-func-in-ball} is, in fact, fulfilled if the first one in \eqref{eq:approximation-func-in-ball} and the second one in \eqref{eq:unique-C1-curve-locally-extension-boundary} hold. This fact shall be proven in \cref{prop:levelset-unique-component-continuity} below; see also, the proof of that proposition.
\end{remark}
\begin{proof}
    Without loss of generality, assume that $x_0 = (0,0)$, i.e., $x_0^1 := 0$ and $x_0^2 :=0$.
    We first see from the condition \eqref{eq:nonvanishing-gradient-at-point} that the following implication holds
    \begin{equation}
        \label{eq:ordering-x1}
        \begin{cases}
            & x = (x^1,x^2), \tilde{x} = (\tilde{x}^1,\tilde{x}^2) \in \Cu \cap \overline{B}_{\R^2}(0, r),\\
            & x \neq \tilde{x}
        \end{cases}
        \implies x^1 \neq \tilde{x}^1.
    \end{equation}
    Consequently, one has $p^1 \neq q^1$, where $p^1$ and $q^1$ stand, respectively, for the first component of the end points $p$ and $q$.
    We can assume without loss of generality that
    \[
        p^1 < q^1.
    \]
    Let $\Cu_{x^1}$ be the projection of the set $\Cu \cap \overline{B}_{\R^2}(x_0, r)$ onto the $x^1$-axis. Since $\Cu \cap \overline{B}_{\R^2}(x_0, r)$ is connected and compact, $\Cu_{x^1}$ is also connected and compact; see, e.g. Theorem 4.22 in \cite{Rudin}.
    Therefore, $\Cu_{x^1}$ must be a closed bounded interval of $\R$. Combining this with the implication \eqref{eq:ordering-x1} yields
    \begin{equation}
        \label{eq:x1-projection}
        \Cu_{x^1} = [p^1,q^1].
    \end{equation}
    Moreover, since $0 \in \Cu \cap \overline{B}_{\R^2}(x_0, r)$, and $p$ and $q$ are the end points of the arc $\Cu \cap \overline{B}_{\R^2}(x_0, r)$, there holds
    \begin{equation}
        \label{eq:pq-x1-order}
        p^1 < 0 < q^1,
    \end{equation}
    due to the implication \eqref{eq:ordering-x1}.
    We now prove assertions \ref{item:paramterization-g} and \ref{item:parameterization-approximation}.
    \\
    \noindent\emph{Ad \ref{item:paramterization-g}:}
    In light of \cref{prop:C1-extension-on-convex-domain}, $y$ admits a $C^1$ extension, denoted also by $y$, on $\R^2$. By the continuity of $\nabla y$, we can assume that the condition \eqref{eq:nonvanishing-gradient-at-point} is valid
    on an open neighborhood $\mathcal{O}$ of $\overline{B}_{\R^2}(0, r) \cap \{y=t\}$ (possibly with a different positive constant, also denoted by $\alpha$),
    which contains the arc $\Cu \cap \overline{B}_{\R^2}(0, r)$.
    According to the condition \eqref{eq:radii-condition}, we can assume that the open set $\mathcal{O}$ satisfies
    \begin{equation}
        \label{eq:radii-condi2}
        \mathcal{O} \cap \{y=t\} = \mathcal{O} \cap \Cu.
    \end{equation}
    Therefore, one can apply the Implicit Function Theorem in every point $x = (x^1,x^2)$ of $\Cu \cap \overline{B}_{\R^2}(0, r)$.
    We now take $x = (x^1,x^2) \in \Cu \cap \overline{B}_{\R^2}(0, r)$ arbitrarily but fixed. Thanks to the Implicit Function Theorem, there exist constants $h_x, k_x, \rho_x >0$ and a $C^1$-function $g_x: (x^1- h_x, x^1+k_x) \to \R$ satisfying
    \begin{equation}
        \label{eq:gx-func}
        \left\{
            \begin{aligned}
                & B_{\R^2}(x, \rho_x) \subset \mathcal{O}, \quad g_x(x^1) = x^2,\\
                & y(\tau, g_x(\tau)) = t \quad \text{for all } \tau \in (x^1- h_x, x^1+k_x), \\
                & \left((\tau,\xi) \in B_{\R^2}(x, \rho_x) \, \text{and } y(\tau, \xi) = t \right) \quad \implies \quad \xi = g_x(\tau).
            \end{aligned}
        \right.
    \end{equation}
    Since $\Cu \cap \overline{B}_{\R^2}(0, r) \subset \bigcup \{B_{\R^2}(x, \rho_x) \mid x \in \Cu \cap \overline{B}_{\R^2}(0, r) \}$ and $\Cu \cap \overline{B}_{\R^2}(0, r)$ is compact, there exist finitely many points $x_1,x_2, \ldots, x_s \in \Cu \cap \overline{B}_{\R^2}(0, r)$ with $s \in \mathbb N$ and $C^1$-functions $g_i: (x^1_i- h_i, x^1_i+k_i) \to \R$ with $g_i := g_{x_i}$, $h_i := h_{x_i}$, $k_i := k_{x_i}$, $1 \leq i \leq s$, such that
    \[
        \Cu \cap \overline{B}_{\R^2}(0, r) \subset \cup_{i=1}^s B_{\R^2}(x_i, \rho_i)
    \]
    with $\rho_i := \rho_{x_i}$. From this and \eqref{eq:x1-projection}, we have
    \[
        [p^1, q^1] \subset \cup_{i=1}^s (x^1_i- h_i, x^1_i+k_i) =: V.
    \]
    We now construct the function $g_0$ defined over the open set $V$ by setting
    \begin{equation}
        g_0(\tau) = g_i(\tau) \quad \text{if} \quad \tau \in (x^1_i- h_i, x^1_i+k_i) \quad \text{for some } i.
    \end{equation}
    Thanks to the implication in \eqref{eq:gx-func}, the function $g_0$ is actually well-defined and is of class $C^1$ on the open neighborhood $V$ of $[p^1, q^1]$.
    Moreover, for later use, we obtain by simple computation that
    \begin{equation}
        \label{eq:implicit-func-deri}
        g_0'(\tau) = - \frac{\frac{\partial y}{\partial{x^1}}(\tau, g_0(\tau))}{\frac{\partial y}{\partial{x^2}}(\tau, g_0(\tau))}, \quad \tau \in V.
    \end{equation}
    In view of \eqref{eq:nonvanishing-gradient-at-point} and from the continuity of $\nabla y$ over $\overline\Omega$, there holds
    \begin{equation}
        \label{eq:implicit-func-deri-bound}
        |g_0'(\tau)| \leq \frac{\norm{\nabla y}_{C(\overline{\Omega})}}{\alpha} \quad \text{for all } \tau \in V.
    \end{equation}
    Now setting $h_0 := - p^1$ and $k_0 := q^1$ and exploiting \eqref{eq:pq-x1-order}, one has $h_0, p_0 >0$.
    On the other hand, we have \eqref{eq:unique-C1-curve-locally-extension-boundary} from \eqref{eq:gx-func} and the definition of function $g_0$.
    We thus obtain \ref{item:paramterization-g}.

    \noindent\emph{Ad \ref{item:parameterization-approximation}:}
    The proof of assertion \ref{item:parameterization-approximation} is divided into several steps as follows. \\
    \noindent\emph{$\bullet$ Step 1: There is an integer $n_0^1$ such that, for any $n \geq n_0^1$, a number $\xi_n \in (-r_0,r_0)$ uniquely exists and fulfills
        \begin{equation}
            \label{eq:xi-n}
            y_n(0,\xi_n) =t.
        \end{equation}
        In fact, as a consequence of the last condition in \eqref{eq:approximation-func-in-ball}, the functions $y(0,\cdot)$ and $y_n(0,\cdot)$ are both either strictly increasing or strictly decreasing. We now only consider the situation where both functions are strictly increasing, since the other is analyzed analogously. It suffices to prove that \eqref{eq:xi-n}, for each $n$ large enough, admits at least one solution in $(-r_0,r_0)$. To this end, by contradiction assume that
        there exists a subsequence, denoted in the same way, of $\{n\}$ such that \eqref{eq:xi-n} has no solutions in $(-r_0,r_0)$ for all $n \geq 1$. We then split $\{n\}$ into subsequences, also denoted by $\{n\}$, satisfying one of two inequalities
        \begin{multicols}{2}
            \begin{enumerate}[label= (I\arabic*)]
                \item \label{item:bigger-t} $y_n(0,\xi) > t$ for all $n \geq 1$ and $\xi \in (-r_0, r_0)$;
                \item \label{item:less-t} $y_n(0,\xi) < t$ for all $n \geq 1$ and $\xi \in (-r_0, r_0)$.
            \end{enumerate}
        \end{multicols}
        In order to show a contradiction, we now only consider \ref{item:bigger-t} since the case \ref{item:less-t} is similarly analyzed. We have for all $n\geq 1$ that
        \[
            y_n(0,-\frac{r_0}{2}) > t = y(0,0) > y(0,-\frac{r_0}{2}),
        \]
        which yields that
        \[
            \norm{y_n -y}_{C(\overline\Omega)} \geq |y_n(0,-\frac{r_0}{2}) - y(0,-\frac{r_0}{2})| > t - y(0,-\frac{r_0}{2}) \quad \text{for all } n \geq 1.
        \]
        Passing to the limit and exploiting the first condition in \eqref{eq:approximation-func-in-ball} give $0 \geq t - y(0,-\frac{r_0}{2}) > 0$, a contradiction.
    }%
    \noindent\emph{$\bullet$ Step 2: For any $n \geq n_0^1$, there exist constants $h_n, k_n >0$ and a function $g_n \in C^1[x_0^1-h_n,x_0^1+k_n]$, all of which fulfill assertions \ref{item:parameterization-gn} and \ref{item:boundary-points}, where $n_0^1$ is defined in Step 1.} To this end, we first see that $x_{0,n} := (0, \xi_n) = (x_0^1, \xi_n) \in B_{\R^2}(x_0,r_0)$ with $\xi_n \in (-r_0, r_0)$ being a unique solution to \eqref{eq:xi-n}. Using the argument in the proof of assertion \ref{item:paramterization-g} for the point $x_{0,n}$ and the function $y_n$ in the places of $x_0$ and $y$, respectively, Step 2 is proven.
    Moreover, similar to \eqref{eq:implicit-func-deri} and \eqref{eq:implicit-func-deri-bound}, there hold for function $g_n$
    \begin{equation}
        \label{eq:gn-derivative}
        g_n'(\tau) = - \frac{\frac{\partial y_n}{\partial{x^1}}(\tau, g_n(\tau))}{\frac{\partial y_n}{\partial{x^2}}(\tau, g_n(\tau))} \quad \text{and} \quad |g_n'(\tau)| \leq C \norm{\nabla y_n}_{C(\overline{\Omega})}
    \end{equation}
    for all $\tau \in [x_0^1-h_n, x_0^1+k_n]$.
    \\
    \noindent\emph{$\bullet$ Step 3: Verification of assertion \ref{item:limit-hn-kn}.}
    Indeed, we set the points belonging the boundary $\partial B_{\R^2}(x_0,r_0)$ of $\{y=t\}$ and $\{y_n=t\}$ as follows
    \[
        p_0 := (-h_0,g_0(-h_0)), \quad q_0 := (k_0, g_0(k_0)), \quad p_n := (-h_n,g_n(-h_n)), \quad \text{and} \quad q_n := (k_n, g_n(k_n)).
    \]
    Recall that all of constants $h_0, h_n, k_0, k_n$ are positive.
    Thanks to the boundedness of $\{p_n\}$ and the closedness of $\partial B_{\R^2}(x_0,r_0)$, we have from the first condition in \eqref{eq:approximation-func-in-ball} that
    $p_n \to p_0$, i.e.,
    \[
        h_n \to h_0 \quad \text{and} \quad g_n(-h_n) \to g_0(-h_0).
    \]
    Analogously, we have
    \[
        k_n \to k_0 \quad \text{and} \quad g_n(k_n) \to g_0(k_0).
    \]
    We then obtain \ref{item:limit-hn-kn}. \\
    \noindent\emph{$\bullet$ Step 4: Verification of assertion \ref{item:gn-g-Lip}}. Indeed, for any $\epsilon >0$, there exists an integer $n_0^2$ such that
    \[
        [-h_0 + \epsilon, k_0 - \epsilon] \subset [-h_n, k_n] \cap [-h_0, k_0] \quad \text{for all } n \geq n_0^2.
    \]
    For any $\tau \in [-h_0 + \epsilon, k_0 - \epsilon]$, by employing the Mean Value Theorem, we arrive at
    \begin{align*}
        t = y_n(\tau,g_n(\tau)) & = y(\tau,g_n(\tau)) + w_n(\tau,g_n(\tau)) \\
        & = y(\tau,g_0(\tau)) + \frac{\partial y}{\partial x^2}(\tau,g_0(\tau) + \theta_n (g_n(\tau)-g_0(\tau)))(g_n(\tau)-g_0(\tau)) + w_n(\tau,g_n(\tau)) \\
        & = t + \frac{\partial y}{\partial x^2}(\tau, g_0(\tau)+\theta_n (g_n(\tau)-g_0(\tau)))(g_n(\tau)-g_0(\tau)) + w_n(\tau,g_n(\tau))
    \end{align*}
    for some $\theta_n \in (0,1)$ and $w_n := y_n - y$. Obviously, one has $(\tau, g_0(\tau)+\theta_n (g_n(\tau)-g_0(\tau))) \in B_{\R^2}(x_0,r_0) \cap \overline\Omega$ since $(\tau, g_0(\tau))$ and $(\tau, g_n(\tau))$ belong to $B_{\R^2}(x_0,r_0) \cap \overline\Omega$. Combing this with the last condition in \eqref{eq:approximation-func-in-ball} yields the estimate in \eqref{eq:gn-g-esti}. From this and the derivatives of $g_0$ and $g_n$ in \eqref{eq:implicit-func-deri} and \eqref{eq:gn-derivative}, all the limits in \eqref{eq:gn-g-limits} then follows from the first condition in \eqref{eq:approximation-func-in-ball} and from \ref{item:limit-hn-kn}. Finally, the limit in \eqref{eq:approximation-func-in-ball} and the estimate in \eqref{eq:gn-derivative} imply \eqref{eq:gn-deri-bound}.
\end{proof}

\begin{proposition}
    \label{prop:decomposition-levelset}
    Let $t \in \R$ be arbitrary and let $y \in C^1(\overline\Omega)$ be such that $y_{\mid\partial\Omega} =0$ and that $\{y=t \} \neq \emptyset$.
    Assume that $\Cu$ is a connected component of $\{y=t\}$. Then the following properties hold:
    \begin{enumerate}[label=(\alph*)]
        \item \label{item:C1-closedcurve} If
            \begin{equation}
                \label{eq:nonvanishing-gradient-on-component}
                |\nabla y(x)| > 0 \quad \text{for all } x \in \Cu,
            \end{equation}
            then $\Cu$ is a closed $C^1$ simple curve;
        \item \label{item:decomposition-levelset} If
            \begin{equation}
                \label{eq:nonvanishing-gradient-on-levelset}
                |\nabla y(x)| > 0 \quad \text{for all } x \in \{y=t \},
            \end{equation}
            then the level set $\{y=t \}$ decomposes into finitely many disjoint closed $C^1$ simple curves.
    \end{enumerate}
\end{proposition}
\begin{proof}
    \emph{Ad \ref{item:C1-closedcurve}:}
    We first show that for any $x \in \Cu$, there exist positive constants $r_x >0$ and a $C^1$-function $\gamma_x: [-1,1] \to \R^2$ satisfying
    \begin{equation}
        \label{eq:unique-C1-curve-locally}
        \overline B_{\R^2}(x, r_x) \cap \{y=t\} = \{ \gamma_x(\eta) \mid \eta \in [-1,1] \} \quad \text{and} \quad \gamma_x(\eta_x) =x, \gamma_x(\pm 1) \in \partial B_{\R^2}(x, r_x) \cap \{y=t\}
    \end{equation}
    for some $\eta_x \in (-1,1)$.
    For that purpose, we first take $x:=(x^1,x^2) \in \Cu$ arbitrarily. Without loss of generality, we can assume that $\frac{\partial y}{\partial{x^2}}(x) \neq 0$ and thus $|\frac{\partial y}{\partial{x^2}}(\zeta) | > 0$ for all $\zeta \in \overline{B}_{\R^2}(x, r_x) \cap \overline\Omega$ for some constant $r_x >0$.
    We now can take $r_x$ small enough (if necessary) such that the arc $\Cu \cap \overline{B}_{\R^2}(x, r_x)$ has two different two end points.

    \Cref{lem:Implicit_Function-Theorem-extended-boundary} implies that there exist constants $h_x,s_x >0$ and a $C^1$-function $g_x: [x^1 - h_x, x^1 + s_x] \to \R$ satisfying
    \begin{equation}
        \label{eq:parameterization-local}
        \left\{
            \begin{aligned}
                & g_x(x^1) = x^2,\\
                & \overline B_{\R^2}(x, r_x) \cap \{y=t\} = \{ (\tau, g_x(\tau)) \mid \tau \in [x^1 - h_x, x^1 + s_x] \}, \\
                & \{(x^1-h_x, g_x(x^1-h_x)), (x^1+s_x, g_x(x^1+s_x))\} = \partial\overline B_{\R^2}(x, r_x) \cap \{y=t\}.
            \end{aligned}
        \right.
    \end{equation}
    Setting $\eta := \frac{2}{s_x + h_x} \tau + 1- \frac{2(x^1+s_x)}{s_x + h_x}$ and $\gamma_x(\eta) := (\tau, g_x(\tau))$ and exploiting \eqref{eq:parameterization-local}, we have \eqref{eq:unique-C1-curve-locally}.

    Combining \eqref{eq:unique-C1-curve-locally} with the compactness of $\Cu$, there exist points $x_1,\ldots, x_{m_\Cu} \in \Cu$ and positive constants $r_{i} >0$, and $C^1$-functions $\gamma_i:[-1,1] \to \R^2$, $1 \leq i \leq m_\Cu$, such that
    \begin{equation}
        \label{eq:finite-union}
        \left\{
            \begin{aligned}
                & \Cu \subset \cup_{i=1}^{m_\Cu} B_{\R^2}(x_i, r_i) \cap \{y=t\},\\
                &\overline B_{\R^2}(x_i, r_i) \cap \{y=t\} = \{ \gamma_i(\eta) \mid \eta \in [-1,1] \},\\
                &\gamma_i(\pm 1) \in \partial B_{\R^2}(x_i, r_i) \cap \{y=t\}.
            \end{aligned}
        \right.
    \end{equation}
    From this and the fact that $\Cu$ is a connected component of $\{y=t\}$, we deduce that $\Cu$ is a closed $C^1$ simple curve.

    \noindent\emph{Ad \ref{item:decomposition-levelset}:} Since $\{y=t\}$ is compact, we also have \eqref{eq:finite-union} with $\Cu$ being replaced by $\{y=t\}$. Then the level set $\{y=t\}$ is a union of finitely many closed simple curves.
\end{proof}

\subsection{Green's first identity}
\label{sec:Green-first-identity}

We first state the well-known result on the decomposition of open sets in a finite dimensional Euclidean space into at most countably many disjoint open connected components. Since we could not find its proof in the literature, we provide here the arguments for the sake of completeness.
\begin{proposition}
    \label{prop:countable-decomposition-openset}
    Any nonempty open set $V \subset \R^d$, $d \in \N$, can be expressed as an union of disjoint, open, connected components of $V$, at most countable in number, that is, there exist an index set $I$, which is at most countable, and disjoint, open and connected components $V_i$, $i\in I$, of $V$ such that
    \[
        V = \bigcup_{i \in I} V_i.
    \]%
\end{proposition}
\begin{proof}
    For any $x \in V$, we define the set
    \[
        \mathcal{O}_x =\cup \{ Z \mid Z \subset V, x \in Z, Z \, \text{is open and connected in } \R^d \}.
    \]
    Obviously, $\mathcal{O}_x$ is nonempty and open since $V$ is open. Moreover, $\mathcal{O}_x$ is connected; see, e.g. Chapter VI in \cite{Carothers}. On the other hand, $\mathcal{O}_x$ is maximal with respect to containment and thus is an open component of $V$.
    Therefore, the family $\{\mathcal{O}_x \mid x \in V \}$ is a collection of open connected components of $V$. From this and the fact that two arbitrary connected components of $V$ are either identical or disjoint, we can now apply Lindel\"{o}f's Theorem to deduce the desired decomposition of $V$.
\end{proof}

The following lemma represents the decomposition of an open set given by level sets of two functions into at most countably many disjoint open connected subdomains.
\begin{lemma}
    \label{lem:decomposition-domains}
    Let $y_1, y_2 \in C^1(\overline\Omega)$, let $t \in \R$ be given, and let $\mathcal{O}$ be an open subset in $\overline\Omega$.
    Define the open sets $\mathcal{S}^+$ and $\mathcal{S}^{-}$ by
    \[
        \mathcal{S}^{-} := \{ y_1 < t < y_2 \} \cap \mathcal{O} \quad \text{and} \quad \mathcal{S}^{+} := \{ y_1 > t > y_2 \} \cap \mathcal{O}.
    \]
    Assume that the level set $\{y_j=t \}$ has one and only one closed $C^1$ simple curve $\mathcal{C}_j$ in $\mathcal{O}$ for $j=1,2$.
    Assume further that
    \begin{equation}
        \label{eq:boundary-inclusion}
        \partial \mathcal{S}^{\pm} \subset \Cu_1 \cup \Cu_2
    \end{equation}
    and that
    \begin{equation}
        \label{eq:nonvanishing-gradient-Spm}
        \nabla y_1(x) \cdot \nabla y_2(x) \geq \alpha >0 \quad \text{for all } x \in \mathcal{S}^+ \cup \mathcal{S}^{-}
    \end{equation}
    for some constant $\alpha$.
    Let $\wideparen{My_jN}$ stands for a curve lying in $\{y_j = t \}$, $j=1,2$ with end points $M$ and $N$; see \cref{fig:components-decomposition}.
    Then, the following assertions hold:
    \begin{enumerate}[label=(\alph*)]
        \item \label{item:coincidence} if $\Cu_1 = \Cu_2$, then
            $\mathcal{S}^{+} = \mathcal{S}^{-} = \emptyset$;

        \item \label{item:nonintersection} if $\mathcal{C}_1 \cap \mathcal{C}_2 = \emptyset$, then one of the sets $\mathcal{S}^{\pm}$ is empty and the boundary of the other consists of $\Cu_1$ and $\Cu_2$;
        \item \label{item:intersection} if $\Cu_1 \neq \Cu_2$ and $\mathcal{C}_1 \cap \mathcal{C}_2 \neq \emptyset$,
            then the intersection $\mathcal{C}_1 \cap \mathcal{C}_2$ has at most countably many closed connected components. In other words, there exist index sets $I^+, I^{-}, I^{0}$, which are at most countable (and might be empty), and points $M^{+}_i, N^{+}_i$, $M^{-}_m, N^{-}_m$, and $M^{0}_k, N^{0}_k$ with $i \in I^+$, $ m \in I^{-}$ , $k \in I^{0}$, such that
            \begin{equation} \label{eq:boundary-points}
                \mathcal{C}_1 \cap \mathcal{C}_2 = \bigcup_{i \in I^+, m \in I^{-}, k \in I^0} \{M^+_i, N^{+}_i, M^{-}_m, N^{-}_m, \wideparen{M^{0}_kN^{0}_k} \},
            \end{equation}
            where $\wideparen{M^{0}_kN^{0}_k}, k \in I^0,$ denote the curves lying in $\mathcal{C}_1 \cap \mathcal{C}_2$ with end points $M^{0}_k$ and $N^{0}_k$.
            Furthermore, there hold:
            \begin{enumerate}[label=(\roman*)]
                \item If $\mathcal{S}^+ = \emptyset$, then the index set $I^+$ is empty;
                \item If the set $\mathcal{S}^{+}$ is nonempty, then it has at most countably many open connected components $\mathcal{S}^{+}_i$, $i \in I^+$, such that $\partial \mathcal{S}^{+}_i$ consists of two curves $\wideparen{M^{+}_iy_1N^{+}_i}$ and $\wideparen{M^{+}_iy_2N^{+}_i}$.
            \end{enumerate}
            Analogous assertions hold for the set $\mathcal{S}^{-}$.
    \end{enumerate}
\end{lemma}
\begin{proof}
    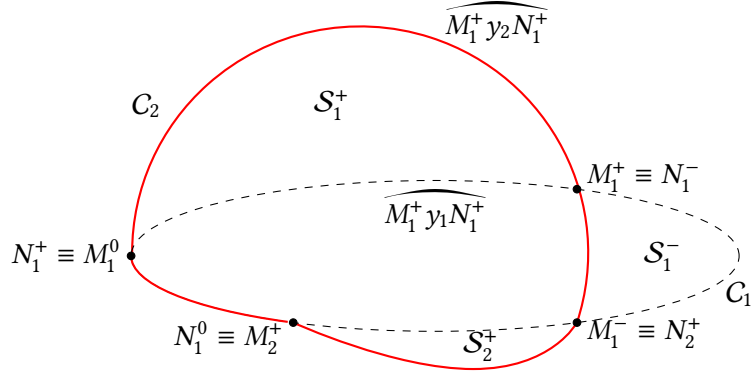
\begin{figure}[t]
        \centering
        \begin{tikzpicture}[>=latex]
            \draw[color=red, thick] (1.8667,-0.8844) arc[radius =3, start angle= -18, end angle= 180];
            \draw[color=red, thick] plot [smooth,tension=1] coordinates { (-1.8667,-0.8844) (0.5,-1.5) (1.8667,-0.8844)};
            \draw[dashed] (0,0) ellipse (4cm and 1cm);
            \draw[color=red,thick] (-4,0) arc(180:241:4cm and 1cm);
            \draw (-1.0,2.0) node[left] {$\mathcal{S}^+_1$};
            \draw (3.0,0.0) node {$\mathcal{S}^{-}_1$};
            \draw (4.0,-0.5) node {$\Cu_1$};
            \draw (-3.8,2.0) node {$\Cu_2$};
            \draw (0.6,-1.2) node {$\mathcal{S}^+_2$};
            \draw (-4,0) node[left]{$N_1^+ \equiv M_1^0$};
            \draw (1.8667,0.8844+ 0.16) node[right]{$M_1^+ \equiv N_1^{-}$};
            \draw (1.8667,-0.8844-0.14) node[right]{$M_1^{-} \equiv N^+_2$};
            \draw (-1.8667,-0.8844-0.2) node[left] {$N_1^0 \equiv M^{+}_2$};
            \draw (0,0.6) node {$\wideparen{M_1^+y_1N_1^+}$};
            \draw (0,3.1) node[right] {$\wideparen{M_1^+y_2N_1^+}$};
            \foreach \Point/\PointLabel in {(-4,0)/, (1.8667,0.8844)/,(1.8667,-0.8844)/,(-1.8667,-0.8844)/}
            \draw[fill=black] \Point circle (0.05) node[above right] {$\PointLabel$};
        \end{tikzpicture}
        \caption{A decomposition into components of the set $\mathcal{{S}}^+ \cup \mathcal{S}^{-}$ where $I^+ = \{1,2\}$, $I^{-} = I^0 = \{1\}$ when $\Cu_1 \cap \Cu_2 \neq \emptyset$. Here $\Cu_1$ and $\Cu_2$ are, respectively, plotted by dashed and solid curves}
        \label{fig:components-decomposition}
    \end{figure}
    \noindent\emph{Ad \ref{item:coincidence}:} By contradiction, assume that $\mathcal{S}^{-} \cup \mathcal{S}^{+} \neq \emptyset$. Assume now that $\mathcal{S}^{-} \neq \emptyset$. There then exists a point $x_0 \in \overline{\mathcal{S}^{-}}$ such that $y_1(x_0) = \min\{ y_1(x) \mid x \in \overline{\mathcal{S}^{-}} \}$. Thanks to \eqref{eq:boundary-inclusion} and the fact that $\Cu_1 = \Cu_2$, one has $x_0 \notin \partial \mathcal{S}^{-}$ and thus $x_0 \in \mathcal{S}^{-}$. Since $\mathcal{S}^{-}$ is open, we have $\nabla y_1(x_0) = 0$, contradicting \eqref{eq:nonvanishing-gradient-Spm}. Hence $\mathcal{S}^{-} = \emptyset$. Similarly, one has $\mathcal{S}^{+} = \emptyset$.
    \noindent\emph{Ad \ref{item:nonintersection}:}
    Without loss of generality, we assume now that $\Omega^1 := \{y_1 <t\} \cap \mathcal{O}$ is surrounded by $\Cu_1$. We then have
    \[
        y_1(x) > t \quad \text{for all } x \in \mathcal{O} \backslash (\Omega^1 \cup \Cu_1).
    \]
    Since $\Cu_1 \cap \Cu_2 = \emptyset$, there are two possibilities:
    \begin{enumerate}
        \item[(I)] \label{item:caseI} one of $\Cu_1$ and $\Cu_2$ is not surrounded by the other; or
        \item[(II)] \label{item:caseII} one of $\Cu_1$ and $\Cu_2$ is surrounded by the other.
    \end{enumerate}
    $\bullet$ For case (I): If $\{y_2 < t\} \cap \mathcal{O}$ is surrounded by $\Cu_2$, then $\mathcal{S}^{-} = \Omega^1$. An argument as in the proof of \ref{item:coincidence} thus gives a contradiction. Otherwise, if $\{ y_2 >t \} \cap \mathcal{O}$ is surrounded by $\Cu_2$, then $\mathcal{S}^{+} = \mathcal{O}\backslash (\overline{\Omega^1} \cup \overline{\{ y_2 >t \} \cap \mathcal{O}})$. Therefore, one has $\Cu_1 \cup \Cu_2 \subsetneq \partial \mathcal{S}^{+}$, which contradicts \eqref{eq:boundary-inclusion}. Thus case (I) is impossible.\\
    $\bullet$ For case (II): We only consider the case where $\Cu_1$ is surrounded by $\Cu_2$. The argument for the other case is similar.
    If $\{y_2 > t\} \cap \mathcal{O}$ is surrounded by $\Cu_2$, then $\mathcal{S}^{-} = \Omega^1$, and we thus have a contradiction by using arguments analogous to ones in the proof of \ref{item:coincidence}. Otherwise, if $\{y_2 < t\} \cap \mathcal{O}$ is surrounded by $\Cu_2$, then $\mathcal{S}^{-} = \emptyset$, and the boundary of $\mathcal{S}^{+}$ consists of $\Cu_1$ and $\Cu_2$.

    \noindent\emph{Ad \ref{item:intersection}:} We first consider the case where $\mathcal{S}^+ \neq \emptyset$.
    Thanks to \cref{prop:countable-decomposition-openset}, the set $\mathcal{{S}}^+$ decomposes into at most countably many disjoint open components.
    There then exist an at most countable index set $I^+$ and disjoint open connected components $\mathcal{S}^+_i$, $i \in I^+$, of $\mathcal{S}^+$ satisfying
    \[
        \mathcal{S}^+ = \cup \{\mathcal{S}^+_i \mid i \in I^+ \}.
    \]
    Moreover, one has $\partial \mathcal{S}^+ = \cup \{ \partial\mathcal{S}^+_i \mid i \in I^+ \} \subset \Cu_1 \cup \Cu_2$. Besides, for any $i \in I^+$, the boundary $\partial \mathcal{S}^+_i$ consists of two arcs $\wideparen{M^{+}_iy_1N^{+}_i}$ and $\wideparen{M^{+}_iy_2N^{+}_i}$ lying in $\Cu_1$ and $\Cu_2$, respectively, for some end points $M_i^+, N_i^+ \in \Cu_1 \cap \Cu_2$. Similarly, if $\mathcal{S}^- \neq \emptyset$, then it is an union of at most countably many disjoint open connected components $\mathcal{S}^-_m$, $m \in I^-$, of $\mathcal{S}^-$ satisfying that
    the boundary $\partial \mathcal{S}_m^{-}$ consists of two arcs $\wideparen{M^{-}_my_1N^{-}_m}$ and $\wideparen{M^{-}_my_2N^{-}_m}$ lying in $\Cu_1$ and $\Cu_2$, respectively, for some end points $M_m^-, N_m^- \in \Cu_1 \cap \Cu_2$. Define the set $\mathcal{A}$ of closed arcs with end points in $\mathcal P := \{M^+_i, N^+_i, M^{-}_m, N^{-}_m: i \in I^+, m \in I^{-} \}$ via
    \[
        \mathcal{A} := \{ \wideparen{MN} \subset \Cu_1 \cap \Cu_2 \mid M, N \in \mathcal P, M \neq N, \wideparen{MN} \cap (\partial \mathcal{S}^+ \cup \partial \mathcal{S}^-) = \{M, N\} \}.
    \]
    Obviously, $\mathcal{A}$ consists of at most countable arcs. By renaming the end points of arcs in $\mathcal{A}$, we obtain an at most countable index set $I^0$ and points $M^0_k, N^0_k \in \mathcal{P}$, $k \in I^0$, that satisfy \eqref{eq:boundary-points}.
\end{proof}

We now prove Green's first identity over nonempty connected components of an open set determined by two functions with nonvanishing gradients.
\begin{lemma}
    \label{lem:Green-identity}
    Assume that $y_1, y_2 \in C^1(\overline\Omega)$ and $t \in \R$ such that $\{y_1 > t > y_2 \} \neq \emptyset$. Let
    $\mathcal{S}$ be a nonempty open connected component of $\{y_1 > t > y_2\}$ satisfying $\partial \mathcal{S} \cap \partial \Omega = \emptyset$. Assume further that there is a positive constant $\alpha$ satisfying
    \begin{equation}
        \label{eq:nonvanishing-gradient}
        \nabla y_1(x) \cdot \nabla y_2(x) \geq \alpha >0 \quad \text{for all } x \in \overline{\mathcal{S}}.
    \end{equation}
    Let $v \in H^1(\Omega) \cap C(\overline\Omega)$ and $\varphi \in W^{2,1}(\Omega) \cap C^1(\overline\Omega)$ be arbitrary. Then
    \begin{equation}
        \label{eq:Green-identity}
        \int_{\mathcal{S}} \nabla v \cdot \nabla \varphi dx = - \int_{\mathcal{S}} v \Delta \varphi dx - \int_{\partial \mathcal{S} \cap \{y_1 =t \} } v \nabla \varphi \cdot \frac{\nabla y_1}{|\nabla y_1|} d\mathcal{H}^1(x) + \int_{\partial \mathcal{S} \cap \{y_2 =t \} } v \nabla \varphi \cdot \frac{\nabla y_2}{|\nabla y_2|} d\mathcal{H}^1(x).
    \end{equation}
\end{lemma}
\begin{proof}
    \begin{figure}
        \centering
        \begin{subfigure}[b]{0.495\textwidth}
            \centering
            \resizebox{\textwidth}{!}{%
                \begin{tikzpicture}
                    \draw[thick] (0,0) ellipse (3cm and 1.5cm);
                    \draw[thin] (0,0) ellipse (2.cm and 0.8cm);
                    \node[align=right] at (3.05,1.2) {$\{y_2=t\}$};
                    \node[align=right] at (-1.35,0) {$\{y_1=t\}$};
                    \node[align=right] at (0,1.25) {$\mathcal{S}$};
                    \node[align=right] at (3.8,-0.653834842-0.3) {$\quad \quad\nu_2 = \frac{\nabla y_2}{|\nabla y_2|}$};
                    \node[align=right] at (0.3,-0.8+0.45) {$\nu_1$};
                    \tkzDefPoint(2.7,-0.653834842){C}
                    \tkzDefPoint(2.7+0.6,-0.653834842-0.25){D}
                    \tkzDefPoint(0,-0.8){E}
                    \tkzDefPoint(0,-0.8+0.45){F}
                    \draw[->] (C) to (D);
                    \draw[->] (E) to (F);
                \end{tikzpicture}
            }
            \caption{A component $\mathcal{S}$ when $\{y_1=t\} \cap \{y_2=t\} = \emptyset$}
            \label{fig:component-nocusp}
        \end{subfigure}
        \begin{subfigure}[b]{0.495\textwidth}
            \centering
            \resizebox{\textwidth}{!}{%
                \begin{tikzpicture}[>=latex]
                    \draw[thick] (-1,0) circle (3cm);
                    \draw[thin] (0,0) ellipse (4cm and 1cm);
                    \draw (-1.0,2.0) node[left] {$\mathcal{S}$};
                    \draw (4.5,-0.5) node {$\{y_1 = t\}$};
                    \draw (2.0,-2.3) node {$\{y_2 = t\}$};
                    \draw (-4,0) node[right]{$N$};
                    \draw (0,0.75) node {$\wideparen{My_1N}$};
                    \draw (0.6,2.8) node[right] {$\wideparen{My_2N}$};
                    \draw (-2.8566 + 0.05,0.6) node[right]{$N_1$};
                    \draw (-3.9172,0.7) node[left]{$N_2$};
                    \draw[dashed] (-3.9172,0.7) -- (-2.8566,0.7);
                    \foreach \Point/\PointLabel in {(-4,0)/, (1.8667,0.8844)/M, (-3.9172,0.7)/, (-2.8566,0.7)/}
                    \draw[fill=black] \Point circle (0.05) node[above right] {$\PointLabel$};
                    \draw[->] (-4,0) to (-5,0) node[left] {$\nu_2$};
                \end{tikzpicture}
            }
            \caption{A component $\mathcal{S}$ with a cusp at $N$}
            \label{fig:component-cusp}
        \end{subfigure}
        \caption{A nonempty open component $\mathcal{S}$ of the set $\{y_1 > t > y_2\}$ with the outward normal vectors $\nu_1 =- \frac{\nabla y_1}{|\nabla y_1|}$ on $\{y_1=t\}$ and $\nu_2 = \frac{\nabla y_2}{|\nabla y_2|}$ on $\{y_2=t\}$ where $\{y_1=t\}$ is thin while $\{y_2=t\}$ is thick }
        \label{fig:component-S}
    \end{figure}
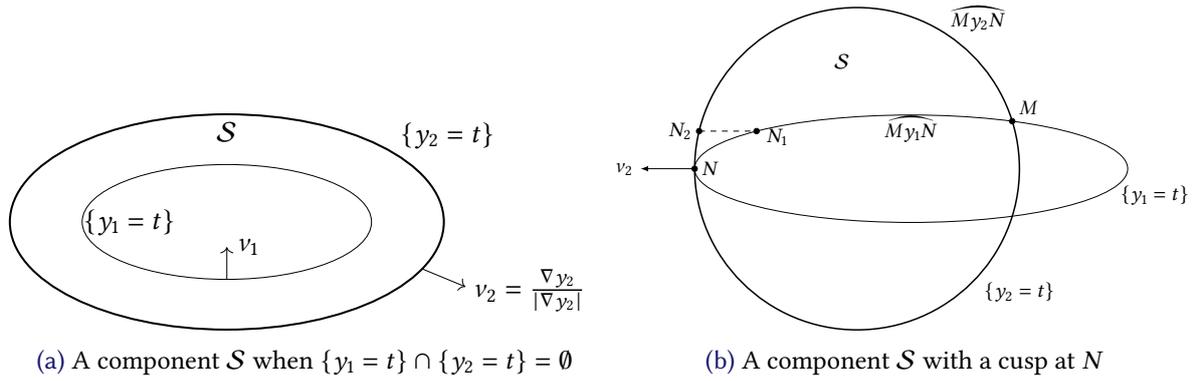
    We first observe that all the integrals in \eqref{eq:Green-identity} are well-defined. Note further that $\partial \mathcal{S} \subset \{y_1 =t\} \cup \{y_2 =t\}$ and that the outward normal vector $\nu(x)$ at point $x \in \partial \mathcal{S}$ is defined as
    \begin{equation}
        \label{eq:outward-normal-vector}
        \nu(x) = \begin{cases}
            -\frac{\nabla y_1(x)}{|\nabla y_1(x)|} & \text{for } x \in \partial \mathcal{S} \cap \{y_1=t\},\\
            \frac{\nabla y_2(x)}{|\nabla y_2(x)|} & \text{for } x \in \partial \mathcal{S} \cap \{y_2=t\}.
        \end{cases}
    \end{equation}
    We now consider the following two cases.\\
    \emph{$\bullet$ Case 1: The boundary $\partial \mathcal{S}$ consists of two closed curves having no common points; see, \cref{fig:component-nocusp}.} In this situation, thanks to \eqref{eq:nonvanishing-gradient} and the fact that $\partial \mathcal{S} \cap \partial \Omega = \emptyset$, we deduce from the Implicit Function Theorem that the domain $\mathcal{S}$ has $C^1$ boundary in the sense of Definition 1.2.1.1 in \cite{Grisvard1985}. Applying the classical Green formula over the domain $\mathcal{S}$ and employing \eqref{eq:outward-normal-vector}, we derive \eqref{eq:Green-identity}.
    \\
    \emph{$\bullet$ Case 2: The boundary $\partial \mathcal{S}$ consists of two arcs $\wideparen{My_1N}$ and $\wideparen{My_2N}$ with end points $M, N \in \{y_1=t\} \cap \{y_2=t\}$ (here $M$ and $N$ might be identical); see, \cref{fig:component-cusp}}. If $\partial \mathcal{S}$ has no cusps, i.e., there are no points where two branches of $\partial \mathcal{S}$ meet and the tangents at which of each branch are coincidental, then the boundary $\partial \mathcal{S}$ is a curvilinear polygon of class $C^1$ in the sense of Definition 1.4.5.1 in \cite{Grisvard1985}, as a result of the combination of the Implicit Function Theorem and \eqref{eq:nonvanishing-gradient}. Similar to Case 1, we also have \eqref{eq:Green-identity}. For the situation where $\partial \mathcal{S}$ has cusps, we shall prove \eqref{eq:Green-identity} via approximating $\mathcal{S}$ by regular domains having boundaries, which are curvilinear polygons of class $C^1$, and via using the fact that $\varphi \in C^1(\overline\Omega)$ and $v \in C(\overline\Omega)$. To this end, without loss of generality, we assume that $\partial \mathcal{S}$ has only one cusp at $N$. For any $\epsilon >0$ small enough,
    by applying \cref{lem:regular-points}, there exist $N_i \in \{y_i = t \} \cap \partial \mathcal{S}$, $i=1,2$, such that
    \[
        l_{\wideparen{NN_1}} + l_{\wideparen{NN_2}} + N_1N_2 < C\epsilon
    \]
    for some positive constant $C$ independent of $\epsilon$,
    where $l_{\wideparen{NN_i}}$ denotes the length of the arc $\wideparen{NN_i}$.
    Moreover, $N_1$ and $N_2$ are not cusps of $\partial\mathcal{S}_\epsilon$,
    where $\mathcal{S}_\epsilon$ denotes a subdomain of $\mathcal{S}$, whose boundary consists of the arcs $\wideparen{My_1N_1}, \wideparen{My_2N_2}$ and the line segment $N_1N_2$.
    Now applying Green's first formula on the domain $\mathcal{S}_\epsilon$ then implies that
    \begin{equation*}
        \int_{\mathcal{S}_\epsilon} \nabla v \cdot \nabla \varphi dx = - \int_{\mathcal{S}_\epsilon} v \Delta \varphi dx + \int_{\partial \mathcal{S}_\epsilon } v \nabla \varphi \cdot \nu_\epsilon d\mathcal{H}^1(x),
    \end{equation*}
    where $\nu_\epsilon$ stands for the outward normal vector on $\partial \mathcal{S}_\epsilon$. Since $\nu_\epsilon = \nu$ on $\partial \mathcal{S} \cap \partial \mathcal{S}_\epsilon$ defined in \eqref{eq:outward-normal-vector}, we can rewrite the above identity as
    \begin{multline}
        \label{eq:Green-identity-app}
        \int_{\mathcal{S}_\epsilon} \nabla v \cdot \nabla \varphi dx = - \int_{\mathcal{S}_\epsilon} v \Delta \varphi dx - \int_{\partial \mathcal{S} \cap \{y_1 =t \} } v \nabla \varphi \cdot \frac{\nabla y_1}{|\nabla y_1|} d\mathcal{H}^1(x) + \int_{\partial \mathcal{S} \cap \{y_2 =t \} } v \nabla \varphi \cdot \frac{\nabla y_2}{|\nabla y_2|} d\mathcal{H}^1(x) \\
        \begin{aligned}
            & + \int_{ \wideparen{N_1y_1N} } v \nabla \varphi \cdot \frac{\nabla y_1}{|\nabla y_1|} d\mathcal{H}^1(x) - \int_{ \wideparen{N_2y_2N} } v \nabla \varphi \cdot \frac{\nabla y_2}{|\nabla y_2|}d\mathcal{H}^1(x) + \int_{N_1N_2} v \nabla \varphi \cdot \nu_\epsilon d\mathcal{H}^1(x).
        \end{aligned}
    \end{multline}
    Using the continuity over $\overline\Omega$ of $v$ and $\nabla \varphi$, it follows from the choice of $N_1$ and $N_2$ that
    \[
        \left|\int_{ \wideparen{N_1y_1N} } v \nabla \varphi \cdot \frac{\nabla y_1}{|\nabla y_1|} d\mathcal{H}^1(x) - \int_{ \wideparen{N_2y_2N} } v \nabla \varphi \cdot \frac{\nabla y_2}{|\nabla y_2|} d\mathcal{H}^1(x)+ \int_{N_1N_2} v \nabla \varphi \cdot \nu_\epsilon d\mathcal{H}^1(x) \right| \leq C_1 \epsilon
    \]
    for some positive constant $C_1$. By letting $\epsilon \to 0^+$ in \eqref{eq:Green-identity-app}, we thus deduce \eqref{eq:Green-identity} from Lebesgue's Dominated Convergence Theorem.
\end{proof}

\begin{remark}
    \label{rem:Green-identity-emptyset}
    Since $\partial \emptyset = \emptyset$,

    it is obvious to have the identity \eqref{eq:Green-identity} when $\mathcal{S} = \emptyset$.
\end{remark}

The following version of Green's first identity over an open set determined via functions with nonvanishing gradients is a consequence of \cref{lem:decomposition-domains,lem:Green-identity}.
\begin{proposition}
    \label{prop:Green-identity-general}
    Let $y_1, y_2 \in C^1(\overline\Omega)$, let $t \in \R$ be given, and let $\mathcal{O}$ be an open subset in $\overline\Omega$ such that
    \begin{equation}
        \label{eq:nonvanishing-gradient-2}
        \nabla y_1(x) \cdot \nabla y_2(x) \geq \alpha >0 \quad \text{for all } x\in \mathcal{O}
    \end{equation}
    for some constant $\alpha$.
    Define the open sets $\mathcal{S}^+$ and $\mathcal{S}^{-}$ by
    \[
        \mathcal{S}^{-} := \{ y_1 < t < y_2 \} \cap \mathcal{O} \quad \text{and} \quad \mathcal{S}^{+} := \{ y_1 > t > y_2 \} \cap \mathcal{O}.
    \]
    Assume that the level set $\{y_j=t \}$ has one and only one closed $C^1$ simple curve $\mathcal{C}_j$ in $\mathcal{O}$ for $j=1,2$. Assume further that $\Cu_j \cap \partial \Omega = \emptyset$, $j=1,2$,
    and that $\partial \mathcal{S}^{\pm} \subset \Cu_1 \cup \Cu_2$.
    Let $v \in H^1(\Omega) \cap C(\overline\Omega)$ and $\varphi \in W^{2,1}(\Omega) \cap C^1(\overline\Omega)$ be arbitrary. Then the following identity holds
    \begin{equation}
        \label{eq:Green-identity-general}
        \int_{\Omega} \left(\1_{\mathcal{S}^+} - \1_{\mathcal{S}^{-}} \right) \nabla v \cdot \nabla \varphi dx = - \int_{\Omega} \left(\1_{\mathcal{S}^+} - \1_{\mathcal{S}^{-}} \right) v \Delta \varphi dx - \int_{\Cu_1} v \nabla \varphi \cdot \frac{\nabla y_1}{|\nabla y_1|} d\mathcal{H}^1(x) + \int_{\Cu_2} v \nabla \varphi \cdot \frac{\nabla y_2}{|\nabla y_2|} d\mathcal{H}^1(x).
    \end{equation}
\end{proposition}
\begin{proof}
    For the case where $\Cu_1 = \Cu_2$, we have $\mathcal{S}^+ = \mathcal{S}^{-} = \emptyset$ and thus $\1_{\mathcal{S}^+} = \1_{\mathcal{S}^{-}} = 0$. We then obtain \eqref{eq:Green-identity-general}.
    For the case where $\Cu_1 \cap \Cu_2 = \emptyset$, one of the sets $\mathcal{S}^+$ and $\mathcal{S}^{-}$ is empty and the boundary of the other consists of $\Cu_1$ and $\Cu_2$. We then derive \eqref{eq:Green-identity-general} from the classical Green's first formula.
    It remains to consider the situation where $\Cu_1 \cap \Cu_2 \neq \emptyset$; see \cref{fig:components-decomposition}. As a result of \cref{lem:decomposition-domains}, there exist at most countable index sets $I^+, I^-$, $I^0$, points $M_i^+, N_i^+, M_m^-,N_m^-,M_k^0,N_k^0$, and open connected domains $\mathcal{S}^+_i, \mathcal{S}^{-}_m$ with $i \in I^+, m \in I^-, k \in I^0$ satisfying assertion \ref{item:intersection} in \cref{lem:decomposition-domains}. (Note that some of the sets $I^+, I^-, I^0$ might be empty.) For any $i \in I^+$, applying \cref{lem:Green-identity} and \cref{rem:Green-identity-emptyset} to open connected domain $\mathcal{S}_i^+$ yields
    \begin{equation*}
        \int_{\mathcal{S}^+_i} \nabla v \cdot \nabla \varphi dx = - \int_{\mathcal{S}^+_i} v \Delta \varphi dx - \int_{\partial \mathcal{S}_i^+ \cap \Cu_1 } v \nabla \varphi \cdot \frac{\nabla y_1}{|\nabla y_1|} d\mathcal{H}^1(x) + \int_{\partial \mathcal{S}_i^+ \cap \Cu_2 } v \nabla \varphi \cdot \frac{\nabla y_2}{|\nabla y_2|} d\mathcal{H}^1(x).
    \end{equation*}
    Summing up the above identities over all $i \in I^+$ gives
    \begin{equation}
        \label{eq:Green-identity-S-plus}
        \int_{\mathcal{S}^+} \nabla v \cdot \nabla \varphi dx = - \int_{\mathcal{S}^+} v \Delta \varphi dx - \int_{\partial \mathcal{S}^+ \cap \Cu_1 } v \nabla \varphi \cdot \frac{\nabla y_1}{|\nabla y_1|} d\mathcal{H}^1(x) + \int_{\partial \mathcal{S}^+ \cap \Cu_2 } v \nabla \varphi \cdot \frac{\nabla y_2}{|\nabla y_2|} d\mathcal{H}^1(x).
    \end{equation}
    Similarly, there holds
    \begin{equation}
        \label{eq:Green-identity-S-minus}
        \int_{\mathcal{S}^-} \nabla v \cdot \nabla \varphi dx = - \int_{\mathcal{S}^-} v \Delta \varphi dx + \int_{\partial \mathcal{S}^- \cap \Cu_1 } v \nabla \varphi\cdot \frac{\nabla y_1}{|\nabla y_1|} d\mathcal{H}^1(x) - \int_{\partial \mathcal{S}^- \cap \Cu_2 } v \nabla \varphi \cdot \frac{\nabla y_2}{|\nabla y_2|} d\mathcal{H}^1(x).
    \end{equation}
    By subtracting \eqref{eq:Green-identity-S-minus} from \eqref{eq:Green-identity-S-plus}, one has
    \begin{multline}
        \label{eq:Green-identity-lack-boundary}
        \int_{\Omega} \left(\1_{\mathcal{S}^+} - \1_{\mathcal{S}^{-}} \right) \nabla v \cdot \nabla \varphi dx = - \int_{\Omega} \left(\1_{\mathcal{S}^+} - \1_{\mathcal{S}^{-}} \right) v \Delta \varphi dx \\
        - \int_{(\partial \mathcal{S}^+ \cup \partial \mathcal{S}^- ) \cap \Cu_1} v \nabla \varphi \cdot \frac{\nabla y_1}{|\nabla y_1|} d\mathcal{H}^1(x) + \int_{(\partial \mathcal{S}^+ \cup \partial \mathcal{S}^- ) \cap \Cu_2} v \nabla \varphi \cdot \frac{\nabla y_2}{|\nabla y_2|} d\mathcal{H}^1(x).
    \end{multline}
    On the other hand, due to \eqref{eq:boundary-points}, it follows that
    \begin{equation}
        \label{eq:boundary-decom}
        \Cu_j = \left( {(\partial \mathcal{S}^+ \cup \partial \mathcal{S}^- ) \cap \Cu_j}\right) \cup \{ \wideparen{M^{0}_kN^{0}_k} \mid k \in I^0 \} \quad \text{for } j =1,2
    \end{equation}
    with $\wideparen{M^{0}_kN^{0}_k}$ denoting the arcs lying on $\Cu_1 \cap \Cu_2$ with end points $M^0_k$ and $N_k^0$. Obviously, one has
    \[
        \frac{\nabla y_1(x)}{|\nabla y_1(x)|} = \frac{\nabla y_2(x)}{|\nabla y_2(x)|} \quad \text{for all } x \in \wideparen{M^{0}_kN^{0}_k} \, \text{and } k \in I^0.
    \]
    Combing this with \eqref{eq:Green-identity-lack-boundary} and \eqref{eq:boundary-decom} yields \eqref{eq:Green-identity-general}.
\end{proof}

\subsection{Continuity of level sets in term of functions}
\label{sec:continuity-levelset}

This subsection is devoted to the continuity of level sets in term of functions determining these level sets. Namely, we will show that when functions $y_n$ tend to $y$ in $C^1(\overline{\Omega})$, then, in any $\epsilon$-neighborhood (with $\epsilon$ small enough) of an arbitrary connected component of the level set $\{y=t\}$ on which the gradient of $y$ does not vanish, there exists one and only one connected component of $\{y_n=t\}$ for each $n$ sufficiently large, see \cref{prop:levelset-unique-component-continuity} and \cref{cor:levelset-continuity-in-functions} below.

From now on, for any $\epsilon>0$ and any set $V \subset \overline \Omega$, we denote by $V^{\epsilon}$ the open $\epsilon$-neighborhood in $\overline \Omega$ of $V$, that is,
\[
    V^{\epsilon} := \{x \in \overline \Omega \mid \, \text{dist}(x,V) < \epsilon \},
\]
where dist$(x,V)$ is the distance from $x$ to $V$. If $\Cu$ is a closed simple curve in $\overline \Omega$, we define the open sets $\Cu_{\epsilon}^+$ and $\Cu_{\epsilon}^{-}$ (illustrated in \cref{fig:epsilon-neighborhood}) as follows
\[
    \Cu_{\epsilon}^{-} := \{ x \in \overline\Omega \mid \, 0 < \text{dist}(x,\Cu) < \epsilon \, \text{and $x$ is surrounded by } \Cu \} \quad \text{and} \quad \Cu_{\epsilon}^{+} :=\Cu^{\epsilon} \backslash (\Cu_{\epsilon}^{-} \cup \Cu).
\]
\begin{figure}[t]
    \centering
    \resizebox{4.75cm}{2.25cm}{%
        \begin{tikzpicture}[>=latex]
            \draw[dashed] (0,0) circle (2.2cm);
            \draw (1.5,0) node[right] {$\mathcal{C}_\epsilon^{+}$};
            \draw (0,0) circle (1.5cm);
            \draw (1.5,0) node[left] {$\mathcal{C}_\epsilon^{-}$};
            \draw (-1.5,0) node {$\mathcal{C}$};
            \draw[dashed] (0,0) circle (0.8cm);
            \draw[<->] (0,1.5) -- (0,2.2) node[right,midway] {$\epsilon$};
        \end{tikzpicture}
    }
    \caption{sets $\Cu_{\epsilon}^{+}$ and $\Cu_{\epsilon}^{-}$}
    \label{fig:epsilon-neighborhood}
\end{figure}
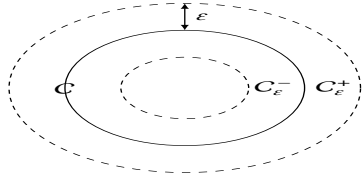

In the remainder of this section, we shall consider the following general situation.
\begin{hypothesis}
\item \label{hypo:levelsets}
    Let $y_n, y \in C^1(\overline\Omega)$, $n \geq 1$, such that $y_n = y = 0$ on the boundary $\partial\Omega$ and
    \begin{equation}
        \label{eq:levelsets-hypothesis}
        y_n \to y \quad \text{strongly in } C^1(\overline\Omega).
    \end{equation}
\end{hypothesis}
We have the following result on strong positivity of the product of gradients of $y_n$ and $y$ around an $\epsilon$-neighborhood of any connected component in the level set $\{y=t\}$.

\begin{proposition}
    \label{prop:strong-positivity-gradients}
    Let functions $y_n, y$, $n \geq 1$, satisfy hypothesis \ref{hypo:levelsets}. Assume that $\Cu$ is a connected component of $\{y=t\}$ for some $t \in \R$ and fulfills the condition \eqref{eq:nonvanishing-gradient-on-component}. Then there exist constants $\epsilon_0 >0$, $\alpha_\Cu >0$, and $n_0 \in \N$ such that
    \begin{equation}
        \label{eq:nonvanishing-gradient-ball}
        \nabla y(x) \cdot \nabla y_{n}(x') \geq \alpha_\Cu \quad \text{for all } n \geq n_0, x, x' \in \Cu^{\epsilon_0} \, \text{with } |x - x'| \leq 2\epsilon_0.
    \end{equation}
\end{proposition}
\begin{proof}
    By contradiction, assume that there exist sequences $\epsilon_k \to 0^+$, $n_k \to \infty$ satisfying
    \begin{equation}
        \label{eq:nonvanishing-gradient-esti}
        x_k, x_k' \in \Cu^{\epsilon_k}, \quad |x_k - x_k'| \leq 2\epsilon_k \quad \text{and} \quad \nabla y(x_k) \cdot \nabla y_{n_k}(x'_k) < \frac{1}{k} \quad \text{for all } k \geq 1.
    \end{equation}
    By extracting a subsequence if necessary, we conclude from the boundedness of $\{x_k\}$ and $\{x_k'\}$ and the first two conditions in \eqref{eq:nonvanishing-gradient-esti} as well as from the continuity of $y$ that $x_k \to x$ and $x_k' \to x$ for some $x \in \Cu$. Passing to the limit the last condition in \eqref{eq:nonvanishing-gradient-esti} and using \eqref{eq:levelsets-hypothesis} yield
    \[
        |\nabla y(x)|^2 = 0,
    \]
    which contradicts \eqref{eq:nonvanishing-gradient-on-component}. The proposition is proven.
\end{proof}

\begin{proposition}
    \label{prop:levelset-unique-component-continuity}
    Let functions $y_n, y$, $n \geq 1$, satisfy hypothesis \ref{hypo:levelsets}. Let $\Cu$ be a connected component of $\{y=t\}$, $t \in \R$ that satisfies \eqref{eq:nonvanishing-gradient-on-component}. Then a constant $\epsilon_0 >0$ exists such that for any $\epsilon \in (0,\epsilon_0)$, there is a positive integer $n_0 = n_0(\epsilon)$ such that $\Cu^\epsilon$ contains one and only one connected component $\Cu_n$ of $\{y_n=t\}$ for all $n \geq n_0$.
\end{proposition}
\begin{proof}
    Since $\Cu$ and $\{y=t\} \backslash \Cu$ are both closed and disjoint, a constant $\epsilon_0 >0$ thus exists and fulfills
    \begin{equation}
        \label{eq:epsilon-zero-choice}
        \Cu^{\epsilon_0} \cap (\{y=t\} \backslash \Cu)^{\epsilon_0} = \emptyset.
    \end{equation}
    Obviously, we can assume that the constant $\epsilon_0$ in \eqref{eq:epsilon-zero-choice} is identical to the one in \cref{prop:strong-positivity-gradients}.
    Fix $\epsilon \in (0, \epsilon_0)$ and, without loss of generality, assume that
    \begin{equation}
        \label{eq:y-positive-negative-levelset}
        y < t \quad \text{on } \Cu^-_{\epsilon_0} \quad \text{and} \quad y > t \quad \text{on } \Cu^+_{\epsilon_0}.
    \end{equation}
    We now consider two possible situations: $\Cu \cap \partial \Omega \neq \emptyset$ and $\Cu \cap \partial \Omega = \emptyset$.

    $\bullet$ For the case $\Cu \cap \partial \Omega \neq \emptyset$, we have $t=0$ and $\Cu = \partial\Omega$. Moreover, there hold $\Cu^+_{\epsilon_0} = \emptyset$ and $\partial \Omega \subset \{y_n =0\}$. Thanks to \eqref{eq:nonvanishing-gradient-ball} and \cref{prop:decomposition-levelset}, $\Cu_n := \partial\Omega$ is a connected component of $\{y_n=0\}$ for all $n \geq 1$, and is a closed simple curve in $\Cu^\epsilon$ for all $\epsilon \in (0, \epsilon_0)$. If there exists another connected component $\Cu_n'$ of $\{y_n =0\}$ such that $\Cu_n' \cap \Cu^\epsilon \neq \emptyset$ for $n$ sufficient large,
    then by picking a point $x_2 = (x_2^1, x_2^2)\in \Cu_n' \cap \Cu^\epsilon$ and taking
    \[
        x_1 \in \mathrm{argmin}\{|x_2 - x| : x \in \Cu \},
    \]
    one has $|x_1 - x_2| \leq \epsilon$ and
    \begin{equation}
        \label{eq:x12-normal-boundary-direction}
        x_2 - x_1 = \kappa_n \nabla y(x_1) \quad \text{for some } \kappa_n \in \R,
    \end{equation}
    as a result of \cref{lem:x1x2-existence}. Since $\Cu_n \cap \Cu_n' = \emptyset$ and $\Cu_n = \partial\Omega = \Cu$, we obtain $x_1 \neq x_2$ and thus
    \begin{equation}
        \label{eq:kappa-n-nonzero}
        \kappa_n \neq 0.
    \end{equation}
    On the other hand, since $x_1 \in \Cu = \Cu_n$ and $x_2 \in \Cu_n'$, $y_n(x_1) = t = y_n(x_2)$.
    We thus deduce from the Mean Value Theorem that
    $\nabla y_n(x_1 + \theta_n(x_2-x_1)) \cdot (x_2 - x_1) = 0$ with $\theta_n \in (0,1)$. Combining this with \eqref{eq:x12-normal-boundary-direction} and \eqref{eq:kappa-n-nonzero} yields
    \[
        \nabla y_n(x_1 + \theta_n(x_2-x_1)) \cdot \nabla y(x_1) = 0,
    \]
    which contradicts \eqref{eq:nonvanishing-gradient-ball} for $n$ large enough.

    \medskip
    $\bullet$ For the case $\Cu \cap \partial \Omega = \emptyset$,
    we have $\Cu_{\epsilon}^{\pm} \neq \emptyset$.
    We now split the proof into several Claims below.\\
    \noindent\emph{Claim 1: There exists an integer $n_0^1$ such that $\{y_n =t\} \cap \Cu^\epsilon \neq \emptyset$ for all $n \geq n_0^1$}.
    In fact, arguing by contradiction, there exists a subsequence of $\{n\}$, denoted in the same way, such that $\{y_n =t\} \cap \Cu^\epsilon = \emptyset$ for all $n\geq 1$. Splitting the sequence $\{n\}$ into subsequences, also denoted by $\{n\}$, that satisfy one of the following conditions for all $n \geq 1$:
    \begin{equation} \label{eq:impossible}
        \Cu^\epsilon \subset \{ y_n >t\} \quad \text{or} \quad
        \Cu^\epsilon \subset \{ y_n <t\}.
    \end{equation}
    For the first case in \eqref{eq:impossible}, we have
    \[
        \Cu_\epsilon^{-} \subset \{y < t < y_n \} \subset \{ 0 < | y - t | \leq \tau_n \}
    \]
    with $\tau_n := \norm{y_n - y}_{C(\overline\Omega)}$. This implies that $\meas_{\R^2}(\Cu_\epsilon^{-}) \leq \meas_{\R^2}(\{ 0 < | y - t | \leq \tau_n \}) \to 0$ as $n\to \infty$, contradicting the fact that $\meas_{\R^2}(\Cu_\epsilon^{-})>0$. The first case in \eqref{eq:impossible} is then impossible. Similarly, the second one is also absurd.\\

    \noindent\emph{Claim 2: There exists an integer $n_0^2$ such that $\{y_n =t\} \cap \partial\Cu^\epsilon = \emptyset$ for all $n \geq n_0^2$}.
    Arguing by contradiction, there are subsequences $n_k \to \infty$ and $\{x_{k}\}$ with $x_k \in \{y_{n_k} = t \} \cap \partial \Cu^\epsilon$ for all $k \geq 1$.
    From the boundedness of $\{x_k\}$ and the closedness of $\partial \Cu^\epsilon$, we can assume that $x_k \to x \in \partial \Cu^\epsilon$ as $k \to \infty$. Thanks to \eqref{eq:levelsets-hypothesis}, there holds $x \in \{ y = t\}$.
    We then have $\{y=t\} \cap \partial \Cu^\epsilon \neq \emptyset$, which contradicts \eqref{eq:epsilon-zero-choice}.

    \medskip

    From Claims 1 and 2, there exists an integer $n_0$ such that $\{y_n=t\}$ admits a connected component $\Cu_n$ being in $\Cu^\epsilon$ for all $n \geq n_0$. \\

    For any $n \geq n_0$, we define a subset (depending on $n$) of $\Cu$ by setting
    \begin{equation}
        \label{eq:Lambda-set}
        \Lambda_n := \left \{ x_1 \in \Cu \mid \exists x_2 \in \{ y_n = t \} \cap \Cu^\epsilon \, \text{satisfying \eqref{eq:x1x2-distance}--\eqref{eq:x1x2-direction}} \right \},
    \end{equation}
    where
    \begin{equation}
        \label{eq:x1x2-distance}
        |x_1 -x_2| \leq \epsilon
    \end{equation}
    and
    \begin{equation}
        \label{eq:x1x2-direction}
        x_2-x_1 = \kappa \nabla y(x_1), \quad \text{for some } \kappa \in \R.
    \end{equation}
    In view of \cref{lem:x1x2-existence}, we have $\Lambda_n \neq \emptyset$.
    It remains to prove the following claims.

    \medskip
    \noindent\emph{Claim 3: For any $n \geq n_0$ and any $x_1 \in \Cu$, there exists at most one point $x_2 \in \{ y_n = t \} \cap \Cu^\epsilon$ such that $x_1$ and $x_2$ satisfy \eqref{eq:x1x2-distance} and \eqref{eq:x1x2-direction}.}
    Suppose this is not the case, i.e., suppose that there were a point $x_1 \in \Cu$ and two different points $p_2, p_2' \in \{ y_n = t \} \cap \Cu^\epsilon$ satisfying \eqref{eq:x1x2-distance} and \eqref{eq:x1x2-direction} in place of $x_2$, respectively, for some $\kappa:= \kappa$ and $\kappa:= \kappa'$. The relation \eqref{eq:x1x2-distance} for both $x_2:= p_2$ and $x_2:= p_2'$ implies that $|p_2 - p_2'| \leq 2\epsilon$.
    Similarly, \eqref{eq:x1x2-direction} indicates that
    \[
        p_2 - p_2' = (\kappa - \kappa') \nabla y(x_1),
    \]
    which, together with the condition that $p_2 \neq p_2'$, yields $(\kappa - \kappa') \neq 0$.
    Since $p_2, p_2' \in \{ y_n = t \} \cap \Cu^\epsilon$, one has
    \[
        y_n(p_2) = y_n(p_2') = t.
    \]
    By the Mean Value Theorem, there holds
    \[
        \nabla y_n(p_2' + \theta(p_2-p_2')) \cdot (p_2 - p_2') = 0
    \]
    for some constant $\theta \in (0,1)$. Combing this with the fact that $(p_2 -p_2') = (\kappa - \kappa') \nabla y(x_1)$ and that $(\kappa - \kappa') \neq 0$ yields $ \nabla y_n(p_2' + \theta(p_2-p_2')) \cdot \nabla y(x_1) =0$, contradicting \eqref{eq:nonvanishing-gradient-ball}.

    \medskip
    \noindent\emph{Claim 4: There exists an integer $\bar n \geq n_0$ such that $\Lambda_n = \Cu$ for all $n \geq \bar n$.}
    Arguing by contradiction, assume that there exist a subsequence $\{k\}$ of $\{n\}$ and a sequence of points $\{x_{1,k}\} \subset \Cu$ such that,
    for each $k \geq 1$, there is no points $x_{2,k}$ satisfying both \eqref{eq:x1x2-distance} and \eqref{eq:x1x2-direction} corresponding to $x_1 := x_{1,k}$ and $x_2 :=x_{2,k}$. This means that
    the line $d_k$ through $x_{1,k}$ with direction vector $\nabla y(x_{1,k})$ does not intersect $\{ y_{n_k} = t \} \cap B_{\R^2}(x_{1,k}, \epsilon)$ for all $k \geq 1$. By Claim 2, we have $\{ y_{n_k} = t \} \cap \partial \Cu^\epsilon = \emptyset$, and thus the sign of $(y_{n_k}-t)$ does not change over the set $A_kB_k := d_k \cap \overline B_{\R^2}(x_{1,k}, \epsilon)$ with $A_k \in \partial \Cu_{\epsilon}^{+}$ and $B_k \in \partial \Cu_{\epsilon}^{-}$. We now split the sequence $\{k\}$ into subsequences, also denoted by $\{k\}$, that satisfy one of the following conditions for all $k \geq 1$:
    \begin{equation}
        \label{eq:sign-yn-not-change}
        y_{n_k} - t > 0 \quad \text{over } A_kB_k \quad \text{or} \quad y_{n_k} - t < 0 \quad \text{over } A_kB_k.
    \end{equation}
    We now consider the first case in \eqref{eq:sign-yn-not-change} only, since the other can be treated similarly. For all $k \geq 1$, according to \eqref{eq:y-positive-negative-levelset}, one has
    \begin{align*}
        y_{n_k}(B_k) - y(B_k) & > t - y(B_k) > t - \sup \{y(x) \mid x \in (\partial \Cu_\epsilon^{-}) \backslash \Cu \} > m_0 > 0
    \end{align*}
    for some constant $m_0$ independent of $k$. This contradicts the limit $y_{n_k} \to y$ in $C^1(\overline\Omega)$ as $k \to \infty$.

    \medskip

    From Claims 3 and 4, we can define for all $n \geq \bar n$ the following map (depending on $n$)
    \[
        \begin{aligned}
            T: \Cu & \to \{y_n =t \} \cap \Cu^\epsilon\\
            x_1 & \mapsto x_2,
        \end{aligned}
    \]
    with $x_1$ and $x_2$ satisfying \eqref{eq:x1x2-distance}--\eqref{eq:x1x2-direction}. Moreover, by \cref{lem:x1x2-existence}, $T$ is surjective.

    \medskip

    \noindent\emph{Claim 5: For all $n \geq \bar n$, $T$ is continuous.}
    In fact, assume that $x_{1,k} \to x_1$ as $k \to \infty$ with $x_{1,k}, x_1 \in \Cu$ and $x_{2,k} = T(x_{1,k})$, $ k \geq 1$. Then there exists $\{\kappa_k\} \subset \R$ such that
    \[
        |x_{1,k} - x_{2,k}| \leq \epsilon \quad \text{and} \quad x_{2,k} - x_{1,k} = \kappa_k \nabla y(x_{1,k})
    \]
    for all $k \geq 1$. Therefore $\{x_{2,k}\}$ is bounded and $\{\kappa_k\}$ is also bounded due to \eqref{eq:nonvanishing-gradient-on-component}.
    Let $x_2$ and $\kappa$ be an arbitrary limit point of $\{x_{2,k}\}$ and $\{\kappa_k\}$, respectively. We easily have
    \[
        |x_1 - x_2| \leq \epsilon \quad \text{and} \quad x_{2} - x_{1} = \kappa \nabla y(x_{1}).
    \]
    Moreover, the closedness of $\{y_n =t\}$ yields $x_2 \in \{y_n =t\} \cap \Cu^\epsilon$. Combing this with Claim 3 yields $x_2 = T(x_1)$.
    Consequently, the full sequence $\{x_{2,k}\}$ converges to $x_2$ and $T$ is then continuous.

    \medskip

    For any $n \geq \bar n$, we have $T(\Cu) = \{y_n =t \} \cap \Cu^\epsilon$ since $T$ is surjective. Moreover, because $\Cu$ is connected and compact, and $T$ is continuous, $T(\Cu)$ is also connected and compact; see, e.g. Theorem 4.22 in \cite{Rudin}. Therefore $\{y_n =t \} \cap \Cu^\epsilon$ is connected. From this and Claim 2, we can conclude that $\{y_n =t \} \cap \Cu^\epsilon$ is the unique connected component of $\{y_n=t\}$ in $\Cu^\epsilon$. The proof is complete.
\end{proof}

The following result is a direct consequence of \cref{prop:decomposition-levelset,prop:levelset-unique-component-continuity}.
\begin{corollary}
    \label{cor:levelset-continuity-in-functions}
    Let functions $y_n, y$, $n \geq 1$, satisfy hypothesis \ref{hypo:levelsets}. Assume that $y$, together with some value $t \in \R$, satisfies \eqref{eq:nonvanishing-gradient-on-levelset}.
    Then
    \[
        \{y=t\} = \cup_{i=1}^m \Cu_i,
    \]
    for some integer $m \geq 1$,
    where $\Cu_i$, $1 \leq i \leq m$, are $C^1$ closed simple curves. Moreover, a constant $\epsilon_0 >0$ exists such that for any $\epsilon \in (0,\epsilon_0)$, there is a positive integer $n_0 = n_0(\epsilon)$ with
    \[
        \{y_n = t\} = \cup_{i=1}^m \Cu_{i,n} \quad \text{and} \quad \Cu_{i,n} \subset \Cu_i^\epsilon \quad \text{for all } n \geq n_0, 1 \leq i \leq m
    \]
    for $C^1$ closed simple curves $\Cu_{i,n}$.
\end{corollary}

\subsection{Continuity of integrals over level sets}
\label{sec:continuity-integral-levelsets}

The following result stating the continuity of integrals over connected components of a level set is shown under the nonvanishing gradient condition \eqref{eq:nonvanishing-gradient-on-component} by using partitions of unity.

\begin{proposition}
    \label{prop:continuity-integral-levelsets}
    Let functions $y_n, y$, $n \geq 1$, satisfy hypothesis \ref{hypo:levelsets}. Let $\Cu$ be a connected component of $\{y=t\}$, $t \in \R$ that satisfies \eqref{eq:nonvanishing-gradient-on-component}. Assume that $f_n(x) \to f(x)$ for all $x \in \Cu^{\bar \epsilon}$ for some constant $\bar \epsilon >0$. Then
    \begin{equation}
        \label{eq:continuity-integral-levelsets}
        \int_{\{y_n=t \} \cap \Cu^{\bar \epsilon}} f_n(x) \dH^1(x) \to \int_{\Cu} f(x)\dH^1(x) \quad \text{as} \quad n \to \infty.
    \end{equation}
\end{proposition}
\begin{proof}
    In light of \cref{prop:strong-positivity-gradients,prop:levelset-unique-component-continuity}, there exist positive constants $\epsilon_0, \alpha_\Cu >0$ and an integer $n_0$ such that
    \begin{equation}
        \label{eq:continuity-integral-claim1}
        \left\{
            \begin{aligned}
                & \nabla y(x) \cdot \nabla y_n(x') \geq \alpha_\Cu \quad \text{for all } n \geq n_0, x,x' \in \Cu^{\epsilon_0} \, \text{with } |x - x'| \leq 2 \epsilon_0, \\
                & \Cu^{\epsilon} \, \text{contains one and only one a connected component $\Cu_n$ of $\{y_n=t \}$ for all $n \geq n_0$ and $\epsilon \in (0, \epsilon_0)$}.
            \end{aligned}
        \right.
    \end{equation}
    Setting $\alpha_{\min} := \min\{ |\nabla y(x)| \mid x \in \Cu \}$ yields $\alpha_{\min} >0$.
    For each $x \in \Cu$, there exists a constant $r_x \in (0, \min\{\epsilon_0, \bar \epsilon \})$ such that one of the following two estimates is valid:
    \begin{equation*}
        \left|\frac{\partial y}{\partial x^i}(z)\right| \geq \frac{\alpha_{\min}}{2} > 0\quad \text{for all } z \in \overline{B}_{\R^2}(x, r_x) \cap \overline\Omega, i =1,2.
    \end{equation*}
    Since $\Cu$ is compact, it is covered by finitely many open balls $B_{\R^2}(x_1, r_{1}), \ldots, B_{\R^2}(x_{m_{\Cu}}, r_{m_{\Cu}})$ with $r_j := r_{x_j}$ for all $1 \leq j \leq m_{\Cu}$. By using partitions of unity, see. e.g. Theorem C.21 and Exercise C.22 in \cite{GiovanniLeoni2017}, there exist nonnegative functions $\psi_j \in C^\infty_c(\Omega)$ such that
    \[
        \textrm{supp}(\psi_j) \subset B_{\R^2}(x_j, r_{j}) \quad \text{and} \quad \sum_{j=1}^{m_{\Cu}} \psi_j(x) = 1 \quad \text{for all } 1 \leq j \leq m_{\Cu}, x \in \Omega.
    \]
    Therefore, in order to show \eqref{eq:continuity-integral-levelsets}, it suffices to prove for all $1 \leq j \leq m_{\Cu}$ that
    \begin{equation}
        \label{eq:continuity-integral-levelsets-ball}
        \int_{\{y_n=t \} \cap \Cu^{\bar \epsilon} \cap B_{\R^2}(x_j, r_j) } f_n(x) \psi_j(x) \dH^1(x) \to \int_{\Cu} f(x)\psi_j(x)\dH^1(x) \quad \text{as} \quad n \to \infty.
    \end{equation}
    To this end, fix $j$, put $x_0 := x_j$, $r_0 := r_j$, and without loss of generality assume that $x_0 = (0,0)$ and that
    \begin{equation*}
        \left|\frac{\partial y}{\partial x^2}(z)\right| \geq \frac{\alpha_{\min}}{2} > 0\quad \text{for all } z \in \overline{B}_{\R^2}(x_0, r_0) \cap \overline\Omega. %
    \end{equation*}
    Applying \cref{lem:Implicit_Function-Theorem-extended-boundary}, exploiting \eqref{eq:continuity-integral-claim1} and the fact that $0< r_j < \min\{\epsilon_0, \bar \epsilon \}$, we conclude that there exist positive constants $h_0, k_0, h_n, k_n$ and $C^1$ functions $g_0, g_n$ satisfying assertions \ref{item:paramterization-g} and \ref{item:parameterization-gn}--\ref{item:gn-g-Lip} in \ref{item:parameterization-approximation} of \cref{lem:Implicit_Function-Theorem-extended-boundary}.
    Hence, the limit \eqref{eq:continuity-integral-levelsets-ball} can be expressed as
    \begin{equation*}
        \int_{-h_n}^{k_n} f_n(\tau, g_n(\tau)) \psi_j(\tau, g_n(\tau)) \sqrt{1+ g_n'(\tau)^2} d\tau \to \int_{-h_0}^{k_0} f(\tau, g_0(\tau))\psi_j(\tau, g_0(\tau)) \sqrt{1+ g_0'(\tau)^2} d\tau \quad \text{as} \quad n \to \infty,
    \end{equation*}
    or, equivalently,
    \begin{equation}
        \label{eq:continuity-integral-levelsets-tau}
        \int_{-h_n}^{k_n} \phi_n(\tau) d\tau - \int_{-h_0}^{k_0} \phi_0(\tau) d\tau \to 0
    \end{equation}
    with
    \[
        \phi_n(\tau) := f_n(\tau, g_n(\tau)) \psi_j(\tau, g_n(\tau)) \sqrt{1+ g_n'(\tau)^2} \quad \text{and} \quad \phi_0(\tau) := f(\tau, g_0(\tau)) \psi_j(\tau, g_0(\tau)) \sqrt{1+ g_0'(\tau)^2}.
    \]
    Taking now $\epsilon >0$ small enough, we rewrite the right-hand side term in \eqref{eq:continuity-integral-levelsets-tau} as
    \begin{multline*}
        \int_{-h_n}^{k_n} \phi_n(\tau) d\tau - \int_{-h_0}^{k_0} \phi_0(\tau) d\tau = \int_{-h_0 + \epsilon}^{k_0 - \epsilon} [\phi_n(\tau) - \phi_0(\tau)] d\tau
        \\
        - \int_{-h_0}^{-h_0+\epsilon} \phi_0(\tau) d\tau -\int_{k_0-\epsilon}^{k_0} \phi_0(\tau) d\tau
        + \int_{-h_n}^{-h_0 + \epsilon } \phi_n(\tau) d\tau + \int_{k_0 - \epsilon}^{k_n} \phi_n (\tau) d\tau.
    \end{multline*}
    In view of assertion \ref{item:gn-g-Lip} in \cref{lem:Implicit_Function-Theorem-extended-boundary}, there holds
    \[
        \phi_n(\tau) \to \phi_0(\tau) \, \text{for all } \tau \in [-h_0+ \epsilon, k_0 - \epsilon] \quad \text{and} \quad
        |\phi_n(\tau)| \leq C \quad \text{for all } \tau \in [-h_n, k_n].
    \]
    Lebesgue's Dominated Convergence Theorem then implies that
    \[
        \lim\limits_{n \to \infty} \left|\int_{-h_n}^{k_n} \phi_n(\tau) d\tau - \int_{-h_0}^{k_0} \phi_0(\tau) d\tau \right| \leq \tilde{C}\epsilon,
    \]
    where we have used the limits $h_n \to h_0$ and $k_n \to k_0$; see assertion \ref{item:limit-hn-kn} in \cref{lem:Implicit_Function-Theorem-extended-boundary}. We thus obtain \eqref{eq:continuity-integral-levelsets-tau}.
\end{proof}

\section{An nonsmooth quasilinear elliptic optimal control problem}
\label{sec:nonsmooth-optimal-control}

\subsection{Main assumptions and preliminary results}\label{sec:assumption}
Let $a$ be a finitely $PC^2$-function of the form
\begin{equation}
    \label{eq:PC1-rep}
    a(t) := \1_{(-\infty, \bar t]}(t) a_0(t) + \1_{(\bar t, \infty)} a_1(t) \quad \text{for all } t \in \R,
\end{equation}
for a given number $\bar t \in \R$ and given functions $a_0 \in C^2((-\infty, \bar t])$ and $a_1 \in C^2([\bar t, \infty))$ with $a_{0}(\bar t) = a_{1}(\bar t)$.
Obviously, the function $a$ is of class $C^2$ over the intervals $(-\infty, \bar t) \cup (\bar t, \infty)$, but not even of class $C^1$ in general. %

\begin{remark}
    \label{rem:PC2-finitely}
    Let us emphasize that the results and the underlying analysis in this paper can be applied to the situation in which the function $a$ is continuous and is twice continuously differentiable on finitely many intervals (i.e., a finitely $PC^2$ function; see \cite{ClasonNhuRosch_os2nd} for a precise definition).
    However, in order to keep the presentation concise and to be able to focus on the main arguments, we restrict the presentation to the simplest such situation given by \eqref{eq:PC1-rep}.
\end{remark}

By $\{a'\}_{\bar t+0}^{\bar t-0}$, we denote the difference between the one-sided derivatives of $a$ at $\bar t$ from left and right, i.e.,
\[
    \{a'\}_{\bar t+0}^{\bar t-0} : = \lim\limits_{t \to \bar t^{-}} a'(t) - \lim\limits_{t \to \bar t^{+}} a'(t) = a'_{0}(\bar t) - a'_{1}(\bar t).
\]
By setting
\begin{equation}
    \label{eq:sigma-i}
    \sigma_0 := |\{a'\}_{\bar t+0}^{\bar t-0}| = |a'_{0}(\bar t) - a'_{1}(\bar t)|,
\end{equation}
we see that this term determines the differentiability of $a$ and plays a crucial part in the second-order optimality conditions for \eqref{eq:P}; see \cite{ClasonNhuRosch_os2nd}.
Moreover, $a$ is directionally differentiable and its directional derivative is given by
\begin{equation} \label{eq:directional-der-a1}
    a'(t; s) = \1_{(-\infty, \bar t)}(t)a_0'(t)s + \1_{( \bar t, \infty)}(t)a_1'(t)s +\1_{\{\bar t\}}(t) [\1_{(0, \infty)}(s) a_{1}'(\bar t)s + \1_{(-\infty, 0)}(s) a_{0}'(\bar t)s], \quad t,s \in \R.
\end{equation}
The following assumptions shall hold throughout the whole paper except in \cref{sec:Q2-formulation}, where we will only require the convexity of $\Omega$ instead of \cref{ass:domain} below.
\begin{assumption}
\item \label{ass:domain}
    $\Omega \subset \R^2$ is an open bounded convex polygonal.
\item \label{ass:b_func}
    The Lipschitz continuous function $b: \overline\Omega \to \R$ satisfies $b(x) \geq \underline{b} > 0$ for all $x \in \overline\Omega$. %
\item \label{ass:PC1-func}
    $a: \R \to \R$ is nonnegative and given by \eqref{eq:PC1-rep}. %

\item \label{ass:cost_func}
    $L: \Omega \times \R \to \R$ is a Carath\'{e}odory function that is of class $C^2$ w.r.t. the second variable with $L(\cdot,0) \in L^1(\Omega)$. Besides, for any $M>0$, there exist $C_M >0$ and $\psi_M \in L^{\bar p}(\Omega)$ ($\bar p > 2$) such that
    $ %
    |\frac{\partial L}{\partial y}(x,y)| \leq \psi_M(x) \, \text{and} \, |\frac{\partial^2 L}{\partial y^2}(x,y)| \leq C_M
    $ %
    for all $y \in \R$ with $|y| \leq M$, and a.e. $x \in \Omega$.
\end{assumption}

From \cref{ass:cost_func}, we deduce by explicit computation, a Taylor expansion, and Lebesgue's Dominated Convergence Theorem that the functional $ L^\infty(\Omega) \ni y \mapsto \int_\Omega L(x,y(x)) dx \in \R$ is of class $C^2$.

\medskip

In the remainder of this subsection, we state some known results for the state equation, the adjoint state equation, and the optimality conditions for \eqref{eq:P}; see, e.g. \cite{ClasonNhuRosch_os2nd}.
Let us first consider the state equation
\begin{equation} \label{eq:state}
    -\dive [(b + a(y) )\nabla y ] = u \, \text{in } \Omega, \quad y =0 \, \text{on } \partial\Omega.
\end{equation}
\begin{theorem}[{cf. \cite{ClasonNhuRosch_os2nd}, Thms.~3.1 and 3.5}]
    \label{thm:control2state-oper}
    Let \crefrange{ass:domain}{ass:PC1-func} hold. Then, the control-to-state operator $S: W^{-1,p}(\Omega) \ni u \mapsto y_u \in W^{1,p}_0(\Omega)$ with $y_u$ being the unique solution to \eqref{eq:state} is of class $C^1$. Moreover, for any $u,v \in W^{-1,p}(\Omega)$ with $p > 2$ and $y_u:= S(u)$, $z_v:= S'(u)v$ is the unique solution to
    \begin{equation} \label{eq:deri}
        -\dive [(b + a(y_u)) \nabla z_v + \1_{\{ y_u \neq \bar t \}}a'(y_u) z_v\nabla y_u ] = v \, \text{in } \Omega, \quad z_v =0 \, \text{on } \partial\Omega.
    \end{equation}
    Moreover, there exists a number $p_* >2$ such that for any $p \in [2,p_*)$ and for any bounded set $U \subset L^p(\Omega)$, there hold $S(u) \in W^{1,p}_0(\Omega) \cap W^{2,p}(\Omega)$ and
    $ %
    \norm{S(u)}_{W^{2,p}(\Omega)} \leq C_U. %
    $ %
\end{theorem}
\begin{proof}
    The wellposedness and the continuous differentiability of $S$ follows from Theorems 3.1 and 3.5 in \cite{ClasonNhuRosch_os2nd}. On the other hand, the $H^2$- and $W^{1,\infty}(\Omega)$-regularity of solutions to \eqref{eq:state}
    was also shown in Theorem 3.1 in \cite{ClasonNhuRosch_os2nd} when the right-hand side $u$ belongs to $L^q(\Omega)$ with $q >2$. Moreover, if $U$ is a bounded subset of $L^q(\Omega)$ with $q>2$, then there holds
    $ %
    \norm{S(u)}_{H^2(\Omega)} + \norm{S(u)}_{W^{1,\infty}(\Omega)} \leq C_U %
    $ %
    for all $u \in U$.
    To show the higher $W^{2,p}$-regularity as well as the corresponding a priori estimate, we observe that \eqref{eq:state} can be rewritten as
    \begin{equation} \label{eq:state-transfer}
        -\Delta y = \frac{1}{ b+ a(y)}[u + \nabla b \cdot \nabla y + \1_{\{ y \neq \bar t \}} a'(y)|\nabla y|^2] \, \text{in } \Omega, \quad
        y =0 \, \text{on } \partial\Omega
    \end{equation}
    (see, e.g. equations (A.1) and (A.3) in \cite{ClasonNhuRosch_os2nd}, Lem.~A.1).
    Since $\Omega$ is assumed to be a convex polygon in $\R^2$, Theorem 4.4.3.7 in \cite{Grisvard1985} shows that there exists a constant $p_* := 2/(2-\min\{\pi \omega_{\max}^{-1},2\}) >2$ depending on the maximal interior angle $\omega_{\max} <\pi$ of the domain $\Omega$ such that any solution $y$ to \eqref{eq:state-transfer} belongs to $W^{2,p}(\Omega)$ provided that $u \in L^p(\Omega)$ for all $p \in (2, p_*)$. Of course, we have $y \in H^2(\Omega)$ when $u \in L^2(\Omega)$ due to the convexity of $\Omega$. Finally, the $W^{2,p}$-estimate of solutions $y$ is derived by applying Theorem 4.3.2.4 in \cite{Grisvard1985} to \eqref{eq:state-transfer} and using the a priori $W^{1,\infty}(\Omega)$-estimates of $y$, \cref{ass:b_func,ass:PC1-func}.
\end{proof}

We now consider the adjoint state equation
\begin{equation} \label{eq:adjoint-state}
    -\dive [(b + a(y_u)) \nabla \varphi] + \1_{\{ y_u \neq \bar t \}}a'(y_u) \nabla y_u \cdot \nabla \varphi = v \, \text{in } \Omega, \quad \varphi =0 \, \text{on } \partial\Omega
\end{equation}
for $u \in W^{-1,p}(\Omega)$, $p >2$, $v\in H^{-1}(\Omega)$, and $y_u:= S(u)$.
\begin{theorem}[{cf. \cite{ClasonNhuRosch_os2nd}, Lem.~4.1}]
    \label{thm:adjoint-equation}
    Let \crefrange{ass:domain}{ass:PC1-func} be satisfied and let $p, q> 2$ be arbitrary. Then, for any $u \in W^{-1,p}(\Omega), v \in H^{-1}(\Omega)$, a unique $\varphi \in H^1_0(\Omega)$ exists and uniquely solves \eqref{eq:adjoint-state}.
    Furthermore, if $U$ is a bounded subset in $L^p(\Omega)$, then for any $u \in U$ and any $v \in L^q(\Omega)$, the solution $\varphi$ of \eqref{eq:adjoint-state} belongs to $H^2(\Omega) \cap W^{1,\infty}(\Omega)$ and there holds
    $ %
    \norm{\varphi}_{H^{2}(\Omega)} + \norm{\varphi}_{W^{1,\infty}(\Omega)} \leq C_{U} \norm{v}_{L^{q}(\Omega)}.
    $ %
    Moreover, if $u \in L^p(\Omega)$ and $v \in L^r(\Omega)$ with $r \in (2, p_*)$, then $\varphi \in W^{2,r}(\Omega)$, where $p_*$ is defined as in \cref{thm:control2state-oper}.
\end{theorem}
\begin{proof}
    Thanks to Lemma 4.1 in \cite{ClasonNhuRosch_os2nd},
    all conclusions except the last one of the theorem are verified. To derive the $W^{2,r}$-regularity of $\varphi$, we argue as in the proof of \cref{thm:control2state-oper}.
\end{proof}

\begin{remark}
    \label{rem:regularity-derivative-S}
    Despite the $W^{2,p}$-regularity of the state and adjoint state, the function $z_v := S'(u)v$ defined in \eqref{eq:deri} is only in $W^{1,p}(\Omega)$
    due to the fact that we cannot apply the chain rule for the divergence acting on the term $\1_{\{ y \neq \bar t \}}a'(y) z\nabla y$ even with $y \in W^{2,p}(\Omega)$ and $z \in W^{1,p}(\Omega)$. Indeed, for any $y \in W^{2,p}(\Omega)$, the vector-valued function $f(x) := \1_{\{ y \neq \bar t \}}(x)a'(y(x)) \nabla y(x)$ in general does not belong to $(W^{1,p}(\Omega))^N$ with $N=2$. For example, considering $N :=1$, $\Omega := (0,2) \subset \R^1$, and $\bar t := 1$, we define functions
    \[
        a(t) = |t- 1| \quad \text{and} \quad y(x) = x^2, \quad x \in (0,2).
    \]
    Easily, we have
    \[
        f(x) = \1_{\{ y \neq 1\}}(x) a'(y(x)) \nabla y(x) = 2x \1_{\{ x \neq 1 \}}(x) \sign(x^2-1), \quad x \in (0,2)
    \]
    and thus $f \notin W^{1,1}(0,2)$.
\end{remark}

The optimal control problem \eqref{eq:P} can be expressed in the form
\begin{equation}\label{eq:P2}
    \tag{P}
    \min_{u\in \mathcal{U}_{ad}} j(u) = \int_\Omega L(x,S(u)(x))\dx + \frac\nu2\norm{u}_{L^2(\Omega)}^2 %
\end{equation}
with
\[
    \mathcal{U}_{ad} := \{u \in L^\infty(\Omega) \mid \alpha \leq u(x) \leq \beta \qquad \text{for a.e. } x \in \Omega \}.
\]
Under \crefrange{ass:domain}{ass:cost_func}, the cost functional $j: L^2(\Omega) \to \R$ is of class $C^1$. Moreover, there holds
\begin{equation} \label{eq:obje-deri}
    j'(u)v = \int_{\Omega} (\varphi_u + \nu u )v \dx \quad \text{for } u, v \in L^2(\Omega)
\end{equation}
with $\varphi_u \in H^{1}_0(\Omega)$ solving \eqref{eq:adjoint-state} corresponding to the right-hand side term $v$ substituted by $\frac{\partial L}{\partial y}(\cdot, S(u))$; see \cite{ClasonNhuRosch_os2nd}, Thm.~4.2. We have the following first-order necessary optimality conditions
from Theorem 4.3 in \cite{ClasonNhuRosch_os2nd} and thus derive the regularity of the optimal control as well as the corresponding state and adjoint state from \cref{thm:control2state-oper,thm:adjoint-equation} and from Sobolev embeddings.
\begin{theorem}[{\cite{ClasonNhuRosch_os2nd}, Thm.~4.3}]
    \label{thm:1st-OC}
    Assume that \crefrange{ass:domain}{ass:cost_func} are satisfied.
    Then there exists at least one minimizer $\bar u$ of \eqref{eq:P2}. Moreover, there exists an adjoint state $\bar\varphi \in H^1_0(\Omega)$ such that for $\bar y:=S(\bar u)$,
    \begin{subequations}
        \label{eq:1st-OS}
        \begin{align}
            &
            -\dive [(b + a(\bar y)) \nabla \bar y] = \bar u \quad \text{in } \Omega, \quad \bar y = 0 \, \text{on } \partial\Omega, \label{eq:state_OS} \\
            &-\dive [(b +a(\bar y) ) \nabla \bar\varphi ] + \1_{\{ \bar y \neq \bar t \} } a'(\bar y) \nabla \bar y \cdot \nabla \bar \varphi = \frac{\partial L}{\partial y}(x, \bar y) \, \text{in } \Omega, \quad \bar \varphi =0 \, \text{on } \partial\Omega,
            \label{eq:adjoint_OS} \\
            &\int_\Omega (\bar \varphi + \nu \bar u )( u - \bar u ) \dx \geq 0 \quad \text{for all } u \in \mathcal{U}_{ad}. \label{eq:normal_OS}
        \end{align}
    \end{subequations}
    Furthermore, $\bar y \in W^{2,p}(\Omega)$ and $\bar \varphi \in W^{2,r}(\Omega)$ for any $p, r \in (2, p_*)$ and $r \leq \bar p$ with $\bar p$ and $p_*$, respectively, defined in \cref{ass:cost_func} and \cref{thm:control2state-oper}. Therefore, $\bar y$ and $\bar \varphi$ belong to $C^1(\overline\Omega)$ and $\bar u$ is Lipschitz continuous on $\overline\Omega$.
\end{theorem}
Assume that $\bar\varphi \in H^1_0(\Omega)$ satisfies \eqref{eq:1st-OS}.
The \emph{critical cone} of the problem \eqref{eq:P2} at $\bar u$ is defined as
\begin{equation} \label{eq:critical-cone}
    \mathcal{C}({\mathcal{U}_{ad};\bar u}) := \{ v \in L^2(\Omega) \mid v \geq 0 \, \text{if } \bar u = \alpha, v \leq 0 \, \text{if } \bar u = \beta, v = 0 \, \text{if } \bar \varphi + \nu \bar u \neq 0 \ \text{a.e. in } \Omega
    \}.
\end{equation}

In the rest of this subsection, we shall provide second-order necessary and sufficient
optimality conditions for \eqref{eq:P}. To this end, the curvature functional of $j$ is first introduced and can be separated into three contributions. For any $(u,y,\varphi)\in L^2(\Omega)\times H^1(\Omega)\times W^{1,\infty}(\Omega)$, the smooth part and the first-order nonsmooth part of the curvature in direction $(v_1,v_2)\in L^2(\Omega)^2$ are given by
\begin{align*}%
    & Q_s(u,y,\varphi; v_1,v_2) := \frac{1}{2} \int_\Omega \frac{\partial^2 L}{\partial y^2}(\cdot,y)z_{v_1} z_{v_2} \dx + \frac{\nu}{2} \int_\Omega v_1v_2 \dx
    - \frac{1}{2} \int_\Omega\1_{ \{ y \neq \bar t \}} a''(y)z_{v_1}z_{v_2} \nabla y \cdot \nabla \varphi \dx, \\
    &Q_1(u,y,\varphi; v_1,v_2) :=
    - \frac{1}{2} \int_\Omega [a'(y; z_{v_1}) \nabla z_{v_2} +a'(y; z_{v_2})\nabla z_{v_1} ]\cdot \nabla\varphi \dx,
\end{align*}
for $z_{v_i} := S'(u)v_i$, $i=1,2$, respectively.
The critical part for our analysis is of course the second-order nonsmooth part, which requires some additional notation.
For ease of exposition, we use the following notation in the remainder.
For any $y \in C(\overline\Omega)$ and any $\tau_1, \tau_2 \in \R$, we define the set
$ %
\Omega_{y}^{[\tau_1, \tau_2]} := \{ y \in [\bar t + \tau_1, \bar t + \tau_2 ] \};
$ %
similar sets such as $\Omega_{y}^{[\tau_1, \tau_2)}$ are defined in the same way. Let $\delta >0$ be arbitrary but fixed.
For any $y, \hat y \in C(\overline\Omega)$, we set
\begin{equation}
    \label{eq:Omega-123-sets}
    \left\{
        \begin{aligned}
            \Omega_{y, \hat y }^{2} & := \Omega_{\hat y }^{(0, \delta)} \cap \Omega_{y }^{(-\delta, 0]} = \{\hat y \in (\bar t, \bar t + \delta), y \in (\bar t-\delta, \bar t] \},\\
            \Omega_{y, \hat y }^{3} & := \Omega_{\hat y}^{(-\delta, 0)} \cap \Omega_{y }^{[0,\delta)} = \{ \hat y \in (\bar t- \delta, \bar t), y \in [\bar t,\bar t + \delta) \}.
        \end{aligned}
    \right.
\end{equation}
For any $s \in \R$, $u, v \in L^2(\Omega)$, $y \in C(\overline\Omega) \cap H^1(\Omega)$, and $\varphi \in W^{1,\infty}(\Omega)$, we set
\begin{equation} \label{eq:zeta-func}
    \left\{
        \begin{aligned}
            \zeta_0(u,y;s,v) &:= - \{a'\}_{\bar t+0}^{\bar t-0} (\bar t - S(u+sv)) \1_{\Omega_{S(u+sv), y }^{3}}, \quad \zeta_1(u,y;s,v) := \{a'\}_{\bar t+0}^{\bar t-0} (\bar t - S(u+sv)) \1_{\Omega_{S(u+sv), y }^{2}}, \\
            \zeta(u,y;s,v)& := \zeta_0(u,y;s,v) + \zeta_1(u,y;s,v) = \{a'\}_{\bar t+0}^{\bar t-0} (\bar t - S(u+sv))\left[ \1_{\Omega_{S(u+sv), y }^{2}} - \1_{\Omega_{S(u+sv), y }^{3}} \right]
        \end{aligned}
    \right.
\end{equation}
with $ \{a'\}_{\bar t+0}^{\bar t-0}$ defined as in \eqref{eq:sigma-i}.
We then define for any $\{s_n\} \in c_0^+:= \{ \{s_n\}\subset (0,\infty) \mid s_n \to 0 \}$ and $v \in L^2(\Omega)$
\begin{multline}
    \label{eq:key-term-sn}
    \tilde{Q}(u,y,\varphi;\{s_n\}, v) := \liminf\limits_{n \to \infty} \frac{1}{s_n^2} \int_\Omega \sum_{i =0 }^1 \zeta_i(u,y;s_n,v) \nabla y \cdot \nabla \varphi \dx \\
    = \{a'\}_{\bar t+0}^{\bar t-0} \liminf\limits_{n \to \infty} \frac{1}{s_n^2} \int_\Omega (\bar t - S(u+s_nv)) \left[\1_{\Omega_{S(u+s_n v), y }^{2}} -\1_{\Omega_{S(u+s_nv), y }^{3}}\right] \nabla y \cdot \nabla \varphi \dx.
\end{multline}
The second-order nonsmooth part of the curvature in direction $v \in L^2(\Omega)$ is then given by
\begin{equation}
    \label{eq:key-term}
    Q_2(u,y,\varphi; v) := \inf \{ \tilde{Q}(u,y,\varphi; \{s_n\}, v) \mid \{s_n\} \in c_0^+ \},
\end{equation}
and finally the total curvature in direction $v$ is
\begin{equation}
    \label{eq:curvature}
    Q(u,y,\varphi;v) := Q_s(u,y, \varphi; v,v) + Q_1(u,y, \varphi; v,v) +Q_2(u,y, \varphi; v).
\end{equation}
\begin{remark}
    \label{rem:curvature-term-simplification}
    The definitions of the sets $\Omega_{y, \hat y }^{2}$ and $\Omega_{y, \hat y }^{3}$ in \eqref{eq:Omega-123-sets} are identical to the ones for $\Omega_{y, \hat y }^{1,2}$ and $\Omega_{y, \hat y }^{0,3}$ in Lemma 3.3 in \cite{ClasonNhuRosch_os2nd} for the case where $K:=1$, $t_0 := -\infty$, $t_1 := \bar t$, and $t_2 := \infty$. Similarly, the definitions of the functionals $Q_s$, $Q_1$, and $Q_2$ in this subsection can be derived from the associated ones in \S\,5.1 in \cite{ClasonNhuRosch_os2nd}.
\end{remark}

According to Proposition 5.6 and Lemma 5.7 in \cite{ClasonNhuRosch_os2nd}, $Q_2$ is weakly lower semicontinuous in the last variable and satisfies
\begin{equation} \label{eq:sigma-bounded}
    |Q_2(u,S(u),\varphi; v) | \leq \Sigma(S(u)) \norm{\nabla \varphi}_{L^\infty(\Omega)} \norm{S'(u)v}_{L^\infty(\Omega)}^2
    \quad\text{for all }u, v \in L^2(\Omega)\text{ and } \varphi \in W^{1,\infty}(\Omega),
\end{equation}
with the \emph{jump functional}
\begin{equation}
    \label{eq:E-functional}
    \begin{aligned}[b]
        \Sigma(y) & := \sigma_0 \limsup\limits_{r \to 0^+} \frac{1}{r} \sum_{m= 1}^{2} \int_\Omega [ \1_{\{ 0 < |y - \bar t | \leq r \}} | \partial_{x_m} y | ]\dx, \quad y \in W^{1,1}(\Omega) \cap C(\overline\Omega)
    \end{aligned}
\end{equation}
for $\sigma_0$ defined in \eqref{eq:sigma-i}.
Also, from Corollary 5.5 in \cite{ClasonNhuRosch_os2nd}, it holds for any $u \in L^2(\Omega)$, $\{s_n\} \in c_0^+$ and $v_n \rightharpoonup v$ in $L^2(\Omega)$ that
\begin{equation} \label{eq:sigma-tilde}
    \liminf\limits_{n \to \infty} \frac{1}{s_n^2} \int_\Omega \sum_{i =0}^1 \zeta_i(u,S(u);s_n,v_n) \nabla S(u) \cdot \nabla \varphi \dx = \tilde{Q}(u,S(u),\varphi;\{s_n\}, v) \geq Q_2(u,S(u),\varphi;v),
\end{equation}
provided that $\Sigma(S(u)) < \infty$.

We are now ready to state the second-order necessary and sufficient optimality conditions for \eqref{eq:P} in general, for which the
difference
between these conditions is only in the strictness of the inequality (``no-gap'').
\begin{theorem}[second-order necessary optimality conditions, {\cite{ClasonNhuRosch_os2nd}, Thm.~5.9}]
    \label{thm:2nd-OS-nec}
    Let \crefrange{ass:domain}{ass:cost_func} be fulfilled. Assume that $\bar u$ is a local minimizer to \eqref{eq:P2} such that $\Sigma(\bar y) < \infty$ for $\bar y:=S(\bar u)$. Then, there is a $\bar\varphi \in W^{1,\bar p}_0(\Omega) \cap W^{1,\infty}(\Omega)$, with $\bar p$ defined in \cref{ass:cost_func}, that together with $\bar u, \bar y$ satisfies \eqref{eq:1st-OS} and
    \begin{equation}
        \label{eq:2nd-OS-nec}
        Q(\bar u, \bar y, \bar \varphi; v) \geq 0
        \quad\text{for all } v\in \mathcal{C}({\mathcal{U}_{ad};\bar u})
    \end{equation}
    with $Q$ defined in \eqref{eq:curvature}.
\end{theorem}

\begin{theorem}[second-order sufficient optimality conditions, {\cite{ClasonNhuRosch_os2nd}, Thm.~5.10}]
    \label{thm:2nd-OS-suf}
    Let \crefrange{ass:domain}{ass:cost_func} be valid. Assume that $\bar u$ is a feasible point of \eqref{eq:P2} such that $\Sigma(\bar y) < \infty$ for $\bar y:=S(\bar u)$. Assume further that there is a $\bar\varphi \in W^{1,\bar p}_0(\Omega) \cap W^{1,\infty}(\Omega)$, with $\bar p$ defined in \cref{ass:cost_func}, that together with $\bar u, \bar y$ satisfies \eqref{eq:1st-OS} and
    \begin{equation}
        \label{eq:2nd-OS-suff}
        Q(\bar u, \bar y, \bar \varphi; v) > 0
        \quad\text{for all } v\in \mathcal{C}({\mathcal{U}_{ad};\bar u}) \setminus \{0\}
    \end{equation}
    with $Q$ defined in \eqref{eq:curvature}.
    Then there exist constants $c_0, \rho_0 >0$ satisfying
    \begin{equation*}
        j(\bar u) + c_0 \norm{u - \bar u}_{L^2(\Omega)}^2 \leq j(u) %
        \qquad\text{for all } u \in \mathcal{U}_{ad} \cap \overline B_{L^2(\Omega)}(\bar u, \rho_0).
    \end{equation*}
\end{theorem}
It is noted that the term $2Q(\bar u,\bar y,\bar \varphi;v)$ can be seen as a second-order generalized derivative of $j$ at $\bar u$ in the direction $v$; see, e.g. the proof of Theorem 5.9 in \cite{ClasonNhuRosch_os2nd} and Remark 5.1 in \cite{Nhu2021}).

\subsection{An explicit formula for the curvature functional} \label{sec:Q2-formulation}
In this subsection, we assume that the domain $\Omega$ is open, bounded, and convex in $\R^2$ only.
We shall establish an explicit formula for the curvature term $Q(u,y,\varphi;v)$, defined in \eqref{eq:curvature}, for two situations:
\begin{enumerate}[label=(\roman*)]
    \item \label{item:curvature-nonvanishing-gradient} The gradient of $y$ does not vanish on connected components of $\{y= \bar t\}$;
    \item \label{item:curvature-vanishing-gradient} The gradient of $y$ vanishes on connected components of $\{y= \bar t\}$.
\end{enumerate}
For the first situation, we will apply the results shown in \cref{sec:levelset-facts}.
For the latter situation, we need the following notion.
\begin{definition} \label{def:locally-convex-concave}
    A function $y: \overline\Omega \to \R$ is called
    \emph{uniformly locally convex-concave} on a set $V \subset \overline\Omega$ if an $\epsilon >0$ exists such that for any $x \in V$, $y$ is either convex or concave on $B_{\R^2}(x,\epsilon) \cap \overline \Omega$.
\end{definition}
\begin{proposition}
    \label{prop:gradient-vanish}
    Let $t \in \R$ be arbitrary and let $y \in C^1(\overline\Omega)$ be uniformly locally convex-concave on the level set $\{y=t\}$. Assume that $\Cu$ is a {connected component} of $\{y=t\}$. If $\nabla y$ vanishes at some point $x_0 \in \Cu$, then $\nabla y(x) = 0$ for all $x \in \Cu$.
\end{proposition}
\begin{proof}
    There is an $\epsilon>0$ such that, for any $x \in \{y=t\}$, the restriction $y\mid_{B_{\R^2}(x,\epsilon) \cap \overline\Omega}$ is either convex or concave. Since $\nabla y(x_0) =0$, then $x_0$ is a local extremal point of $y(x)$ and so is every point in $B_{\R^2}(x,\epsilon) \cap \Cu$. From this and the connection property of $\Cu$, we have $\nabla y = 0$ on $\Cu$.
\end{proof}
The following result is a direct consequence of \cref{prop:gradient-vanish}.
\begin{corollary}
    \label{cor:gradient-vanish}
    Let $t \in \R$ be arbitrary and let $y \in C^1(\overline\Omega)$ be uniformly locally convex-concave on the level set $\{y=t\}$. Assume that $\Cu$ is a connected component of $\{y=t\}$.
    If $\nabla y(x_0) \neq 0$ for some point $x_0 \in \Cu$, then $\nabla y(x) \neq 0$ for all $x \in \Cu$.
\end{corollary}

The following result will play an important role in establishing an explicit formula of the curvature functional \eqref{eq:curvature}. Its proof rests on the following several lemmas.
\begin{theorem}
    \label{thm:key-formula}
    Let $\{s_n\} \in c_0^+$, $\bar\varphi \in C^1(\overline\Omega) \cap W^{2,1}(\Omega)$ and $\bar y, y_n \in C^1(\overline\Omega)$ such that $y_n =\bar y = 0$ on $\partial\Omega$, $y_n \to \bar y$ in $C^1(\overline\Omega)$ and $(y_n - \bar y)/s_n \to w$ in $W^{1,p}_0(\Omega)$ for some $p >2$ and $w \in W^{1,p}_0(\Omega)$. Let $\Cu$ be a closed connected component of $\{\bar y = \bar t \}$. Assume that one of the following conditions is satisfied:
    \begin{enumerate}[label=(C\arabic*)]
        \item \label{ass:nonvanishing-gradient-An} The gradient of $\bar y$ does not vanish on $\Cu$, i.e.,
            \begin{equation}
                \label{eq:main-hypothesis-y-novanish}
                \min\{|\nabla \bar y(x)|: x \in \Cu\} > 0;
            \end{equation}
        \item \label{ass:vanishing-gradient-An-convex-concave} There holds that
            \begin{equation} \label{eq:main-hypothesis-y-vanish}
                \left\{
                    \begin{aligned}
                        & \text{$\bar y$ is uniformly locally convex-concave on $\Cu$},\\
                        & \nabla \bar y = 0 \, \text{on } \Cu,\\
                        & \Ha^{N-1}(\{\bar y=t \} \cap \Cu^{\tilde\epsilon}) \leq C_0, \quad \text{f.a.a. } t \in (\bar t-r_0,\bar t+r_0), N=2,
                    \end{aligned}
                \right.
            \end{equation}
            for some constants $\tilde \epsilon, r_0, C_0>0$;
        \item \label{ass:vanishing-gradient-structure}
            There holds that
            \begin{equation} \label{eq:main-hypothesis-y-vanish-structure}
                \left\{
                    \begin{aligned}
                        & \nabla \bar y = 0 \, \text{on } \Cu,\\
                        & \meas_{\R^2}( \{0 < | \bar y - \bar t| < r \} \cap \Cu^{\tilde\epsilon}) \leq c_s r, \quad \text{for all } r \in (0, r_0),
                    \end{aligned}
                \right.
            \end{equation}
            for some constants $\tilde \epsilon, r_0, c_s>0$.
    \end{enumerate}
    Then there exists an $\epsilon_0 = \epsilon_0(\Cu) \in(0, \tilde \epsilon)$ such that for any $\epsilon \in (0, \epsilon_0)$,
    \begin{equation}
        \label{lim:key}
        \frac{A_n(\Cu,\epsilon)}{s_n^2} \to \frac{1}{2} \int_{\Cu} \1_{\{|\nabla \bar y| > 0\}} w^2 \frac{\nabla \bar y \cdot \nabla \bar \varphi}{|\nabla \bar y|} \dH^{1}(x) \quad \text{as} \quad n \to \infty,
    \end{equation}
    where
    \begin{equation}
        \label{eq:An-term}
        A_n(\Cu,\epsilon) := \int_{\Cu^\epsilon} (\bar t - y_n) [\1_{ \Omega_{y_n, \bar y }^{2} } -\1_{ \Omega_{y_n, \bar y }^{3}}] \nabla \bar y \cdot \nabla \bar \varphi \dx
    \end{equation}
    with $ \Omega_{y_n, \bar y }^{2}$ and $\Omega_{y_n, \bar y }^{3}$ defined in \eqref{eq:Omega-123-sets}.
\end{theorem}
\begin{remark}
    \label{rem:oscillation}
    Note that \eqref{eq:main-hypothesis-y-vanish} does not require that the level set $\{\bar y= \bar t\}$ has measure zero.
    When $N=1$, the last condition in \eqref{eq:main-hypothesis-y-vanish} means that
    for a.e. $t$ in a neighborhood of $\bar t$, the level sets $\{\bar y= t\}$ consist of finitely many points; in other words, the function $\bar y$ oscillates around the values $\bar t$ only finitely many times. In \cite{ClasonNhuRosch_os2nd}, Exam.~5.3, this condition was shown to be equivalent to the finiteness of the jump functional $\Sigma(\bar y)$ introduced in \eqref{eq:E-functional} for the case $N=1$.
\end{remark}

\begin{remark}
    \label{rem:An-term}
    According to the definition of the sets $\Omega_{y, \hat y}^{j}$ with $j = 2,3$ in \eqref{eq:Omega-123-sets}, the sets $\Omega_{y_n, \bar y }^{2}$ and $\Omega_{y_n, \bar y }^{3}$ can be expressed as
    \begin{equation}
        \label{eq:Omega23-rewritten}
        \Omega_{y_n, \bar y }^{2} = \tilde \Omega_{y_n, \bar y }^{2} \cup \{ \bar y \in (\bar t, \bar t + \delta), y_n = \bar t \} \quad \text{and} \quad \Omega_{y_n, \bar y }^{3} = \tilde \Omega_{y_n, \bar y }^{3} \cup \{ \bar y \in (\bar t - \delta, \bar t), y_n = \bar t \}
    \end{equation}
    with
    \begin{equation}
        \label{eq:Omega23-interior}
        \left\{
            \begin{aligned}
                \tilde \Omega_{y, \hat y }^{2} & := \{ \hat y \in (\bar t, \bar t +\delta), y \in (\bar t - \delta, \bar t) \},\\
                \tilde \Omega_{y, \hat y }^{3} & := \{ \hat y \in (\bar t - \delta, \bar t), y \in (\bar t, \bar t + \delta) \}
            \end{aligned}
        \right.
    \end{equation}
    for $y, \hat y \in C(\overline\Omega)$.
    Obviously, one has
    \begin{equation}
        \label{eq:Omega23-interior-symmetric}
        \tilde \Omega_{y, \hat y }^{2} = \tilde \Omega_{ \hat y, y }^{3}.
    \end{equation}
    Moreover, thanks to \eqref{eq:An-term}, the term $A_n(\Cu,\epsilon)$ can be rewritten as
    \begin{equation}
        \label{eq:An-term-rewritten}
        A_n(\Cu,\epsilon) := \int_{\Cu^\epsilon} (\bar t - y_n) [\1_{ \tilde \Omega_{y_n, \bar y }^{2} } -\1_{ \tilde \Omega_{y_n, \bar y }^{3}}] \nabla \bar y \cdot \nabla \bar \varphi \dx.
    \end{equation}
\end{remark}

\begin{lemma}
    \label{lem:gradient-nonvanish}
    Under assumption \ref{ass:nonvanishing-gradient-An} in \cref{thm:key-formula}, then there exists an $\epsilon_0 = \epsilon_0(\Cu) >0$ such that
    \begin{equation}
        \label{lim:key-nonvanish}
        \frac{A_n(\Cu,\epsilon)}{s_n^2} \to \frac{1}{2} \int_{\Cu} w^2 \frac{\nabla \bar y \cdot \nabla \bar \varphi}{|\nabla \bar y|} \dH^{1}(x) \quad \text{as} \quad n \to \infty \quad \text{for all } \epsilon \in (0, \epsilon_0).
    \end{equation}
\end{lemma}
\begin{proof}
    By \cref{prop:decomposition-levelset}, $\Cu$ is a closed $C^1$ simple curve in $\overline\Omega$.
    Thanks to \cref{prop:strong-positivity-gradients},
    there are constants $\epsilon_1, \alpha >0$ such that, for $n$ large enough,
    \begin{equation}
        \label{eq:gradient-positive}
        \nabla y_n (x)\cdot \nabla \bar y(x') \geq \alpha >0 \quad \text{for all} \quad x, x' \in \Cu^{\epsilon_1} \, \text{with } |x - x'| \leq 2 \epsilon_1.
    \end{equation}
    Since $\Cu$ is a closed component in $\{\bar y=\bar t\}$, there holds $\Cu^{\epsilon_2} \cap (\{\bar y=\bar t\}\backslash \Cu)^{\epsilon_2} = \emptyset$ for some constant $\epsilon_2>0$.
    Moreover, the sign of $(\bar y-\bar t)$ in $\Cu^-_{\epsilon_2}$ is opposite to the one in $\Cu^+_{\epsilon_2}$. Without loss of generality, we can thus assume that \begin{equation}
        \label{eq:y-positive-negative}
        \bar y < \bar t \quad \text{on } \Cu^-_{\epsilon_2} \quad \text{and} \quad \bar y > \bar t \quad \text{on } \Cu^+_{\epsilon_2}.
    \end{equation}
    Set $\epsilon_0 := \min\{\epsilon_1, \epsilon_2\}$, $\tau_n := \norm{y_n-\bar y}_{C(\overline\Omega)}$, fix any $\epsilon \in (0,\epsilon_0)$, and define the sets (depending on $\epsilon$)
    \begin{equation}
        \label{eq:Omega23-local-def}
        \Omega_n^2 := \tilde \Omega_{y_n, \bar y }^{2} \cap \Cu^{\epsilon} \quad \text{and} \quad \Omega_n^3 := \tilde \Omega_{y_n, \bar y }^{3} \cap \Cu^{\epsilon}.
    \end{equation}
    Here the sets $\tilde \Omega_{y_n, \bar y }^{2}$ and $\tilde \Omega_{y_n, \bar y }^{3}$ are defined in \eqref{eq:Omega23-interior}.
    Obviously, for $n$ large enough
    such that $\tau_n < \delta/2$, we have
    \begin{subequations}
        \label{eq:sets-local}
        \begin{align}
            \Omega_n^2 & = \{ \bar y \in (\bar t, \bar t +\delta), y_n \in (\bar t-\delta, \bar t) \} \cap \Cu^{\epsilon} = \{ y_n < \bar t < \bar y \} \cap \Cu^{\epsilon} \subset \Cu^+_{\epsilon}, \\
            \Omega_n^3 & = \{ \bar y \in (\bar t- \delta, \bar t), y_n \in (\bar t, \bar t+\delta) \} \cap \Cu^{\epsilon} = \{ \bar y < \bar t < y_n \} \cap \Cu^{\epsilon} \subset \Cu^-_{\epsilon}.
        \end{align}
    \end{subequations}
    As a result of \cref{prop:levelset-unique-component-continuity}, for $n$ large enough, $\Cu^{\epsilon}$ contains one and only one connected component $\Cu_n$ of $\{y_n = \bar t\}$. By \cref{prop:decomposition-levelset},
    the set $\Cu_n = \{y_n = \bar t\} \cap \Cu^{\epsilon}$ is a closed $C^1$ simple curve in $\R^2$. We now consider two cases.

    \noindent\emph{Case 1: $\Cu = \partial\Omega$.} In this case, we have $\bar t=0$ and thus $\Cu_n = \partial\Omega$ for sufficiently large $n$. Moreover, one has $\Cu^+_{\epsilon} = \emptyset$, $\Cu^\epsilon = \Cu \cup \Cu^{-}_{\epsilon}$ and then $\Omega_n^2 = \emptyset$ for $n$ large enough.
    On the other hand, since $\Cu_n = \{y_n = \bar t\} \cap \Cu^{\epsilon}$ is a closed simple curve in $\R^2$ for $n$ large enough and $y_{n} \to \bar y$ in $C(\overline\Omega)$, we deduce from \eqref{eq:y-positive-negative} that $y_n \leq \bar t$ on $\Cu^\epsilon$ and therefore $\Omega_n^3 = \emptyset$ for sufficiently large $n$.
    From this and the definition of $A_n(\Cu,\epsilon)$, we have \eqref{lim:key-nonvanish} because of the vanishing on $\partial\Omega$ of $w$.

    \noindent\emph{Case 2: $\Cu \neq \partial\Omega$.} In this case, both $\Cu_{\epsilon}^{+}$ and $\Cu_{\epsilon}^{-}$ are nonempty.
    To estimate $A_n(\Cu,\epsilon)$,
    we use the expression \eqref{eq:An-term-rewritten} of $A_n(\Cu,\epsilon)$ and split it into two terms as follows:
    \begin{multline} \label{eq:An-local}
        A_n(\Cu,\epsilon) := \int_{\Cu^\epsilon} (\bar t - y_n) [\1_{ \tilde{\Omega}^{2}_{y_n, \bar y} } -\1_{ \tilde{\Omega}^{3}_{y_n, \bar y}}] \nabla (\bar y-y_n) \cdot \nabla \bar \varphi \dx \\
        + \int_{\Cu^\epsilon} (\bar t - y_n) [\1_{ \tilde{\Omega}^{i,2}_{y_n, \bar y} } -\1_{ \tilde{\Omega}^{3}_{y_n, \bar y}}] \nabla y_n \cdot \nabla \bar \varphi \dx =B_n + C_n
    \end{multline}
    with
    \begin{equation}
        \label{eq:Bn-Cn-term}
        B_n := \int_{\Cu^\epsilon} (\bar t - y_n) [\1_{ \tilde{\Omega}^{2}_{y_n, \bar y} } -\1_{ \tilde{\Omega}^{3}_{y_n, \bar y}}] \nabla (\bar y-y_n) \cdot \nabla \bar \varphi \dx \quad \text{and} \quad C_n := \int_{\Cu^\epsilon} (\bar t - y_n) [\1_{ \tilde{\Omega}^{2}_{y_n, \bar y} } -\1_{ \tilde{\Omega}^{3}_{y_n, \bar y}}] \nabla y_n \cdot \nabla \bar \varphi \dx.
    \end{equation}
    Moreover, we deduce from the definition of sets $\Omega_n^2$ and $\Omega_n^3$ that
    \[
        B_n = \int_{\Cu^\epsilon} (\bar t - y_n) [\1_{ \Omega_n^2 } -\1_{ \Omega_n^3}] \nabla (\bar y-y_n) \cdot \nabla \bar \varphi \dx \quad \text{and} \quad C_n = \int_{\Cu^\epsilon} (\bar t - y_n) [\1_{ \Omega_n^2 } -\1_{ \Omega_n^3}] \nabla y_n \cdot \nabla \bar \varphi \dx.
    \]
    From the facts
    \begin{equation}
        \label{eq:structure-esti}
        \left\{
            \begin{aligned}
                &|\bar t - \bar y|, |\bar t - y_n| \leq |\bar y-y_n| \quad \text{on} \quad \Omega_n^2 \cup \Omega_n^3,\\
                & \Omega_n^2 \cup \Omega_n^3 \subset \{0 < |\bar y-\bar t| \leq \tau_n\},
            \end{aligned}
        \right.
    \end{equation}
    we deduce from H\"{o}lder's inequality for $p':= \tfrac{p}{p-1}$ that
    \begin{equation}
        \label{eq:infinitesimal-value}
        |B_n| \leq \norm{y_n -\bar y}_{C(\overline\Omega)} \norm{\nabla( y_n-\bar y)}_{L^p(\Omega)} \norm{\nabla \bar \varphi \1_{\{0 < |\bar y-\bar t| \leq \tau_n\}}}_{L^{p'}(\Omega)} = o(s^2_n).
    \end{equation}
    From this and \eqref{eq:An-local}, there holds
    \begin{equation}
        \label{eq:An-esti}
        A_n(\Cu,\epsilon) = o(s_n^2) + C_n.
    \end{equation}
    We now estimate $C_n$. To this end, we shall employ \cref{prop:Green-identity-general}.
    We first rewrite $C_n$ as
    \begin{equation} \label{eq:Cn-term}
        C_n = - \frac{1}{2} \int_{\Cu^\epsilon} [\1_{ \Omega_n^2 } -\1_{ \Omega_n^3}] \nabla (y_n-\bar t)^2 \cdot \nabla \bar \varphi \dx.
    \end{equation}
    From \eqref{eq:sets-local} and the limit $y_n \to \bar y$ in $C^1(\overline\Omega)$ that, for $n$ large enough, there holds
    \begin{equation}
        \label{eq:boundary-inclusion-Omega23}
        \partial \Omega_n^2 \cup \partial \Omega_n^3 \subset \Cu \cup \Cu_n.
    \end{equation}

    Applying \cref{prop:Green-identity-general} for functions $y_1 := \bar y, y_2 := y_n$ and $v:= (y_n-\bar t)^2$, $ \varphi := \bar \varphi$, and open sets $\mathcal{S}^+ := \Omega_n^2$, $\mathcal{S}^- := \Omega_n^3$ (see \eqref{eq:sets-local} for the definition of the sets $ \Omega_n^2, \Omega_n^3$) yields
    \begin{multline*}
        C_n = \frac{1}{2} \int_{\Cu^\epsilon} [\1_{ \Omega_n^2 } -\1_{ \Omega_n^3}] (y_n-\bar t)^2 \Delta\bar \varphi dx \\
        + \frac{1}{2} \int_{ \Cu } (y_n-\bar t)^2 \nabla\bar \varphi \cdot \frac{\nabla \bar y}{|\nabla \bar y|} \dH^{1}(x) - \frac{1}{2} \int_{ \Cu_n } (y_n-\bar t)^2 \nabla\bar \varphi \cdot \frac{\nabla y_n}{|\nabla y_n|} \dH^{1}(x),
    \end{multline*}
    or equivalently,
    \begin{equation}
        \label{eq:Cn-esti}
        C_n = \frac{1}{2} \int_{\Cu^\epsilon} [\1_{ \Omega_n^2 } -\1_{ \Omega_n^3}] (y_n-\bar t)^2 \Delta\bar \varphi dx+ \frac{1}{2} \int_{ \Cu } (y_n-\bar y)^2 \nabla\bar \varphi \cdot \frac{\nabla \bar y}{|\nabla \bar y|} \dH^{1}(x).
    \end{equation}
    In view of \eqref{eq:structure-esti},
    there then holds
    \begin{equation}
        \label{eq:Cn-first-term-esti}
        \left|\int_{\Cu^\epsilon} [\1_{ \Omega_n^2 } -\1_{ \Omega_n^3}] (y_n-\bar t)^2 \Delta\bar \varphi dx \right| \leq \norm{y_n-\bar y}_{C(\overline\Omega)}^2 \int_{\{0 < |\bar y-\bar t| \leq \tau_n \}} |\Delta \bar \varphi| \dx = o(s_n^2),
    \end{equation}
    as a result of Lebesgue's Dominated Convergence Theorem.
    Combining this with \eqref{eq:Cn-esti} gives
    \begin{equation}
        \label{eq:Cn-computation}
        C_n = o(s_n^2) +\frac{1}{2} \int_{\Cu} (y_n-\bar y)^2 \nabla \bar \varphi \cdot \frac{\nabla \bar y}{|\nabla \bar y|} \dH^{1}(x).
    \end{equation}
    The combination of \eqref{eq:Cn-computation} with \eqref{eq:An-esti} and the limits $(y_n-\bar y)/s_n \to w$ in $W^{1,p}_0(\Omega)$ yields \eqref{lim:key-nonvanish}.
\end{proof}

\begin{lemma}
    \label{lem:gradient-vanish}
    Under assumption \ref{ass:vanishing-gradient-An-convex-concave},
    there is an $\epsilon_0 = \epsilon_0(\Cu) >0$ such that
    \begin{equation}
        \label{lim:key-vanish}
        \frac{A_n(\Cu,\epsilon)}{s_n^2} \to 0 \quad \text{as} \quad n \to \infty \quad \text{for all } \epsilon \in (0, \epsilon_0).
    \end{equation}
\end{lemma}
\begin{proof}
    Without loss of generality, %
    assume that there thus exists an $\epsilon_3>0$ such that $\bar y$ is convex on $B(x,\epsilon_3) \cap \overline\Omega$ for all $x \in \Cu$; see \cref{fig:vanish}.
    We first set $\epsilon_0 := \min\{\tilde\epsilon,\epsilon_2, \epsilon_3\}$ with constants $\tilde\epsilon$ given in \eqref{eq:main-hypothesis-y-vanish} and $\epsilon_2$ defined as in the proof of \cref{lem:gradient-nonvanish}. Let us take $\epsilon \in (0,\epsilon_0)$ arbitrarily but fixed and reuse all symbols defined in the proof of
    \cref{lem:gradient-nonvanish}.
    Moreover, the relations \eqref{eq:sets-local}--\eqref{eq:Cn-term} are still valid.
    The convexity of $\bar y$ and the fact that $\nabla \bar y =0$ on $\Cu$ imply that $\bar y > \bar t$ on $\Cu^\epsilon \backslash \Cu$.
    This and \eqref{eq:sets-local} guarantee that
    \begin{equation} \label{eq:Omega3}
        \Omega_n^2 = \{y_n < \bar t \} \cap (\Cu^{\epsilon} \backslash \Cu) \quad \text{and}\quad \Omega_n^3 =\emptyset \quad \text{for all $n \geq n_0$ for some integer $n_0 \in \N$}
    \end{equation}
    and thus
    \begin{equation}
        \label{eq:Cn-term-vanish-ass}
        C_n = \int_{\Omega_n^2}(\bar t-y_n)\nabla y_n \cdot \nabla \bar \varphi \dx.
    \end{equation}
    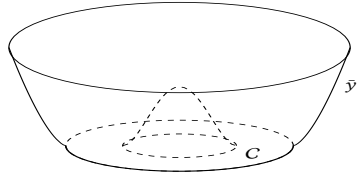
\begin{figure}
        \centering
        \resizebox{4.75cm}{2.25cm}{%
            \begin{tikzpicture}
                \tkzDefPoint(0,0){A}
                \tkzDefPoint(-2,0){B}
                \tkzDefPoint(2,0){C}
                \tkzDefPoint(-1,0){D}
                \tkzDefPoint(1.1,-0.2){E}
                \draw [thin, domain=-3:-2] plot (\x,{2.5*(abs(\x+2))^(1.5)});
                \draw [dashed, domain=-1:1] plot (\x,{1.5*((\x)^2-1)^2});
                \draw [thin, domain=2:3] plot (\x,{2.5*(\x-2)^(1.5))});
                \draw[dashed] (A) ellipse (2.cm and 0.65cm);
                \draw[dashed] (A) ellipse (1cm and 0.3cm);
                \draw[thick] (B) arc(-180:0:2.cm and 0.65cm);
                \draw[thin] (0,2.5) ellipse (3cm and 1.15cm);
                \draw (E) node[align=right]{$\quad\mathcal{C}$};
                \draw (3,1.5) node[align=right]{$\bar y$};
            \end{tikzpicture}
        }
        \caption{a closed component $\Cu$ of the level set $\{\bar y=\bar t\}$ with positive measure $\meas_{\R^2}(\Cu)$}
        \label{fig:vanish}
    \end{figure}
    We now split the sequence $\{n\}$ into subsequences, also denoted by $\{n\}$, that satisfy one of the following conditions:
    \begin{multicols}{2}
        \begin{enumerate}[label= (\alph*)]
            \item \label{item:nonempty2-vanish} $\Omega_n^2 \neq \emptyset$ for all $n \geq n_0$;
            \item \label{item:empty2-vanish} $\Omega_n^2 = \emptyset$ for all $n \geq n_0$.
        \end{enumerate}
    \end{multicols}
    For \ref{item:empty2-vanish},
    we have from \eqref{eq:An-local} that $A_n(\Cu,\epsilon) =0$ for all $n \geq n_0$.
    The limit \eqref{lim:key-vanish} thus follows.
    It remains to consider \ref{item:nonempty2-vanish}.
    To this end, by the Morse--Sard Theorem in Sobolev spaces (see; e.g. Theorem 5 in \cite{Figalli2008} and Corollary 5.2 in \cite{Bourgain2013}), for a.e. $t \in \R$, the level set $\{y_n=t\}$ is a finite disjoint family of $C^1$ simple curves in $\R^2$ and $\nabla y_n$ does not vanish on $\{y_n =t\}$. Therefore, for any $n \geq n_0$, there exists a $\bar t_n \in \R$ such that
    \begin{equation}
        \label{eq:tni-choice}
        \left\{
            \begin{aligned}
                & \bar t \geq \bar t_n \neq 0, \\
                & \bar t -\bar t_n = o(s_n^2), \\
                & \{y_n=\bar t_n\} \, \text{consists of finite disjoint closed $C^1$ simple curves},\\
                & |\nabla y_n(x)| > 0 \quad \text{for all } x \in \{ y_n = \bar t_n\}.
            \end{aligned}
        \right.
    \end{equation}
    Note that any closed curve in $\{y_n = \bar t_n\}$ does not intersect with the boundary $\partial\Omega$ since $\bar t_n \neq 0$.
    From the expression of $C_n$ in \eqref{eq:Cn-term} and the identities in \eqref{eq:Omega3}, one has
    \begin{equation*}
        \begin{aligned}
            C_n = \int_{\Omega_n^2}(\bar t-y_n)\nabla y_n \cdot \nabla \bar \varphi \dx &= \int_{\{\bar t_n \leq y_n \leq \bar t \} \cap \Cu^{\epsilon}}(\bar t-y_n)\nabla y_n \cdot \nabla \bar \varphi \dx\\
            & \quad + \int_{\{y_n < \bar t_n \} \cap \Cu^{\epsilon}}(\bar t-y_n)\nabla y_n \cdot \nabla \bar \varphi \dx =: D_n^1 + D_n^2.
        \end{aligned}
    \end{equation*}
    Obviously, one has from the choice of $\bar t_n$ in \eqref{eq:tni-choice} and the fact $\norm{y_n -\bar y}_{W^{1,p}_0(\Omega)} = O(s_n)$ that
    \begin{equation*}
        \begin{aligned}
            |D_n^1| &\leq \int_{\R^2} \1_{\{\bar t_n \leq y_n \leq \bar t \} \cap \Cu^{\epsilon}}|\bar t-y_n||\nabla y_n| | \nabla \bar \varphi| \dx \\
            & \leq \int_{\R^2} \1_{\{\bar t_n \leq y_n \leq \bar t \} \cap \Cu^{\epsilon}}|\bar t-y_n||\nabla (y_n-\bar y)| | \nabla \bar \varphi| \dx
            + \int_{\R^2} \1_{\{\bar t_n \leq y_n \leq \bar t \} \cap \Cu^{\epsilon}}|\bar t-y_n||\nabla \bar y| | \nabla\bar \varphi| \dx \\
            & \leq |\bar t - \bar t_n| \int_{\R^2} \1_{\{\bar t_n \leq y_n \leq \bar t \} \cap \Cu^{\epsilon}} |\nabla (y_n-\bar y)|| \nabla \bar \varphi| \dx
            + \int_{\R^2} \1_{\{\bar t_n \leq y_n \leq \bar t \} \cap \Cu^{\epsilon}}|\bar t-y_n||\nabla \bar y| | \nabla\bar \varphi| \dx \\
            & =o(s_n^2) + \int_{\R^2} \1_{\{\bar t_n \leq y_n \leq \bar t \} \cap \Cu^{\epsilon}}|\bar t-y_n||\nabla \bar y| | \nabla \bar \varphi| \dx.
        \end{aligned}
    \end{equation*}
    From the inclusion $\{\bar t_n \leq y_n \leq \bar t \} \subset \{\bar t_n - \tau_n \leq \bar y \leq \bar t+\tau_n \}$ and the coarea formula for Lipschitz mappings; see, e.g. \cite{Evans1992}, Thm.~2, p.~117 and \cite{AlbertiBianchiniCrippa2013}, Sec.~2.7, we have
    \begin{equation*}
        \begin{aligned}
            |D_n^1| &\leq o(s_n^2) + \int_{\R} \left[\int_{\{\bar y =t\}} \1_{\{\bar t_n \leq y_n \leq \bar t \} \cap \Cu^{\epsilon}}|\bar t-y_n| | \nabla \bar \varphi| \dH^{1}(x)\right]dt\\
            & = o(s_n^2) + \int_{\bar t_n -\tau_n}^{\bar t+\tau_n} \left[\int_{\{\bar y =t\}} \1_{\{\bar t_n \leq y_n \leq \bar t \} \cap \Cu^{\epsilon}}|\bar t-y_n| | \nabla \bar \varphi| \dH^{1}(x)\right]dt\\
            & \leq o(s_n^2) + C_0(\bar t - \bar t_n+2 \tau_n)(\bar t-\bar t_n) \norm{\nabla \bar \varphi}_{L^\infty(\Omega)} = o(s_n^2),
        \end{aligned}
    \end{equation*}
    due to the choice of $\bar t_n$ in \eqref{eq:tni-choice} and conditions in \eqref{eq:main-hypothesis-y-vanish}.
    For $D_n^2$, we see from \eqref{eq:tni-choice} that $\{y_n = \bar t_n \} \cap \Cu^\epsilon \neq \emptyset$ for $n$ large enough. If
    $\{y_n < \bar t_n\} \cap \Cu^\epsilon =\emptyset$, then $D_n^2=0$. Otherwise, let $\Gamma_n$ be the boundary of $\{y_n < \bar t_n\} \cap \Cu^\epsilon$. There are two possibilities in principle: either an infinite subsequence $\{k\}$ of $\{n\}$ exists and satisfies $\Gamma_k \cap \partial \Cu^\epsilon \neq \emptyset$, or there is no such an subsequence.
    Let us see that the first possibility is not actually a correct assumption. Indeed, if $\Gamma_k \cap \partial \Cu^\epsilon \neq \emptyset$, then $\norm{y_k - \bar y}_{C(\overline\Omega)} \geq \bar y(x) - \bar t >0$ for all $x \in \partial\Cu^\epsilon$. This contradicts the limit $\norm{y_k - \bar y}_{C(\overline\Omega)} \to 0$ as $k \to \infty$.
    Therefore, the second possibility always holds. It then must be true that $\Gamma_n \cap \partial \Cu^\epsilon = \emptyset$ and so $\Gamma_n = \{y_n =\bar t_n\} \cap \Cu^\epsilon$ for $n$ large enough.
    Combining this with the last two conditions in \eqref{eq:tni-choice} and the Implicit Function Theorem,
    we deduce that the open set $\{y_n < \bar t_n\} \cap \Cu^\epsilon$ decomposes into subdomains with Lipschitz boundaries.
    By rewriting $D_n^2$ and then using integration by parts, we have
    \begin{equation*}
        \begin{aligned}
            D_n^2 & = \int_{\{y_n < \bar t_n \} \cap \Cu^{\epsilon}}(\bar t-y_n)\nabla y_n \cdot \nabla \bar \varphi \dx \\
            & = \int_{\{y_n < \bar t_n \} \cap \Cu^{\epsilon}}(\bar t-\bar t_n)\nabla y_n \cdot \nabla \bar \varphi \dx
            + \int_{\{y_n < \bar t_n \} \cap \Cu^{\epsilon}}(\bar t_n-y_n)\nabla y_n \cdot \nabla \bar \varphi \dx \\
            &= \int_{\{y_n < \bar t_n \} \cap \Cu^{\epsilon}}(\bar t- \bar t_n)\nabla y_n \cdot \nabla \bar \varphi \dx - \frac{1}{2} \int_{\{y_n < \bar t_n \} \cap \Cu^{\epsilon}}\nabla (y_n -\bar t_n)^2\cdot \nabla\bar \varphi \dx \\
            & = \int_{\{y_n < \bar t_n \} \cap \Cu^{\epsilon}}(\bar t-\bar t_n)\nabla y_n \cdot \nabla \bar \varphi \dx + \frac{1}{2} \int_{\{y_n < \bar t_n \} \cap \Cu^{\epsilon}} (y_n -\bar t_n)^2\Delta\bar \varphi \dx - \frac{1}{2} \int_{\Gamma_n} (y_n-\bar t_n)^2 \nabla\bar \varphi \cdot \nu_n \dH^{1}(x) \\
            & =\int_{\{y_n < \bar t_n \} \cap \Cu^{\epsilon}}(\bar t-\bar t_n)\nabla y_n \cdot \nabla \bar \varphi \dx + \frac{1}{2} \int_{\{y_n < \bar t_n \} \cap \Cu^{\epsilon}} (y_n -\bar t_n)^2\Delta\bar \varphi \dx,
        \end{aligned}
    \end{equation*}
    where $\nu_n$ stands for the outward unit normal vector to $\Gamma_n$. From this, the first two conditions in \eqref{eq:tni-choice}, and the fact that $\bar y(x) \geq \bar t$ for all $x \in \Cu^\epsilon$, we have
    \begin{align*}
        |D_n^2| & \leq (\bar t -\bar t_n) \norm{\nabla y_n}_{L^\infty(\Omega)} \norm{\nabla \bar \varphi}_{L^\infty(\Omega)}\norm{\1_{\{y_n < \bar t_n \} \cap \Cu^{\epsilon}}}_{L^1(\Omega)} +\frac{1}{2} \int_{\{ y_n < \bar t_n \} \cap \Cu^{\epsilon}} (y_n -\bar y)^2 |\Delta \bar \varphi| \dx \\
        & = o(s_n^2) + \frac{1}{2} \norm{y_n - \bar y}_{C(\overline\Omega)}^2 \int_{\{ y_n < \bar t_n \} \cap \Cu^{\epsilon}} |\Delta \bar \varphi| \dx\\
        & = o(s_n^2) + O(s_n^2) \int_\Omega \1_{\{ y_n < \bar t_n \} \cap \Cu^{\epsilon}} |\Delta \bar \varphi| \dx.
    \end{align*}
    Since $\bar \varphi \in W^{2,1}(\Omega)$ and $0 \leq \1_{\{ y_n < \bar t_n \} \cap \Cu^{\epsilon}} \leq \1_{\{0< |y_n - \bar y| \leq \tau_n \} \cap \Cu^{\epsilon}} \to 0$ a.e. in $\Omega$ as $n \to \infty$, we deduce from Lebesgue's Dominated Convergence Theorem that
    \[
        \int_\Omega \1_{\{ y_n < \bar t_n \} \cap \Cu^{\epsilon}} |\Delta \bar \varphi| \dx \to 0 \quad \text{as } n \to \infty.
    \]
    We thus have $D_n^2 = o(s_n^2)$.
    In conclusion, we derive $D_n^1 = o(s_n^2)$, $D_n^2 = o(s_n^2)$ and thus $C_n = D_n^1 + D_n^2 = o(s_n^2)$.
    We then deduce from \eqref{eq:An-esti} that $A_n(\Cu,\epsilon) = o(s_n^2)$. Consequently, the desired conclusion of the lemma follows.
\end{proof}

\begin{lemma}
    \label{lem:gradient-vanish-structure}
    Under assumption \ref{ass:vanishing-gradient-structure},
    there is an $\epsilon_0 = \epsilon_0(\Cu) >0$ such that
    \begin{equation}
        \label{lim:key-vanish-structure}
        \frac{A_n(\Cu,\epsilon)}{s_n^2} \to 0 \quad \text{as} \quad n \to \infty \quad \text{for all } \epsilon \in (0, \epsilon_0).
    \end{equation}
\end{lemma}
\begin{proof}
    We first set $\epsilon_0 := \min\{\tilde\epsilon,\epsilon_2\}$ with constants $\tilde\epsilon$ given in \eqref{eq:main-hypothesis-y-vanish-structure} and $\epsilon_2$ defined as in the proof of \cref{lem:gradient-nonvanish}. Let us take $\epsilon \in (0,\epsilon_0)$ arbitrarily but fixed and reuse all symbols defined in the proof of \cref{lem:gradient-nonvanish}.
    Moreover, the relations \eqref{eq:sets-local}--\eqref{eq:An-esti} are still valid. From the definition of $C_n$ in \eqref{eq:Bn-Cn-term}, we deduce from \eqref{eq:structure-esti} that
    \begin{align*}
        \frac{1}{s_n^2}|C_n| & \leq \norm{\nabla y_n}_{L^\infty(\{ 0 < | \bar y - \bar t | \leq \tau_n \} \cap \Cu^\epsilon)}\norm{\nabla \bar \varphi}_{L^\infty(\Omega)} \frac{1}{s_n^2}\int_{\Cu^\epsilon} |\bar y - \bar y_n| \1_{\{ 0 < | \bar y - \bar t | \leq \tau_n \}} \dx \\
        & \leq \norm{\nabla y_n}_{L^\infty(\{ 0 < | \bar y - \bar t | \leq \tau_n \} \cap \Cu^\epsilon)}\norm{\nabla \bar \varphi}_{L^\infty(\Omega)} \frac{\tau_n}{s_n^2} \meas_{\R^2}(\{ 0 < | \bar y - \bar t | \leq \tau_n \} \cap \Cu^\epsilon) \\
        & \leq c_s \norm{\nabla y_n - \nabla \bar y}_{L^\infty(\Omega)}\norm{\nabla \bar \varphi}_{L^\infty(\Omega)} \frac{\tau_n^2}{s_n^2} + c_s \norm{\nabla \bar y}_{L^\infty(\{ 0 < | \bar y - \bar t | \leq \tau_n \} \cap \Cu^\epsilon)}\norm{\nabla \bar \varphi}_{L^\infty(\Omega)} \frac{\tau_n^2}{s_n^2}
    \end{align*}
    for all $n$ large enough, where we have exploited the last estimate in \eqref{eq:main-hypothesis-y-vanish-structure}. Letting $n \to \infty$ and using the identity in \eqref{eq:main-hypothesis-y-vanish-structure}
    yields
    \[
        C_n = o(s_n^2),
    \]
    which along with \eqref{eq:An-esti} gives \eqref{lim:key-vanish-structure}.
\end{proof}

Similar to \cref{lem:gradient-nonvanish,lem:gradient-vanish-structure}, we have the following result.
\begin{proposition}
    \label{prop:key-formular-2}
    Assume that either \ref{ass:nonvanishing-gradient-An} or \ref{ass:vanishing-gradient-structure} is fulfilled.
    Then there exists an $\epsilon_0 = \epsilon_0(\Cu) >0$ such that, for any $\epsilon \in (0, \epsilon_0)$,
    \begin{equation}
        \label{lim:key-for-y}
        \frac{\tilde A_n(\Cu,\epsilon)}{s_n^2} \to -\frac{1}{2} \int_{\Cu} \1_{\{|\nabla \bar y| >0 \}} w^2 \frac{\nabla \bar y \cdot \nabla \bar \varphi}{|\nabla \bar y|} \dH^{1}(x) \quad \text{as} \quad n \to \infty,
    \end{equation}
    where
    \begin{equation}
        \label{eq:An-term-for-y}
        \tilde A_n(\Cu,\epsilon) := \int_{\Cu^\epsilon} (\bar t - \bar y) [\1_{ \Omega_{y_n, \bar y }^{2} } -\1_{ \Omega_{y_n, \bar y }^{3}}] \nabla \bar y \cdot \nabla \bar \varphi \dx.
    \end{equation}
\end{proposition}
\begin{proof}
    We first note from \eqref{eq:Omega23-rewritten} and \eqref{eq:Omega23-interior-symmetric} that
    \begin{equation*}
        \left\{
            \begin{aligned}
                \Omega_{y_n, \bar y }^{2} & = \tilde \Omega_{y_n, \bar y }^{2} \cup \{ \bar y \in (\bar t, \bar t +\delta), y_n =\bar t \} = \tilde{\Omega}_{\bar y, y_n}^{3} \cup \{\bar y \in (\bar t, \bar t+\delta), y_n =\bar t \},\\
                \Omega_{y_n, \bar y }^{3} & = \tilde{\Omega}_{y_n, \bar y}^{3} \cup \{\bar y \in (\bar t-\delta, \bar t), y_n = \bar t \} = \tilde{\Omega}_{\bar y, y_n}^{2} \cup \{ \bar y \in (\bar t-\delta, \bar t), y_n = \bar t \},
            \end{aligned}
        \right.
    \end{equation*}
    where the sets $\tilde{\Omega}_{y, \hat y}^{2}$ and $\tilde{\Omega}_{y, \hat y}^{3}$ are defined as in \eqref{eq:Omega23-interior}. Obviously, one has
    \[
        \tilde{\Omega}_{\bar y, y_n}^{2} \cap \{\bar y \in (\bar t-\delta, \bar t), y_n = \bar t \} = \emptyset = \tilde{\Omega}_{\bar y, y_n}^{3} \cap \{ \bar y \in (\bar t, \bar t +\delta), y_n = \bar t \}.
    \]
    Combining these with the definition of $\tilde A_n(\Cu,\epsilon)$ in \eqref{eq:An-term-for-y} yields
    \begin{multline*}
        \tilde A_n(\Cu,\epsilon) = \int_{\Cu^\epsilon} (\bar t - \bar y) [1_{\{\bar y \in (\bar t, \bar t +\delta), y_n =\bar t \} } -\1_{ \{\bar y \in (\bar t-\delta, \bar t), y_n =\bar t \}}] \nabla \bar y \cdot \nabla \bar \varphi \dx \\
        - \int_{\Cu^\epsilon} (\bar t - \bar y) [\1_{ \tilde \Omega_{\bar y, y_n}^{2} } -\1_{ \tilde \Omega_{\bar y, y_n }^{3}}] \nabla \bar y \cdot \nabla \bar \varphi \dx =: \tilde{B}_n -\tilde{C}_n.
    \end{multline*}
    Since $\nabla y_n = 0$ a.e. on $\{y_n =\bar t\}$, $\tilde{B}_n$ can be rewritten as follows
    \begin{equation*}
        \tilde B_n = \int_{\Cu^\epsilon} (y_n - \bar y) [\1_{\{\bar y \in (\bar t, \bar t +\delta), y_n = \bar t \} } -\1_{ \{\bar y \in (\bar t-\delta, \bar t), y_n =\bar t \}}] \nabla (\bar y-y_n) \cdot \nabla \bar \varphi \dx.
    \end{equation*}
    Analogous to \eqref{eq:infinitesimal-value}, one has $\tilde{B}_n = o(s_n^2)$. This implies that
    \begin{equation}
        \label{eq:An-for-y-esti}
        \tilde A_n(\Cu,\epsilon) = o(s_n^2) - \tilde C_n.
    \end{equation}
    From the definitions of $C_n$ in \eqref{eq:Bn-Cn-term} and of $\tilde C_n$, we can get $\tilde C_n$ by interchanging $\bar y$ and $y_n$ in the integrand of $C_n$.
    For the situation where \ref{ass:vanishing-gradient-structure} holds, by using the argument as in the proof of \cref{lem:gradient-vanish-structure}, we have
    \[
        \tilde C_n = o(s_n^2),
    \]
    which together with \eqref{eq:An-for-y-esti} gives \eqref{lim:key-for-y}.
    It remains to consider the situation where \ref{ass:nonvanishing-gradient-An} is satisfied.
    We shall estimate $\tilde C_n$ similarly to the estimate of $C_n$ in the proof of \cref{lem:gradient-nonvanish}. To this end,
    by \eqref{eq:Omega23-interior}, we have for $n$ large enough that
    \begin{equation}
        \label{eq:Omega23-tilde-interior-local}
        \left\{
            \begin{aligned}
                & \tilde \Omega_{\bar y, y_n}^{2} \cap \Cu^\epsilon = \{ y_n \in (\bar t, \bar t + \delta), \bar y \in (\bar t -\delta, \bar t) \} \cap \Cu^\epsilon = \{\bar y < \bar t < y_n \} \cap \Cu^\epsilon =: \tilde{\Omega}_n^2 ,\\
                & \tilde \Omega_{\bar y, y_n}^{3} \cap \Cu^\epsilon = \{ \bar y \in (\bar t, \bar t + \delta), y_n \in (\bar t -\delta, \bar t) \} \cap \Cu^\epsilon = \{\bar y > \bar t > y_n \} \cap \Cu^\epsilon =: \tilde{\Omega}_n^3.
            \end{aligned}
        \right.
    \end{equation}
    In comparison with \eqref{eq:sets-local}, there also hold
    \[
        \tilde{\Omega}_n^2 = \Omega_n^3 \quad \text{and} \quad \tilde{\Omega}_n^3 = \Omega_n^2.
    \]
    We now rewrite $\tilde{C}_n$ as
    \begin{equation}
        \label{eq:Cn-tilde-term}
        \tilde{C}_n = - \frac{1}{2}\int_{\Cu^\epsilon} [\1_{ \tilde \Omega_n^{2} } -\1_{ \tilde \Omega_n^{3}}] \nabla (\bar y - \bar t)^2 \cdot \nabla \bar \varphi \dx.
    \end{equation}
    Now we consider the following two cases. \\
    \noindent$\bullet$ \emph{Case 1: $\Cu = \partial\Omega$}. Then by using the same argument as in Case 1 in the proof of \cref{lem:gradient-nonvanish}, we have $\tilde{C}_n = 0$.\\
    \noindent\emph{$\bullet$ Case 2: $\Cu \neq \partial \Omega$}.
    For this case, one has $\Cu \cap \partial\Omega = \emptyset$.
    Using \eqref{eq:boundary-inclusion-Omega23} and

    applying \cref{prop:Green-identity-general} for functions $y_1:= y_n$, $y_2 := \bar y$, $v:= (\bar y - \bar t)^2$, $\varphi := \bar \varphi$ and the sets $\mathcal{S}^{+} := \tilde \Omega_n^{2}$, $\mathcal{S}^{-} := \tilde \Omega_n^{3}$, we deduce from \eqref{eq:Cn-tilde-term} and \eqref{eq:Omega23-tilde-interior-local} that
    \begin{align*}
        \tilde{C}_n = \frac{1}{2}\int_{\Cu^\epsilon} [\1_{ \tilde \Omega_n^{2} } -\1_{ \tilde \Omega_n^{3}}] (\bar y - \bar t)^2 \Delta \bar \varphi \dx +\frac{1}{2} \int_{\Cu_n }(\bar y - \bar t)^2 \nabla \bar \varphi \cdot \frac{\nabla y_n}{|\nabla y_n|} \dH^1(x) - \frac{1}{2}\int_{\Cu} (\bar y - \bar t)^2 \nabla \bar \varphi \cdot \frac{\nabla \bar y}{|\nabla \bar y|} \dH^1(x)
    \end{align*}
    with $\Cu_n := \{y_n = \bar t\} \cap \Cu^{\epsilon}$ being a $C^1$ closed simple curve for $n$ large enough; see \cref{prop:levelset-unique-component-continuity}.
    Consequently, one has
    \begin{equation}
        \label{eq:Cn-tilde-esti}
        \tilde{C}_n = \frac{1}{2}\int_{\Cu^\epsilon} [\1_{ \tilde \Omega_n^{2} } -\1_{ \tilde \Omega_n^{3}}] (\bar y - \bar t)^2 \Delta \bar \varphi \dx +\frac{1}{2} \int_{\Cu_n }(\bar y - y_n)^2 \nabla \bar \varphi \cdot \frac{\nabla y_n}{|\nabla y_n|} \dH^1(x),
    \end{equation}
    which is similar to that in \eqref{eq:Cn-esti}. Moreover, analogous to \eqref{eq:Cn-first-term-esti}, the first term in the right-hand side of \eqref{eq:Cn-tilde-esti} is $o(s_n^2)$. This implies that
    \begin{equation}
        \label{eq:Cn-tilde-computation}
        \tilde{C}_n = o(s_n^2) + \frac{1}{2} \int_{\{y_n =\bar t \} \cap \Cu^\epsilon} (\bar y-y_n)^2 \nabla\bar \varphi \cdot \frac{\nabla y_n}{|\nabla y_n|} \dH^{1}(x).
    \end{equation}
    This expression is analogous to \eqref{eq:Cn-computation} and also valid to the case $\Cu := \partial\Omega$.
    From \eqref{eq:Cn-tilde-computation} and \eqref{eq:An-for-y-esti}, we derive \eqref{lim:key-for-y}
    by using \cref{prop:continuity-integral-levelsets} together with the limits $y_n \to \bar y$ in $C^1(\overline\Omega)$ and $(y_n-\bar y)/s_n \to w$ in $W^{1,p}_0(\Omega)$ and so in $C(\overline{\Omega})$.
\end{proof}

As a consequence of \cref{thm:key-formula} and \cref{prop:key-formular-2}, we have an explicit formula for the crucial term $\tilde{Q}$ in \eqref{eq:key-term-sn},
and an important limit that will play a significant role in establishing the error estimates for the numerical approximations of \eqref{eq:P}.
\begin{theorem}
    \label{thm:explicit-formular-Q2}
    Assume that the level set $\{\bar y=\bar t\}$ with $\bar y:= S(\bar u)$ decomposes into finitely many connected components. Let $v \in L^2(\Omega)$ and
    $\bar \varphi \in C^1(\overline\Omega) \cap W^{2,1}(\Omega)$.
    Then, for any $\{s_n\} \in c_0^+$ and $\{v_n\} \subset L^2(\Omega)$ such that $v_n \rightharpoonup v$ in $L^2(\Omega)$, $y_n:= S(\bar u + s_n v_n) \to \bar y$ in $C^1(\overline\Omega)$, and $(y_n - \bar y)/s_n \to w$ in $W^{1,p}_0(\Omega)$ for some $p>2$, the following assertions are valid:
    \begin{enumerate}[label=(\alph*)]
        \item \label{item:Q2-explicit}
            If, for any connected component $\Cu$ of $\{\bar y=\bar t\}$, $\bar y$ fulfills either
            \eqref{eq:main-hypothesis-y-novanish}, \eqref{eq:main-hypothesis-y-vanish}, or \eqref{eq:main-hypothesis-y-vanish-structure}, then
            \begin{equation}
                \label{eq:explicit-formular}
                \tilde Q(\bar u,\bar y,\bar \varphi;\{s_n\}, v) = \frac{1}{2} \{a'\}_{\bar t+0}^{\bar t-0}\int_{\{\bar y=\bar t\}}\1_{\{|\nabla \bar y| >0 \}} w^2 \frac{\nabla \bar y \cdot \nabla\bar \varphi}{|\nabla \bar y|} \dH^{1}(x).
            \end{equation}
        \item \label{item:Pn-auxi}
            If, for any connected component $\Cu$ of $\{\bar y=\bar t\}$, $\bar y$ fulfills
            either \eqref{eq:main-hypothesis-y-novanish} or \eqref{eq:main-hypothesis-y-vanish-structure}, then
            \begin{equation}
                \label{eq:Pn-auxi}
                \lim\limits_{n \to \infty} \frac{1}{s_n^2} \int_\Omega (2 \bar t - \bar y - y_n) [\1_{ \Omega_{y_n, \bar y }^{2} } -\1_{ \Omega_{y_n, \bar y }^{3}}] \nabla \bar y \cdot \nabla \bar \varphi \dx = 0.
            \end{equation}
    \end{enumerate}
\end{theorem}
\begin{proof}
    Assume that
    $\{\bar y=\bar t\} = \bigcup_{k=1}^{m} \Cu_{k}$ with $\Cu_k$ being connected components of $\{\bar y=\bar t\}$.
    Let $\epsilon_0 >0$ be such that $(\Cu_{k_1})^{\epsilon_0} \cap (\Cu_{k_2})^{\epsilon_0} =\emptyset$ for all $k_1 \neq k_2$. By \cref{thm:key-formula}, for any $1 \leq k \leq m$, there exists $\epsilon_k >0$ satisfying the claims of
    \cref{thm:key-formula} in place of $\epsilon_0$. We now set $\epsilon_* := \min\{\epsilon_0, \epsilon_k \mid 1 \leq k \leq m\} >0$ and fix $\epsilon \in (0,\epsilon_*)$ arbitrarily.
    We first prove \eqref{eq:explicit-formular}.
    In fact, all assumptions required in
    \cref{thm:key-formula} are fulfilled.
    From \eqref{eq:Omega-123-sets}, we have
    \begin{equation*}
        \Omega_{y_n,\bar y}^{2} \cup \Omega_{y_n,\bar y}^{3} \subset \{0 < |\bar y-\bar t| \leq \tau_n \} \subset \bigcup_{k=1}^{m} (\Cu_k)^\epsilon
    \end{equation*}
    for $n$ large enough, where $\tau_n := \norm{y_n - \bar y}_{C(\overline\Omega)}$. From the definition of $\tilde{Q} :=\tilde Q(\bar u,\bar y,\bar \varphi;\{s_n\}, v) $ in \eqref{eq:key-term-sn} and of $A_n(\Cu,\epsilon)$ in
    \cref{thm:key-formula}, we then deduce from \eqref{eq:sigma-tilde} that
    \begin{equation*}
        \tilde Q = \liminf\limits_{n \to \infty} \frac{1}{s_n^2} \{a'\}_{\bar t+0}^{\bar t-0} \sum_{ k=1}^{m} A_n(\Cu_{k}, \epsilon) = \{a'\}_{\bar t+0}^{\bar t-0} \sum_{ k=1}^{m} \frac{1}{2} \int_{\Cu_k}\1_{\{|\nabla \bar y| >0 \}} w^2\frac{\nabla \bar y \cdot \nabla \bar \varphi}{|\nabla \bar y|} \dH^{1}(x),
    \end{equation*}
    which, together with the disjoint decomposition $\{\bar y=\bar t\} = \bigcup_{k=1}^{m} \Cu_{k}$ yields the identity in \eqref{eq:explicit-formular}.
    Assertion \ref{item:Q2-explicit} is then proven.

    For assertion \ref{item:Pn-auxi}, by using the definitions of $A_n(\Cu,\epsilon)$ and of $\tilde A_n(\Cu,\epsilon)$, respectively, in
    \cref{thm:key-formula} and \cref{prop:key-formular-2}, there holds
    \[
        \int_\Omega (2\bar t - \bar y - y_n) [\1_{ \Omega_{y_n, \bar y }^{2} } -\1_{ \Omega_{y_n, \bar y }^{3}}] \nabla \bar y \cdot \nabla \bar \varphi \dx = \sum_{ k=1}^{m} [A_n(\Cu_{k}, \epsilon) + \tilde A_n(\Cu_{k}, \epsilon) ].
    \]
    We thus derive \eqref{eq:Pn-auxi} by
    \cref{thm:key-formula} and \cref{prop:key-formular-2}.
\end{proof}

As the last preparatory step, the following result shows the finiteness of the jump functional $\Sigma(\bar y)$ determined in \eqref{eq:E-functional} under one of three assumptions \ref{ass:nonvanishing-gradient-An}, \ref{ass:vanishing-gradient-An-convex-concave}, and \ref{ass:vanishing-gradient-structure}.
\begin{proposition}
    \label{prop:E-func-finite}
    Let $\bar u\in L^2(\Omega)$ be such that $\bar y:=S(\bar u)$ is Lipschitz continuous on $\overline\Omega$.
    Assume further that the level set $\{\bar y=\bar t\}$ decomposes into finitely many connected components and that, on each such connected component $\Cu$, either \ref{ass:nonvanishing-gradient-An}, \ref{ass:vanishing-gradient-An-convex-concave}, or \ref{ass:vanishing-gradient-structure} in \cref{thm:key-formula} is fulfilled.
    Then $\Sigma(\bar y) < \infty$.
\end{proposition}
\begin{proof}
    Assume that
    $\{\bar y=\bar t\} = \bigcup_{k=1}^{m} \Cu_{k}$ with $\Cu_k$ being connected components of $\{\bar y=\bar t\}$.
    Let $\epsilon_0 \in (0, \tilde{\epsilon})$ be such that $(\Cu_{k_1})^{\epsilon_0} \cap (\Cu_{k_2})^{\epsilon_0} =\emptyset$ for all $k_1 \neq k_2$, where $\tilde{\epsilon}$ is given in \cref{thm:key-formula}. For $r>0$ small enough, one has
    \begin{equation}
        \label{eq:Sigma-decom}
        \int_{\Omega} \1_{\{0 < |\bar y-\bar t| <r\}} |\nabla \bar y| dx = \sum_{k= 1}^m \int_{\Cu_k^{\epsilon_0}} \1_{\{0 < |\bar y-\bar t| <r\}} |\nabla \bar y| \dx.
    \end{equation}
    If either assumption \ref{ass:nonvanishing-gradient-An} or \ref{ass:vanishing-gradient-structure} is verified on $\Cu_k$ for some $1 \leq k \leq m$, then
    we deduce for any $r$ small enough that
    \begin{equation*}
        \int_{ \Cu_k^{\epsilon_0}} \1_{\{0 < |\bar y-\bar t| <r\}} |\nabla \bar y| \dx = C_k \int_{ \Cu_k^{\epsilon_0}} \1_{\{0 < |\bar y-\bar t| <r\}} \dx = C_k \meas_{\R^2}(\{0 < |\bar y-\bar t| <r\} \cap \Cu_k^{\epsilon_0}) dx \leq C_kr,
    \end{equation*}
    where we have used Lemma 3.2 in \cite{DeckelnichHinze2012} to obtain the last estimate.
    Here $C_k := \max\{|\nabla \bar y(x)| : x \in \overline{\Cu_k^{\epsilon_0}}\}$ for $1 \leq k \leq m$.
    If assumption \ref{ass:vanishing-gradient-An-convex-concave} holds on $\Cu_k$ for some $1 \leq k \leq m$, then by
    applying the coarea formula for Lipschitz mappings (see, e.g. \cite{Evans1992}, Thm.~2, p.~117 and \cite{AlbertiBianchiniCrippa2013}, Sec.~2.7), we have for any $r \in (0, r_0)$ that
    \begin{equation*}
        \int_{\Cu_k^{\epsilon_0}} \1_{\{0 < |\bar y-\bar t| <r\}} |\nabla \bar y| dx = \int_{\R} \int_{ \{\bar y=t \}} \1_{\{0 < |\bar y-\bar t| <r\} \cap \Cu_k^{\epsilon_0} }\dH^{1}(x) dt = \int_{\bar t-r}^{\bar t+r}\int_{ \{\bar y=t \}} \1_{\Cu_k^{\epsilon_0}} \dH^{1}(x) dt \leq 2rC_0.
    \end{equation*}
    From these estimates, \eqref{eq:Sigma-decom}, and the definition of $\Sigma(\bar y)$ in \eqref{eq:E-functional}, we conclude that $\Sigma(\bar y) < \infty$.
\end{proof}

The following theorem now formulates an explicit formula of the nonsmooth curvature functional $Q$ defined in \eqref{eq:curvature}.
\begin{theorem}
    \label{thm:explicit-curvature-func}
    Let \crefrange{ass:domain}{ass:cost_func} be fulfilled. Assume that $\bar u \in \mathcal{U}_{ad}$ and that $v \in \mathrm{cl}[\mathrm{cone}(\mathcal{U}_{ad} -\bar u)]$.
    Assume further that the level set $\{\bar y=\bar t\}$ with $\bar y:= S(\bar u)$ decomposes into finitely many connected components and that, on each such connected component $\Cu$, either \ref{ass:nonvanishing-gradient-An}, \ref{ass:vanishing-gradient-An-convex-concave}, or \ref{ass:vanishing-gradient-structure} in \cref{thm:key-formula} is fulfilled.
    Then, for any $\varphi \in C^1(\overline\Omega) \cap W^{2,1}(\Omega)$, there holds
    \begin{multline} \label{eq:explicit-curvature-func}
        Q(\bar u, \bar y, \varphi; v) = \frac{1}{2} \int_\Omega \frac{\partial^2 L}{\partial y^2}(\cdot, \bar y)z_{v}^2 \dx + \frac{\nu}{2} \int_\Omega v^2 \dx
        - \frac{1}{2} \int_\Omega\1_{ \{ y \neq \bar t \}} a''(\bar y)z_{v}^2 \nabla \bar y \cdot \nabla \varphi \dx \\
        - \int_\Omega a'(\bar y; z_{v}) \nabla z_{v} \cdot \nabla \varphi \dx + \frac{1}{2} [a'_{0}(\bar t) - a'_{1}(\bar t)] \int_{\{\bar y=\bar t\}}\1_{\{|\nabla \bar y| >0 \}} z_v^2 \frac{\nabla \bar y \cdot \nabla \varphi}{|\nabla \bar y|} \dH^{1}(x)
    \end{multline}
    with $z_v := S(\bar u)v$. Here $\mathrm{cl}[\mathrm{cone}(\mathcal{U}_{ad} -\bar u)]$ denotes the closure in $L^2(\Omega)$ of the cone generated by $(\mathcal{U}_{ad} - \bar u)$.
\end{theorem}
\begin{proof}
    We first observe that under the stated assumptions, $\Sigma(\bar y)$ is finite by \cref{prop:E-func-finite}.
    As a result of \eqref{eq:sigma-bounded}, $Q_2(\bar u, \bar y, \varphi; \cdot)$ and thus $Q(\bar u, \bar y, \varphi; \cdot)$ (defined, respectively, in \eqref{eq:key-term} and \eqref{eq:curvature}) are well-defined on $L^2(\Omega)$.
    Let $\{s_n\} \in c_0^+$ be arbitrary. Taking $v \in \mathrm{cl}[\mathrm{cone}(\mathcal{U}_{ad} -\bar u)]$ arbitrarily, then there are sequences $\{v_m\} \subset L^2(\Omega)$, $\{u_m\} \subset \mathcal{U}_{ad}$, and $\{\lambda_m \} \subset (0,\infty)$ such that
    \begin{equation*}
        v_m \to v \quad \text{in } L^2(\Omega) \quad \text{and} \quad v_m = \frac{u_m - \bar u}{\lambda_m} \quad \text{for all } m \in\N.
    \end{equation*}
    By the definition of $\tilde Q(\bar u,\bar y,\bar \varphi;\{s_n\}, v)$ in \eqref{eq:key-term-sn}, there holds
    \begin{equation}
        \label{eq:key-term-sn-subsequence}
        \tilde Q(\bar u,\bar y,\bar \varphi;\{s_n\}, v) = \lim\limits_{k \to \infty} \frac{1}{s_{n_k}^2} \int_\Omega \sum_{i =0 }^1 \zeta_i(\bar u,\bar y;s_{n_k},v) \nabla \bar y \cdot \nabla \bar \varphi \dx = \tilde Q(\bar u,\bar y,\bar \varphi;\{s_{n_k}\}, v)
    \end{equation}
    for some subsequence $\{s_{n_k}\}$ of $\{s_n\}$.
    Since $s_{n_k} \to 0^+$ as $k \to \infty$,
    for any $m \geq 1$, there exists an integer $k(m)$ such that $r_m := s_{n_{k(m)}} \in (0, \lambda_m)$ and thus
    \begin{equation*}
        \bar u + r_m v_m = \left(1 - \frac{r_m}{\lambda_m} \right)\bar u + \frac{r_m}{\lambda_m} u_m \in \mathcal{U}_{ad}.
    \end{equation*}
    Setting $y_m := S(\bar u + r_m v_m)$ yields $(y_m - \bar y)/r_m \to S'(\bar u)v$ in $W^{1,p}_0(\Omega)$ and $y_m \in W^{2,p}(\Omega)$ for some $p >2$, according to \cref{thm:control2state-oper}. From this and the compact embedding $W^{2,p}(\Omega) \Subset C^1(\overline\Omega)$, one has $y_m \to \bar y$ in $C^1(\overline\Omega)$. Using now the limit $v_m \to v$ in $L^2(\Omega)$, \eqref{eq:key-term-sn-subsequence}, as well as \eqref{eq:sigma-tilde}, and then applying \cref{thm:key-formula}, one has
    \begin{align*}
        \tilde Q(\bar u,\bar y,\bar \varphi;\{s_n\}, v) & = \tilde Q(\bar u,\bar y,\bar \varphi;\{r_m\}, v) \\
        & = \lim\limits_{m \to \infty} \frac{1}{r_m^2} \int_\Omega \sum_{i =0 }^1 \zeta_i(\bar u,\bar y;r_m,v_m) \nabla \bar y \cdot \nabla \bar \varphi \dx \\
        & = \frac{1}{2} \{a'\}_{\bar t+0}^{\bar t-0}\int_{\{\bar y=\bar t\}}\1_{\{|\nabla \bar y| >0 \}} (S'(\bar u)v)^2 \frac{\nabla \bar y \cdot \nabla \varphi}{|\nabla \bar y|} \dH^{1}(x).
    \end{align*}
    Since $\{s_n\}$ is taken arbitrarily, we then deduce from the definition of $Q_2$ that
    \[
        Q_2(\bar u, \bar y,\bar \varphi;v) = \tilde Q(\bar u,\bar y,\bar \varphi;\{s_n\}, v) = \frac{1}{2} \{a'\}_{\bar t+0}^{\bar t-0}\int_{\{\bar y=\bar t\}}\1_{\{|\nabla \bar y| >0 \}} (S'(\bar u)v)^2 \frac{\nabla \bar y \cdot \nabla \varphi}{|\nabla \bar y|} \dH^{1}(x).
    \]
    Combining this with the definition of $Q$ in \eqref{eq:curvature} yields \eqref{eq:explicit-curvature-func}.
\end{proof}

\subsection{Explicit second-order optimality conditions}
\label{sec:explicit-2nd-OCs}
The following explicit second-order necessary and sufficient optimality conditions for \eqref{eq:P} are consequences of \cref{thm:2nd-OS-nec,thm:2nd-OS-suf,thm:explicit-formular-Q2}.

\begin{theorem}[explicit second-order necessary optimality conditions]
    \label{thm:2nd-OS-nec-explicit}
    Let \crefrange{ass:domain}{ass:cost_func} be fulfilled. Assume that $\bar u$ is a local minimizer to \eqref{eq:P2}.
    Assume further that the level set $\{\bar y=\bar t\}$ with $\bar y:= S(\bar u)$ decomposes into finitely many connected components and that, on each such connected component $\Cu$, either \ref{ass:nonvanishing-gradient-An}, \ref{ass:vanishing-gradient-An-convex-concave}, or \ref{ass:vanishing-gradient-structure} in \cref{thm:key-formula} is fulfilled.
    Then, there is a $\bar\varphi \in W^{1,\bar p}_0(\Omega) \cap W^{1,\infty}(\Omega)$, with $\bar p$ defined in \cref{ass:cost_func}, that together with $\bar u, \bar y$ satisfies \eqref{eq:1st-OS} and
    \begin{multline} \label{eq:2nd-OC-nec-explicit}
        Q(\bar u, \bar y, \bar \varphi; v) = \frac{1}{2} \int_\Omega \frac{\partial^2 L}{\partial y^2}(\cdot, \bar y)z_{v}^2 \dx + \frac{\nu}{2} \int_\Omega v^2 \dx
        - \frac{1}{2} \int_\Omega\1_{ \{ y \neq \bar t \}} a''(\bar y)z_{v}^2 \nabla \bar y \cdot \nabla \bar \varphi \dx \\
        - \int_\Omega a'(\bar y; z_{v}) \nabla z_{v} \cdot \nabla \bar \varphi \dx + \frac{1}{2} [a'_{0}(\bar t) - a'_{1}(\bar t)] \int_{\{\bar y=\bar t\}}\1_{\{|\nabla \bar y| >0 \}} z_v^2 \frac{\nabla \bar y \cdot \nabla\bar \varphi}{|\nabla \bar y|} \dH^{1}(x) \geq 0
    \end{multline}
    for all $v\in \mathcal{C}({\mathcal{U}_{ad};\bar u})$ with $z_v := S(\bar u)v$.
\end{theorem}
\begin{proof}
    By \cref{prop:E-func-finite}, $\Sigma(\bar y)$ is finite.
    In view of \cref{thm:2nd-OS-nec}, a function $\bar\varphi$ exists and satisfies \eqref{eq:1st-OS}. By \cref{thm:1st-OC}, $\bar \varphi \in W^{2,r}(\Omega) \hookrightarrow C^1(\overline\Omega)$ for some $r>2=N$.
    Moreover, there holds $\mathcal{C}({\mathcal{U}_{ad};\bar u}) \subset \mathrm{cl}[\mathrm{cone}(\mathcal{U}_{ad} -\bar u)]$.
    Applying \cref{thm:2nd-OS-nec,thm:explicit-curvature-func} yields \eqref{eq:2nd-OC-nec-explicit}.
\end{proof}

\begin{theorem}[explicit second-order sufficient optimality conditions]
    \label{thm:2nd-OS-suf-explicit}
    Let \crefrange{ass:domain}{ass:cost_func} be valid. Assume that $\bar u$ is a feasible point of \eqref{eq:P2}.
    Assume that the level set $\{\bar y=\bar t\}$ with $\bar y:= S(\bar u)$ decomposes into finitely many connected components and that, on each such connected component $\Cu$, either \ref{ass:nonvanishing-gradient-An}, \ref{ass:vanishing-gradient-An-convex-concave}, or \ref{ass:vanishing-gradient-structure} in \cref{thm:key-formula} is fulfilled.
    Assume further that there is a $\bar\varphi \in W^{1,\bar p}_0(\Omega) \cap W^{1,\infty}(\Omega)$, with $\bar p$ defined in \cref{ass:cost_func}, that together with $\bar u, \bar y$ satisfies \eqref{eq:1st-OS}
    and
    \begin{multline*} %
        Q(\bar u, \bar y, \bar \varphi; v) = \frac{1}{2} \int_\Omega \frac{\partial^2 L}{\partial y^2}(\cdot, \bar y)z_{v}^2 \dx + \frac{\nu}{2} \int_\Omega v^2 \dx
        - \frac{1}{2} \int_\Omega\1_{ \{ y \neq \bar t \}} a''(\bar y)z_{v}^2 \nabla \bar y \cdot \nabla \bar \varphi \dx \\
        - \int_\Omega a'(\bar y; z_{v}) \nabla z_{v} \cdot \nabla \bar \varphi \dx + \frac{1}{2} [a'_{0}(\bar t) - a'_{1}(\bar t)] \int_{\{\bar y=\bar t\}}\1_{\{|\nabla \bar y| >0 \}} z_v^2 \frac{\nabla \bar y \cdot \nabla\bar \varphi}{|\nabla \bar y|} \dH^{1}(x) > 0
    \end{multline*}
    with $z_v := S'(\bar u)v$
    for all $v\in \mathcal{C}({\mathcal{U}_{ad};\bar u}) \setminus \{0\}$.
    Then there exist constants $c_0, \rho_0 >0$ satisfying
    \begin{equation*}
        j(\bar u) + c_0 \norm{u - \bar u}_{L^2(\Omega)}^2 \leq j(u) %
        \qquad\text{for all } u \in \mathcal{U}_{ad} \cap \overline B_{L^2(\Omega)}(\bar u, \rho_0).
    \end{equation*}
\end{theorem}
\begin{proof}
    Similar to the proof of \cref{thm:2nd-OS-nec-explicit}, the desired conclusion follows from \cref{thm:2nd-OS-suf,thm:explicit-curvature-func}.
\end{proof}

\section{Conclusions}
We have derived necessary and sufficient second-order optimality conditions for a nonsmooth quasilinear elliptic optimal control problems in terms of an explicit curvature functional. The main technical tools are a detailed study of the properties of level sets of $C^1$ functions with nonvanishing gradients, including continuity properties and a Green's identity. These results may be of independent interest in, e.g., the analysis of level set methods or in mathematical imaging.
In the second part of this work, the sufficient second-order conditions will be applied to derive error estimates for a finite element discretization of the nonsmooth optimal control problem.

\appendix
\section{Existence of regular points} \label{sec:app}
\begin{lemma} \label{lem:regular-points}
    Let $y_1, y_2 \in C^1(\overline\Omega)$ and $t \in \R$ be such that $\{y_1 > t >y_2\} \neq \emptyset$. Let $\mathcal{S}$ be a nonempty open connected component of $\{y_1 > t >y_2\}$ that satisfies $\partial \mathcal{S} \cap \partial \Omega = \emptyset$ and \eqref{eq:nonvanishing-gradient} in \cref{lem:Green-identity}. Assume that the boundary of $\mathcal{S}$ has a cusp at $N$; see, \cref{fig:component-cusp}. Then there is a constant $\epsilon_0>0$ such that, for each $0 < \epsilon < \epsilon_0$, points $N_i \in \partial \mathcal{S} \cap \{y_i = t \}$, $i=1,2$, exist that satisfy
    \begin{equation} \label{eq:length-curve-N12}
        l_{\wideparen{NN_i}} < C\epsilon
    \end{equation}
    and
    \begin{equation} \label{eq:length-line-N12}
        N_1N_2 < C \epsilon
    \end{equation}
    for some constant $C>0$, where $l_{\wideparen{NN_i}}$ denotes the length of the arc $\wideparen{NN_i}$.
    Furthermore, $N_1$ and $N_2$ are not cusps of $\partial\mathcal{S}_\epsilon$, where $\mathcal{S}_\epsilon$ is the subdomain of $\mathcal{S}$ obtained by splitting $\mathcal{S}$ via the line $N_1N_2$, and $N \notin \mathcal{S}_\epsilon$.
\end{lemma}
\begin{proof}
    Without loss of generality, assume that $N = (0,0)\in \R^2$ and that the positive $x^2$-axis ray is the common tangent line of two arcs of $\partial \mathcal{S}$ at $N$. Then the tangent vector of $\mathcal{S}$ at $N$ is $\nu_N = (-1,0)$.
    Since $\nu_N = \kappa \nabla y_1(N) = \kappa \nabla y_2(N)$ for some $\kappa \in \R$, we have $\frac{\partial y_1}{\partial x^2}(N) =\frac{\partial y_2}{\partial x^2}(N) = 0$ and
    \begin{equation}
        \label{eq:positive-deri-x2}
        \frac{\partial y_1}{\partial x^1}(x) \cdot \frac{\partial y_2}{\partial x^1}(x) \geq \alpha' >0 \quad \text{for all } x \in \overline{ \mathcal{S}} \cap B_{\R^2}(0,r)
    \end{equation}
    for some constants $r >0$ and $\alpha'$, due to \eqref{eq:nonvanishing-gradient}.
    Applying the Implicit Function Theorem at the point $N$ yields that there exist constants $h, k >0$, $\rho \in (0,r)$ and $C^1$-functions $\kappa_i: (-h,k) \to \R$, $i=1,2$, such that
    \begin{equation}
        \label{eq:gi-func-app}
        \left\{
            \begin{aligned}
                &\quad \kappa_i(0) = 0,\\
                & y_i( \kappa_i(\xi),\xi) = t \quad \text{for all } \xi \in (-h, k), \\
                & \left((\tau,\xi) \in B_{\R^2}(0, \rho) \, \text{and } y_i(\kappa_i(\xi), \xi) = t \right) \quad \implies \quad \tau = \kappa_i(\xi).
            \end{aligned}
        \right.
    \end{equation}
    Moreover, similar to \eqref{eq:implicit-func-deri-bound}, one has
    \begin{equation}
        \label{eq:kappa-deri-bound}
        |\kappa_i'(\xi)| \leq c \quad \text{for all } \xi \in (-h,k)
    \end{equation}
    and for some constant $c>0$.
    Since $\mathcal{S} \cap B_{\R^2}(N, r)$ is in the half space $\{x^2 \geq 0\} := \{(x^1,x^2) \in \R^2 \mid x^2 \geq 0 \}$, we therefore have
    \begin{align}
        \{y_i =t\} \cap \partial\mathcal{S} \cap B_{\R^2}(0, \rho_0) & = \{(\tau,\xi) \in B_{\R^2}(0, \rho_0) \mid \xi \geq 0, y_i(\tau, \xi) = t \} \notag
        \\
        &= \{(\tau,\xi) \in B_{\R^2}(0, \rho_0) \mid \xi \in [0, k_i), \tau = \kappa_i(\xi) \} \label{eq:arc-para-gi}
    \end{align}
    for some constants $\rho_0 \in (0, \rho]$ and $k_i \in (0, k]$, $i=1,2$.
    Moreover, the equation $\kappa_1(\xi) = \kappa_2(\xi)$ has a unique solution $\xi =0$ in the interval $(-h,k)$ because the arcs $\{y_1 =t\} \cap B_{\R^2}(0, \rho)$ and $\{y_2 =t\} \cap B_{\R^2}(0, \rho)$ has only one common point $N$.
    Consequently, one and only one of the following cases is valid:
    \begin{equation}
        \label{eq:g1-less-g2}
        \kappa_1(\xi) > \kappa_2(\xi) \quad \text{for all } \xi \in (0, k) \quad \text{or} \quad \kappa_1(\xi) < \kappa_2(\xi) \quad \text{for all } \xi \in (0, k).
    \end{equation}
    We now consider the first situation in \eqref{eq:g1-less-g2} only since the other is analyzed analogously.
    Thanks to \eqref{eq:positive-deri-x2}, one has two possibilities: For small number $x^2$, in the neighborhood of the origin in $\R$,
    \begin{enumerate}
        \item[(I)] both $y_1(\cdot, x^2)$ and $y_2(\cdot, x^2)$ are strictly increasing; or
        \item[(II)] both $y_1(\cdot, x^2)$ and $y_2(\cdot, x^2)$ are strictly decreasing.
    \end{enumerate}
    For the case (I), one has for some constant $0 < \rho_1 \leq \rho_0$ that
    \[
        \{ y_1 > t\} \cap B_{\R^2}(0, \rho_1) \cap \{\xi \geq 0\} \subset \{(\tau, \xi) \in \R^2 \mid \xi \in [0, k_1), \tau > \kappa_1(\xi) \}
    \]
    and
    \[
        \{ y_2 < t\}\cap B_{\R^2}(0, \rho_1) \cap \{\xi \geq 0\} \subset \{(\tau, \xi) \in \R^2 \mid \xi \in [0, k_2), \tau < \kappa_2(\xi) \}.
    \]
    Combining these with the first situation in \eqref{eq:g1-less-g2} yields a contradiction. Therefore, the case (II) must be true. Thanks to (II) and \eqref{eq:g1-less-g2}, there holds
    \begin{equation}
        \label{eq:S-inclusion-locally}
        \mathcal{S} \cap B_{\R^2}(0, \underline \rho) \subset \{(\tau, \xi) \in \R^2 \mid \xi \in [0, k_*), \kappa_2(\xi) < \tau < \kappa_1(\xi) \} \subset \mathcal{S} \cap B_{\R^2}(0, \overline \rho)
    \end{equation}
    for some positive constants $0 <\underline{\rho} < \overline{\rho} \leq \rho_0$, and $0 <k_* \leq \min\{k_1, k_2 \}$. Setting $\epsilon_0 := \underline{\rho}$ and taking $\epsilon \in (0, \epsilon_0)$ arbitrarily, we now fix $\xi_1 \in (0, \epsilon)$ and set
    $
    N_1 := (\kappa_1(\xi_1), \xi_1)
    $
    and
    $
    N_2 := (\kappa_2(\xi_1), \xi_1).
    $
    Obviously, there holds $N_i \in \{y_i =t \} \cap \partial \mathcal{S}$, $i=1,2$.
    Moreover, using the inequalities \eqref{eq:g1-less-g2} and \eqref{eq:kappa-deri-bound} yields
    \begin{align*}
        N_1N_2 & = |\kappa_2(\xi_1) - \kappa_1(\xi_1)| = \kappa_1(\xi_1) - \kappa_2(\xi_1) \\
        & \leq |\kappa_1(\xi_1) - \kappa_1(0)| + | \kappa_2(\xi_1) - \kappa_2(0)|\\
        & \leq 2c\xi_1.
    \end{align*}
    Thus, \eqref{eq:length-line-N12} is verified. On the other hand, since $\{(\kappa_i(\xi),\xi): 0 \leq \xi \leq \xi_1 \} = \wideparen{NN_i}$, $i=1,2$, we conclude from \eqref{eq:kappa-deri-bound} that
    \begin{align*}
        l_{\wideparen{NN_i}} & = \int_0^{\xi_1} \sqrt{1 + (\kappa_i'(\xi))^2} d\xi \leq \sqrt{1+c^2} \xi_1 \leq \epsilon\sqrt{1+c^2},
    \end{align*}
    which gives \eqref{eq:length-curve-N12}.

    It remains to prove that $N_1$ and $N_2$ are not cusps of $\partial\mathcal{S}_\epsilon$.
    To do this, we first consider the point $N_1(\kappa_1(\xi_1),\xi_1)$. Let $\omega_1 \in [0,\pi/2]$ be the angle between the line $N_1N_2$ and the tangent vector to $\{y_1=t\}$ at $N_1$. We then have
    \begin{align*}
        \cos \omega_1 & = \left| \frac{\nabla y_1}{|\nabla y_1|}(N_1) \cdot (0, 1) \right| = \left| \frac{\frac{\partial y_1}{\partial x^2} }{\sqrt{(\frac{\partial y_1}{\partial x^1})^2 + (\frac{\partial y_1}{\partial x^2})^2 }}(N_1) \right| < 1,
    \end{align*}
    according to \eqref{eq:positive-deri-x2}. Hence $\omega_1 >0$ and thus $N_1$ is not a cusp of $\partial\mathcal{S}_\epsilon$. Similarly, $N_2$ is also not a cusp of $\partial\mathcal{S}_\epsilon$.
\end{proof}

\begin{lemma}
    \label{lem:x1x2-existence}
    Let $y \in C^1(\overline\Omega)$ and $ t \in \R$ be such that $\{y=t \} \neq \emptyset$ and $y = 0$ on $\partial\Omega$.
    Let $\Cu$ be a connected component of $\{y=t\}$ that satisfies \eqref{eq:nonvanishing-gradient-on-component}. Then, for any $\epsilon >0$ and $x_2 \in \Cu^\epsilon$, there holds
    \[
        \mathrm{Arg}(x_2) :=\mathrm{argmin} \{ |x_2 -x| : x \in \Cu \} \neq \emptyset.
    \]
    Moreover, if $x_1 \in \mathrm{Arg}(x_2)$, then
    \begin{equation}
        \label{eq:x1x2-distance-app}
        |x_1 - x_2| \leq \epsilon
    \end{equation}
    and
    \begin{equation}
        \label{eq:x1x2-direction-app}
        x_2 - x_1 = \kappa \nabla y(x_1)
    \end{equation}
    for some constant $\kappa \in \R$.
\end{lemma}
\begin{proof}
    By \cref{prop:decomposition-levelset}, $\Cu$ is a closed $C^1$ simple curve. Moreover, for any $x_2 = (x_2^1, x_2^2) \in \Cu^\epsilon$, we obviously have $\mathrm{Arg}(x_2) \neq \emptyset$ since $\Cu$ is compact.
    Taking $x_1 \in \mathrm{Arg}(x_2)$ arbitrarily, the definition of the set $\Cu^\epsilon$ implies \eqref{eq:x1x2-distance-app}.
    It remains to prove \eqref{eq:x1x2-direction-app}. To do this,
    we first extend $y$ by a $C^1$ extension, denoted by the same name, on $\R^2$ in the case where $\Cu = \partial\Omega$, as a result of \cref{prop:C1-extension-on-convex-domain}. Due to \eqref{eq:nonvanishing-gradient-on-component}, without lost of generality, we now assume that
    \[
        \left|\frac{\partial y}{\partial x^2}(x) \right| \geq \alpha >0 \quad \text{for all } x \in B_{\R^2}(x_1,r)
    \]
    for some constant $r>0$.
    Applying now the Implicit Function Theorem in the point $x_1 = (x_1^1,x_1^2) \in \Cu$ yields that there exist constants $h, k >0$, $\rho \in (0,r)$, and a $C^1$-function $g: (x_1^1-h,x_1^1+k) \to \R$ satisfying
    \begin{equation}
        \label{eq:g-func-x1}
        \left\{
            \begin{aligned}
                & \quad g(x^1_1) = x^2_1,\\
                & y(\tau, g(\tau)) = t \quad \text{for all } \tau \in (x^1_1- h, x^1_1+k), \\
                & \left((\tau,\xi) \in B_{\R^2}(x_1, \rho) \, \text{and } y(\tau, \xi) = t \right) \quad \implies \quad \xi = g(\tau).
            \end{aligned}
        \right.
    \end{equation}
    Moreover, similar to \eqref{eq:implicit-func-deri}, there holds
    \begin{equation}
        \label{eq:g-x1-der}
        g'(\tau) = - \frac{\frac{\partial y}{\partial{x^1}}(\tau, g(\tau))}{\frac{\partial y}{\partial{x^2}}(\tau, g(\tau))}, \quad \tau \in (x^1_1- h, x^1_1+k).
    \end{equation}
    For any $s \in (-h, k)$, one has $x_s := (x_1^1 +s, g(x_1^1+s)) \in \Cu$ and the function
    \[
        f(s) := |x_2 -x_s|^2
    \]
    attains its minimum at $s=0$, according to the definition of $x_1$. There thus holds $f'(0) =0$.
    A simply computation gives
    \begin{align*}
        f(s) & = (x_2^1-x_1^1-s)^2 + (x_2^2 - g(x_1^1 +s))^2
    \end{align*}
    and
    \[
        f'(s) = -2(x_2^1-x_1^1 -s) -2(x_2^2 - g(x_1^1 +s))g'(x_1^1+s).
    \]
    This implies that
    \begin{align*}
        0 = f'(0) & = -2(x_2^1-x_1^1) -2(x_2^2 - g(x_1^1))g'(x_1^1) \\
        & = -2[(x_2^1-x_1^1) + (x_2^2 - g(x_1^1))g'(x_1^1)] \\
        & = -2[(x_2^1-x_1^1) - (x_2^2 - x_1^2)\frac{\frac{\partial y}{\partial{x^1}}(x_1)}{\frac{\partial y}{\partial{x^2}}(x_1)}],
    \end{align*}
    where we have employed the first equation in \eqref{eq:g-func-x1} and the identity \eqref{eq:g-x1-der} to derive the last equation.
    We then obtain
    \[
        (x_2^1-x_1^1) = (x_2^2 - x_1^2)\frac{\frac{\partial y}{\partial{x^1}}(x_1)}{\frac{\partial y}{\partial{x^2}}(x_1)},
    \]
    which yields
    \[
        x_2 -x_1 = \frac{x_2^2-x_1^2}{\frac{\partial y}{\partial{x^2}}(x_1)} \nabla y(x_1).
    \]
    Setting $\kappa := \frac{x_2^2-x_1^2}{\frac{\partial y}{\partial{x^2}}(x_1)}$ finally yields \eqref{eq:x1x2-direction-app}.
\end{proof}

\section*{Acknowledgments}
This work was supported by the DFG under the grants CL 487/2-1 and RO 2462/6-1, both within the priority programme SPP 1962 ``Non-smooth and Complementarity-based Distributed Parameter Systems: Simulation and Hierarchical Optimization''.
The second author was funded by Vietnam National Foundation for Science and Technology Development (NAFOSTED) under grant number 101.01--2019.308.
The authors would like to thank the anonymous referees for their useful suggestions and comments which improved the manuscript greatly.

\bibliographystyle{jnsao}
\bibliography{QuasilinearFEM_I}

\end{document}